\documentclass[reqno]{amsart}
\usepackage{mathrsfs}
\usepackage{amsmath}
\usepackage{amsthm}
\usepackage{amsfonts}
\usepackage{amssymb}
\usepackage{url}
\usepackage{enumerate}
\usepackage[pdftex,bookmarks=true]{hyperref}
  \usepackage[usenames, dvipsnames]{color}
\usepackage{verbatim}
\usepackage{graphicx}

\usepackage{pdfsync}

\newcommand{\V}{\mathcal V}

\renewcommand\Re{{\operatorname{Re}}}
\renewcommand\Im{{\operatorname{Im}}}

\newcommand\R{{\mathbb{R}}}
\newcommand\C{{\mathbb{C}}}

\renewcommand\P{{\mathbb{P}}}
\newcommand\E{{\mathbb{E}}}

\newcommand\T{{\mathbf{T}}}

\newcommand\Var{{\operatorname{Var}}}
\newcommand\Cov{{\operatorname{Cov}}}

\newcommand\Z{{\mathbb{Z}}}

\newcommand\var{{ \operatorname{Var}}}

\newcommand\ep{\varepsilon}
\newcommand\al{\alpha}
\newcommand\la{\lambda}

\newcommand\Bb{{\mathbf b}}

\newcommand\Be{{\mathbf e}}

\newcommand\Bu{{\mathbf u}}
\newcommand\Bv{{\mathbf v}}
\newcommand\Bw{{\mathbf w}}

\newcommand\BD{{\mathbf D}}

\newcommand\BN{{\mathbf N}}

\newcommand\CE{{\mathcal E}}

\newcommand\CG{{\mathcal G}}

\newcommand\eps{\varepsilon}

\newcommand\lang{\langle}
\newcommand\rang{\rangle}

\newcommand\cov{{\operatorname{Cov}}}

\theoremstyle{plain}
  \newtheorem{theorem}[subsection]{Theorem}

  \newtheorem{prop}[subsection]{Proposition}
    
  \newtheorem{fact}[subsection]{Fact}
  \newtheorem{lemma}[subsection]{Lemma}
  \newtheorem{corollary}[subsection]{Corollary}
  \newtheorem{cor}[subsection]{Corollary}

  \newtheorem{condition}{Condition}
  \newtheorem{remark}[subsection]{Remark}
  
  \newtheorem{claim}[subsection]{Claim}

\theoremstyle{definition}

\begin{document}

\title[Random trigonometric polynomials]{Random trigonometric polynomials: universality and non-universality of the variance for the number of real roots}

\author{Yen Do}
\address{Department of Mathematics\\ The University of Virginia\\ 141 Cabell Drive, Charlottesville, VA 22904 USA}
\email{yendo@virginia.edu}  
\thanks{The first author is supported in part by NSF grant DMS-1800855.}

\author{Hoi H. Nguyen}
\address{Department of Mathematics\\ The Ohio State University \\ 231 W 18th Ave \\ Columbus, OH 43210 USA}
\email{nguyen.1261@math.osu.edu}
\thanks{The second author is supported by NSF grant DMS-1752345.}

\author{Oanh Nguyen}
\address{Department of Mathematics\\ Princeton University\\ 304 Washington Road\\ Princeton, NJ 08540 USA}  
\email{onguyen@math.princeton.edu}

\thanks{}

\maketitle

\begin{abstract}  In this paper, we study the number of real roots of random trigonometric polynomials with iid coefficients. When the coefficients have zero mean, unit variance and some finite high moments, we show that the variance of the number of real roots  is asymptotically linear in terms of the expectation; furthermore the multiplicative constant in this linear relationship depends only on the kurtosis of the common distribution of the polynomial's coefficients. This result is in sharp contrast to the classical Kac polynomials whose corresponding variance depends only on the first two moments. Our result is perhaps the first paper to establish the variance for general distribution of the coefficients including discrete ones, for a model of random polynomials outside the family of the Kac polynomials. Our method gives a fine comparison framework throughout Edgeworth expansion, asymptotic Kac-Rice formula and a detailed analysis of characteristic functions. 
 	\end{abstract}

\noindent \textbf{Keywords.} Edgeworth expansion, random polynomials, real roots, universality

\section{Introduction}

Universality for the distribution of roots of random polynomials  is an exciting subject that has attracted the attention of many generations.  When the degree of a polynomial is very large, it is often challenging, even numerically, to solve for the roots, and a very natural question is to obtain an accurate estimate for the number of roots in a given region (in particular in $\R$). There is a large body of studies in the past centuries dedicated to this task, showing that  the typical size of the number of  roots depends mostly on the underlying symmetries of the random polynomials and not on the particular distributions of the coefficients.  These studies often assume a fairly minimal normalization condition, where the coefficients are independent with  fixed means and variances. Results of this type are known in the literature as universality results for the number of (real) roots. 

Among many statistics about  the number of real roots of random polynomials, denoted by $N_n$ (or $N_{n,\R}$), the following are often considered first by many authors: the expectation $\E N_n$, the variance $\Var(N_n)$, and the limiting distribution of the standardization $N_n^*=\frac{N_n-\E N_n}{\sqrt {\Var(N_n)}}$. One of the most studied random polynomials in the literature is perhaps the \textit{Kac polynomial}, 
$$P_n(x)=\xi_0+\xi_1 x+\dots + \xi_n x^n$$ 
where $\xi_j$ are  iid copies of a common random variable $\xi$, often assumed to have zero mean and unit variance. The issue of estimating  $N_n$ for such polynomials  was already raised by Waring as far back as 1782 (\cite{To}, \cite{Kostlan}).  In the early 1940s, Kac \cite{kac-0} (see also \cite{Rice}) developed a magnificent formula for the expectation of number of real roots
 \begin{equation} \label{Kac}     \E N_n = \int_{-\infty} ^ {\infty} \int_{- \infty} ^{\infty} |y| p(t,0,y) dy dt,
 \end{equation}  
 where $p(t,x,y)$ is the probability density for $P_n (t) =x$ and $P'_n (t) =y$. See for instance \cite{AT,AW,EK} for other variants of this Kac-Rice formula. When $\xi$ is standard Gaussian, one can easily evaluate the right-hand side of \eqref{Kac} and obtain 
\begin{equation}\label{eq:Kac:gau}
\E (N_n, G)  =   \left (\frac{2}{\pi} +o(1)\right ) \log n
\end{equation}
where the notation $\E (N_n, G)$ indicates the expectation of $N_n$ when $\xi$ is standard Gaussian.
 Similarly, one can also show that 
 $$\Var (N_n, G) = \left (\frac{4}{\pi}(1- \frac{2}{\pi})+o(1)\right ) \log n.$$ 
 
Evaluating the double integral in the Kac-Rice formula \eqref{Kac} is feasible only when the function $p(t, x, y)$ is sufficiently nice which often requires that the random variable $\xi$ is continuous. It is thus of great interest to understand what happens when $\xi$ is discrete. A crucial example is when $\xi$ is Rademacher, that is $\xi$ takes values $\pm 1$ with equal probability. Even though the Rademacher distribution is arguably the simplest looking discrete distribution that one can think of, it is often the case in the study of random polynomials that a method applicable to Rademacher distribution can be adapted to much more general distributions. 

For the Kac polynomials with Rademacher coefficients, the seminal results of Littlewood and Offord  \cite{lo,  lo-3, lo-4, lo-2} and  Erd\H{o}s and Offord \cite{EO} showed that $\E N_n$ is universal in the sense that the Rademacher case behaves asymptotically like the Gaussian case \eqref{eq:Kac:gau}. In particular, 
\begin{equation}\label{eq:Kac:ra}
\E (N_n, Ra)  =   \left (\frac{2}{\pi} +o(1)\right ) \log n
\end{equation}
where the left-hand side indicates the expectation of $N_n$ with Rademacher coefficients. Ibragimov and Maslova \cite{IM1, IM2, IM3, IM4} (among others)  generalized the method by Erd\H{o}s and Offord to show that $\E N_n$ is universal as long as the random variable $\xi$ has mean 0, variance 1, and belongs to the domain of attraction of the normal distribution.
 
Beyond the Kac polynomials, proving universality for the roots of other classical random ensembles including 
\begin{itemize}
	\item elliptic polynomials, 
	\vskip .05in
	\item hyperbolic polynomials (which include the Kac polynomials),
		\vskip .05in
	\item trigonometric polynomials, 
		\vskip .05in
       \item and Weyl polynomials, 
\end{itemize}
 has become an active direction of research in recent years \cite{BleherD, BD2, IKM, KZ2, TV, DONgV, FK, FK2}. There is also a distinction between local and global universality. The \textit{global universality} concerns the limiting distribution of the empirical measure of all complex roots  and has been established in several papers for many random polynomials, see for instance \cite{KZ2, PR, BlD} and the references therein. The \textit{local universality} concerns the distribution of the roots (complex, real, or both) in smaller/thinner sets  and is developed in a series of work by Tao, Vu, and the current authors \cite{TV, NgNgV, DONgV, ONgV}.

Thanks to these results, the universality of $\E N_n$ has been systematically established for all of the aforementioned classical models of random polynomials, as done in \cite{ONgV}. On the other hand, understanding the universality of $\Var (N_n)$ remains greatly challenging. It is known that for the Kac polynomials and their generalization, this variance is universal \cite{Mas1, ONgVclt}. For other models of random polynomials, to the best of our knowledge, this variance is only known for Gaussian distribution or for some cases, distributions with certain continuous-ness. Our result would be the first to establish the variance for discrete distributions, including the Rademacher one.

We study the random trigonometric polynomial
\begin{equation}\label{eqn:Pn}
P_n(t,Y) = \frac{1}{\sqrt{n}}  \sum_{i=1}^n y_{i1} \cos\left(\frac{it}{n}\right) + y_{i2} \sin\left(\frac{it}{n}\right),
\end{equation}
 where $y_{ij}$ are independent random variables. Let $Y_i=(y_{i1},y_{i2})$ and $Y=(Y_1,\dots, Y_n)$.
 In this note, we are interested in the real roots of the periodic polynomial $P_n(t,Y)$ (although some other statistics of random trigonometric polynomials also play crucial role in several recent interesting studies, such as \cite{AJKS} and \cite{Rivin}). 
We now redefine  $N_n=N_n(Y)$ to be  the number of roots of $P_n$ in one period, namely for $t\in [-n\pi, n\pi]$. It is known from a result of Qualls \cite{Q} in the 1970s that when the $y_{ij}$ are iid standard Gaussian, we have
$$\E N_n = 2 \sqrt{(2n+1)(n+1)/6}.$$

Confirming a striking heuristic by Bogomolny, Bohigas, and Leboeuf \cite{BBL}, about ten years ago, Granville and Wigman \cite{GW} proved the following.
 \begin{theorem}\label{thm:GW} When the $y_{ij}$ are iid standard Gaussian, there exists an explicit positive constant $c_{G}$ such that the variance satisfies
$$\Var(N_n) =(c_{G}+o(1)) n.$$
Furthermore,
$$\frac{N_n -\E N_n}{\sqrt{c_{G}n}} \xrightarrow{d}  \BN(0,1).$$
\end{theorem}
Here, asymptotically, $c_G \approx 0.55826$. More precisely,
$$c_G = \frac{4}{3\pi} \int_{0}^\infty \left  ( \frac{1-g(t)^2- 3g'(t)^2}{(1-g(t)^2)^{3/2}}\left (\sqrt{1-{R^\ast}^2} + R^\ast \arcsin R^\ast \right )-1 \right ) dt + \frac{2}{\sqrt{3}},$$
where 
$$g(t)=\frac{\sin(t)}{t}, \mbox{ and } R^\ast=R^\ast(t) = \frac{g''(t)(1-g(t)^2) + g(t)g'(t)^2}{1/3(1-g(t)^2) - g'(t)^2}.$$
Granville and Wigman established this beautiful result by a delicate method basing on the Kac-Rice formula. More recently, Aza\"is and Le\'on \cite{AL} provided an important alternative approach basing on Wiener chaos decomposition. Roughly speaking, they showed that $P_n(t,Y)$ converges in certain strong sense to the stationary Gaussian process of covariance $r(t)= \sin(t)/t$, from which variance and CLT can be deduced.  

More relevant to our current note, the above result has been extended recently by a ground-breaking result of Bally, Caramellino, and Poly \cite{BCP} to more general distributions where certain continuousness is assumed. To discuss this extension, we first introduce some of their notions. We say that $Y_i$ satisfies the (two-dimensional) Doeblin's condition if there exists $ a_i \in \R^2$ and $r,\eta \in (0,1)$ such that for any $A \subset B_r(a_i)$,
$$\P(Y_i \in A) \ge \eta \la(A).$$
Let $\mathcal{D}(r,\eta)$ denote the sequences of random variables $Y_k= (y_{k1}, y_{k2})$ satisfying the Doeblin's condition, with $\E y_{kj_1} y_{kj_2} = \delta_{j_1j_2}$, and uniformly bounded moments of all orders
$$ \sup_{k}\E|Y_k|^p < \infty\quad\forall p, $$
where the $Y_k$ are independent but not necessarily identically distributed.

Suppose that $(Y_k) \in \mathcal{D}(r,\eta)$ and for all $\al=(\al_1,\dots, \al_m) \in \{1,2\}^m$ with $m=3,4$, the following limits exist 
$$\lim_{n\to \infty} \E \prod_{i=1}^m y_{n\al_i} = y_\infty (\al).$$ 

The following result from \cite{BCP} was formulated for $N_n([0,n\pi],Y)$, the number of roots inside $[0,n\pi]$ of $P_n(t,Y)$\footnote{The authors of \cite{BCP}  considered the number of roots inside $[0,\pi]$ of $P_n(nt,Y)$, which is the same as our $N_n([0,n\pi],Y)$.}. Let $N_n([0,n\pi],G)$ be the number of roots inside $[0,n\pi]$ of $P_n(nt,G)$ which is the random polynomial with coefficients $ y_{ij} $ being standard Gaussian.
\begin{theorem}\cite[Theorem 2.1]{BCP}\label{thm:BCP}  We have 
	$$\lim_n \frac{1}{n} \Var(N_n([0,n\pi],Y)) = \lim_n \frac{1}{n} \Var(N_n([0,n\pi],G)) + \frac{1}{60} y_\ast$$
	with 
	$$y_\ast = (y(1,1,2,2)-1)+(y(2,2,1,1)-1) + (y(1,1,1,1)-3) + (y(2,2,2,2)-3).$$
	In particular, if the $y_{ij}$ are iid copies of a random variable $\xi$ of mean zero, variance one and satisfies the (one-dimensional) Doeblin's condition, then
	\begin{equation}\label{eqn:motive}
	\lim_{n\to \infty} \frac{1}{n} \Var(N_n([0,n\pi],Y)) =  \lim_n \frac{1}{n} \Var(N_n([0,n\pi],G)) + \frac{1}{30} \E(\xi^4-3).
	\end{equation}
\end{theorem}

This result implies strong concentration around the mean of $N_n$. More crucially, it says that the variance is not universal with respect to second order normalization of $\xi$ (having mean zero and variance one). At the same time, it also suggests a possible universal picture that in the limit, the ratio $V_n/n$ asymptotically depends on $y_\ast$, and particularly on the fourth moment  in the iid case. 

In this paper, we confirm this phenomenon and completely remove the Doeblin's condition.

\begin{theorem}[main theorem]\label{thm:var} Assume that $y_{ij}, 1\le i\le n, j=1,2$ are iid copies a random variable $\xi$ of mean zero, variance one, and $\E|\xi|^{M_0}<\infty$ for a sufficiently large positive number $M_0$. Then 
	$$\lim_n \frac{1}{n} \Var(N_n) = c_{G} + \frac{2}{15} \E(\xi^4-3),$$
	where we recall that $c_G$ is the constant from Theorem \ref{thm:GW}.
\end{theorem}
We thus obtain that for the case where $y_{i, j}$ are Rademacher random variables, 
 $$\lim_n \frac{1}{n} \Var(N_n) = c_{G} - \frac{4}{15} \approx 0.29159.$$

Our numerical experiments appear to be in accordance with these results as shown in Figure \ref{figure:2pi}.

Note that our result is stated for the number of roots over $[-n\pi,n\pi]$, but the approach automatically works for roots over $[0,n\pi]$ as well. As a matter of fact, most of our arguments work for random variables $|\xi|$ of bounded $(2+\eps_0)$-moment, except at the Edgeworth expansion step (for instance Theorem \ref{thm:Edgeworth:0}) where we assume boundedness of moments. 

We can view Theorem \ref{thm:BCP} and Theorem \ref{thm:var} as a mixture of universality and non-universality. The fact that the variance is linear in $n$ indicates that there is no correlation (repulsion and attraction) among sufficiently far apart roots, and this phenomenon is universal in the sense that it suffices to assume $|\xi|$ to have bounded moments. However, the multiplicative constant, which is determined by the correlation of nearby roots, is affected by the kurtosis as seen.

Finally, we also invite the reader to Theorem \ref{thm:var:asymp} which says that under a very general setting (including the non-iid case) there is already a significant cancellation in the variance formula. More precisely, there exists a positive constant $c$ such that 
\begin{equation}\label{eqn:var:asymp}
\Var(N_n) = O(n^{2-c}).
\end{equation}

\begin{figure}[ht!]
	\centering
	\includegraphics[width=100mm]{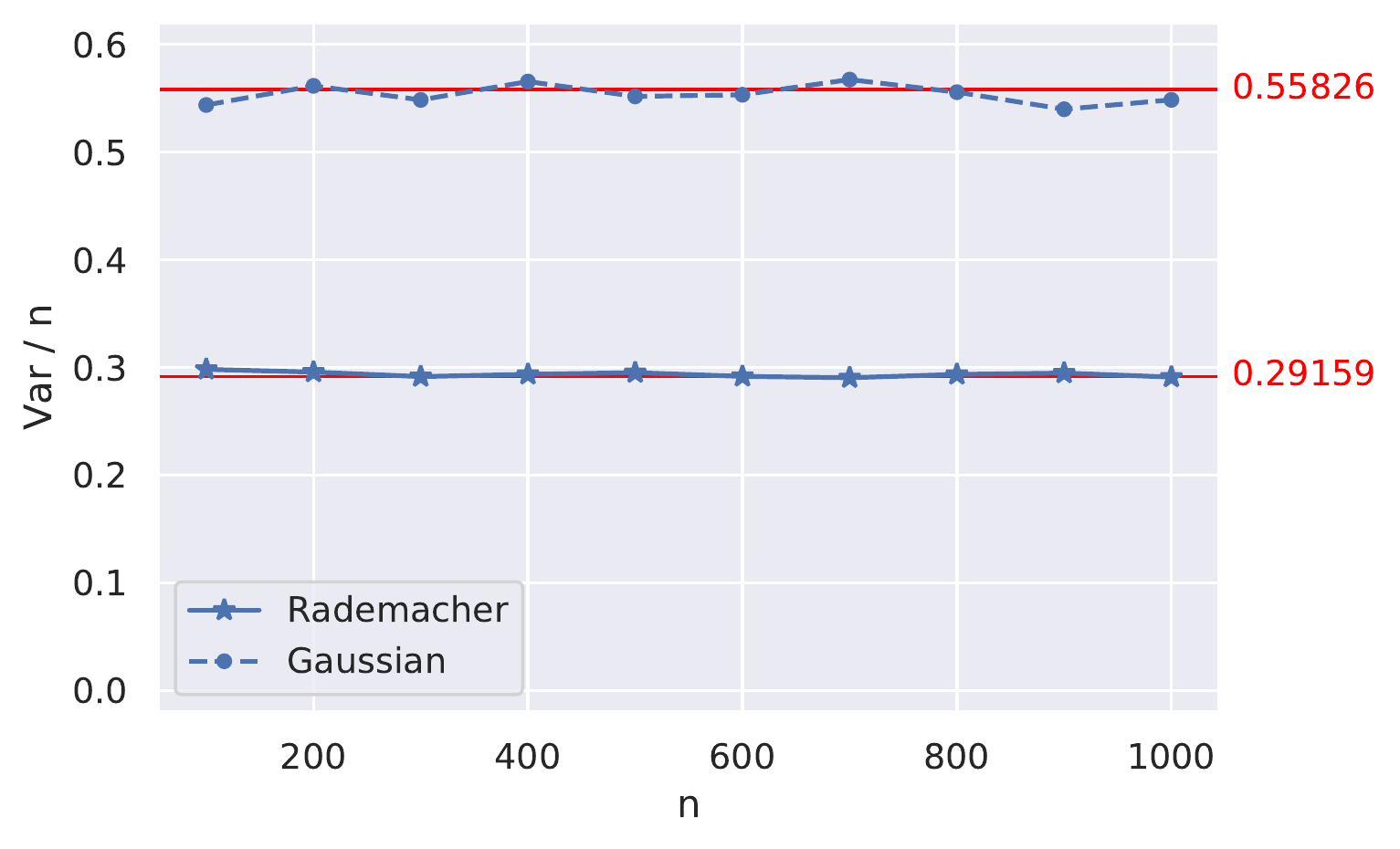}
	\caption{Sample variance (divided by $n$) of the number of roots in $[-n\pi, n\pi]$ for Gaussian random variables (dashed line) and Rademacher random variables (solid line).}
	 \label{figure:2pi}
\end{figure}
\begin{figure}[ht!]
	\centering
	\includegraphics[width=100mm]{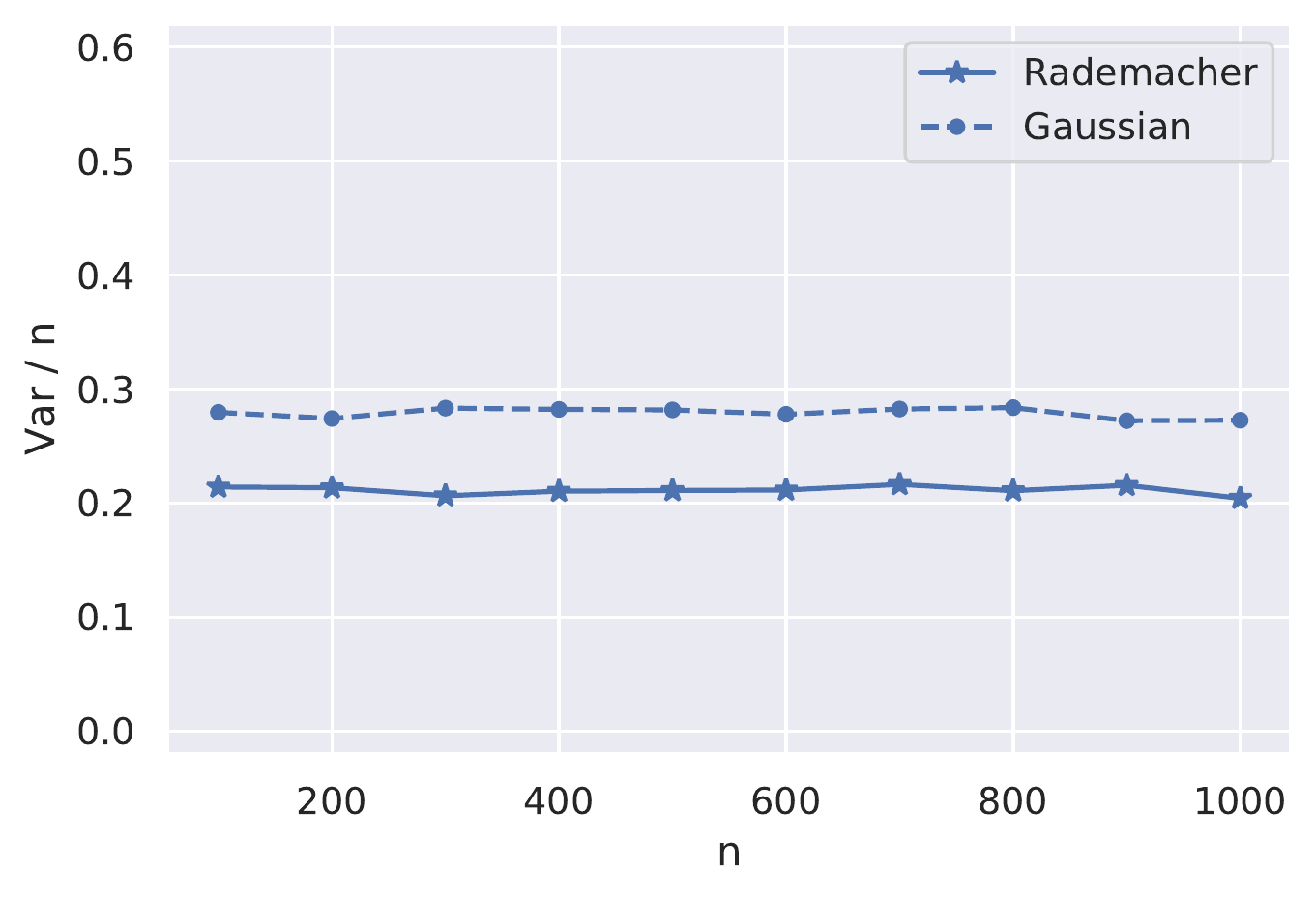}
	\caption{Sample variance (divided by $n$) of the number of roots in $[0,n \pi]$ for Gaussian random variables (dashed line) and Rademacher random variables (solid line).}
\end{figure}

\section{Our methods} We first mention briefly the approach by Bally et al. to prove Theorem \ref{thm:BCP}. Here, powerful tools such as Maliavin calculus and Wiener chaos  theory (see \cite{AL} and the references therein) do not apply under the Doeblin's condition. Instead, the authors above have developed a sophisticated method using the Edgeworth expansion and approximate Kac-Rice formulas basing on their previous results in \cite{BCP-D}.

Generally speaking, for Theorem \ref{thm:var}, we will follow the same machinery. However, as we have to deal with discrete random variables,  none of the results from \cite{BCP-D} and \cite{BCP} could be applied. For instance, in our opinion, it is a non-trivial problem to study the small ball probability for the random walks associated to $P_n(t, Y)$ without Doeblin's condition.

We would also like to point out that, broadly speaking, using the Edgeworth expansion to study the distribution of normalized sums of independent random variables is a classical approach (see \cite{BR}) and this approach was also used by  Bally et al. \cite{BCP-D, BCP}. The novelty in our argument is a very {\it fine} estimate for characteristics functions that works for a large class of distributions (including the discrete cases). This is where we deviate from Bally et al. \cite{BCP-D, BCP}, who used a completely different approach to deal with non-smooth distributions. More precisely, in their papers \cite{BCP-D, BCP}, the authors use the Nummelin splitting (which requires Doeblin's condition) to decompose non-smooth distribution into two parts: a smooth part that can be treated directly by Edgeworth expansion methods, and a noisy part that can be treated by  Wiener chaos techniques. Our modified approach circumvents the need for Nummelin's splitting and therefore avoids the need for anti-concentration conditions like the Doeblin condition in Bally et al. \cite{BCP-D, BCP}.

One trade-off that we need to face in order to obtain the generality of our result is that we necessarily rule out a set of points that are well-approximated by the integer lattice (see Condition \ref{cond:s,t}). To show that this set does not contribute significantly to the whole picture, we utilize a universality result in \cite{ONgV} (Theorem \ref{thm:local:uni}) which, roughly speaking, says that the difference between the variance of the number of roots of $P_n(\cdot, Y)$ and $P_n(\cdot, G)$ over small intervals is negligible. 

In what follows we sketch the highlights, some of which are of independent interest. (For instance, a variant of Theorem \ref{thm:smallball:2} finds some applications in \cite{NgZ}.)

\subsection{Small ball estimates and characteristic functions}\label{sub:smallball} Here, we only assume $\xi$ to have mean zero, variance one and bounded $(2+\eps_0)$-moment for any $\eps_0>0$. 

For $t \in [-n\pi, n\pi]$, we define the vectors
\begin{equation}\label{u:1}
\Bu_i(t):= \left(\cos\left(\frac{it}{n}\right), -\frac{i}{n} \sin \left(\frac{it}{n}\right)\right )\quad \mbox{ and }\quad \Bu_{i}'(t) :=  \left (\sin\left(\frac{it}{n}\right), \frac{i}{n} \cos \left(\frac{it}{n}\right)\right).
\end{equation}

Assume that $y_{ij}, 1\le i\le n, j=1,2$, are iid copies of a random variable $\xi$ of mean zero and variance one. Consider the random walk in $\R^2$
\begin{equation}\label{eqn:St}
S_{n}(Y,t):=\sum_{i=1}^n y_{i1} \Bu_i + y_{i2} \Bu_i'.
\end{equation}

This random walk can also be written as $S_n(t,Y) =\sum_{i=1}^n C_n(i,t) Y_i $, where $Y_i =(y_{i1},y_{i2})$ and 
\begin{equation}\label{eqn:C_n:t}
C_{n}(i,t)=\left(
\begin{matrix}
\cos\left(\frac{it}{n}\right) & \sin \left(\frac{it}{n}\right)\\
-\frac{i}{n} \sin \left(\frac{it}{n}\right) & \frac{i}{n} \cos\left(\frac{it}{n}\right) \\ 
\end{matrix} \right).
\end{equation}

Note that for some values of $t$ such as $t=o(1)$, the random walk does not spread out in the Radamacher case.  We will show that these are the only cases to cause this clustering. 
\begin{condition}\label{cond:t} Let $\tau$ be a constant to be chosen sufficiently small. A number $t \in [-n\pi,n\pi]$ is said to satisfy Condition \ref{cond:t} if there does not exist a non-zero integer $l$ with $|l| \le n^{\tau}$ such that 
$$\left \|l\frac{t}{ \pi n} \right \|_{\R/ \Z} \le n^{-1+8\tau}.$$ 
 \end{condition}
 Here $\|.\|_{\R/ \Z}$ is the distance to the nearest integer. In other words, the above condition requires that $t/\pi n$ cannot be within a distance of $n^{-1+o(1)}$ from rational numbers of denominator $n^{o(1)}$.
 
  \begin{theorem}\label{thm:smallball:2} Let $C>0$ be a given constant. Assume that $t$ satisfies Condition \ref{cond:t} with sufficiently small $\tau$. Then for $\delta = n^{-C}$ and for any open ball $B(a,\delta)$, we have
$$\P\left ( \frac{1}{\sqrt{n}} S_n(t,Y) \in B(a,\delta) \right ) =O\left (\delta^2\right ).$$
\end{theorem}

As mentioned before, our condition on $t$ is almost optimal. Towards Theorem \ref{thm:var}, as we will be dealing with pair correlations, we will need to work with vectors in $\R^4$. Let $s,t$ be given, define the vectors $\Bv_i, \Bv_i'$ as follows
\begin{equation}\label{v:1}
\Bv_i(s,t):= \left(\cos\left(\frac{it}{n}\right), -\frac{i}{n} \sin \left(\frac{it}{n}\right),\cos\left (\frac{is}{n}\right ), -\frac{i}{n} \sin \left (\frac{is}{n}\right )\right )
\end{equation}
and
\begin{equation}\label{v:2}\Bv_{i}'(s,t) =  \left(\sin\left(\frac{it}{n}\right), \frac{i}{n} \cos \left(\frac{it}{n}\right),\sin\left (\frac{is}{n}\right ), \frac{i}{n} \cos \left (\frac{is}{n}\right )\right).
\end{equation}
These vectors are obtained by simply concatenating $\Bu_i(t), \Bu_i(s)$ and $\Bu'_i(t), \Bu'_i(s)$, respectively.

Here we are interested in the random walk 
\begin{equation}\label{eqn:Sst}
S_{n}(s,t,Y):= \sum_{i=1}^n y_{i1} \Bv_i + y_{i2} \Bv_i'.
\end{equation}
Using \eqref{eqn:C_n:t}, if we let $C_n(i,s,t)$ be the $4 \times 2$ matrix obtanied as a joint of $C_n(i,t)$ and $C_n(i,s)$, then we can see that this random walk can also be written as $S_n(s,t,Y) = \sum_{i=1}^n C_n(i,s,t) Y_i$.

We will assume that $s/\pi n$ and $t/\pi n$ cannot be jointly well-approximated by rational numbers.

\begin{condition}\label{cond:s,t} Let $\tau$ be a constant to be chosen sufficiently small. Two numbers $s,t \in [-n\pi,n\pi]$ are said to satisfy Condition \ref{cond:s,t} if there do not exist integers $k,l$ with $|k|,|l| \le n^{\tau}$, not both zero, such that
$$\left \|k \frac{s}{ \pi n} + l \frac{t}{\pi n} \right \|_{\R/ \Z} \le n^{-1+8\tau}.$$
\end{condition}
Note that if $ s, t $ satisfy Condition \ref{cond:s,t} then each of them satistifes Condition \ref{cond:t} separately. It is clear that the measure of $(s/n, t/n) \in [-\pi, \pi]^2$ that does not satisfies the above condition is $n^{-1+O(\tau)}$. We will show the following small ball probability.
\begin{theorem}\label{thm:smallball:4} Let $C>0$ be a given constant. Assume that $s,t$ satisfy Condition \ref{cond:s,t} with sufficiently small $\tau$. Then for $\delta = n^{-C}$and for any open ball $B(a,\delta)$, we have
$$\P\left (\frac{1}{\sqrt{n}}  S_n(s,t,Y) \in B(a,\delta) \right ) =O\left (\delta^4\right ).$$
\end{theorem}
To prove these small ball estimates, we will rely on the following results on the characteristic functions. First, for the random walk $S_n(t,Y)$, let 
$$\phi_{\R^2}(x) = \prod_{i=1}^n \phi_i(x) = \prod_{i=1}^n \E e( y_{i1} \langle \Bu_i, x\rangle) \prod_{i=1}^n \E e(y_{i2} \langle \Bu_i', x \rangle),\quad x\in \R^2,$$ 
where $e(y)= e^{iy}$. We will show that this function decays very fast.

\begin{theorem}\label{thm:fourier:2} Let $C_\ast>0$ be any given constant, and $t$ satisfies Condition \ref{cond:t} for some sufficiently small constant $\tau$. Then the following holds for sufficiently large $n$ and sufficiently small $\tau_\ast$ (depending on $C_\ast$ and $\tau$). For any $n^{5\tau -1/2} \le \|x\|_2 \le n^{C_\ast}$, we have 
$$|\phi_{\R^2}(x)| \le \exp(-n^{\tau_\ast}).$$
\end{theorem}
We note that this was also studied in \cite{KSch} for the Radamacher case, covering up to $\|x\|_2\le n^{1/2+o(1)}$. This result has been improved to $\|x\|_2\le n^{1-o(1)}$ in \cite{NgZ} recently for any $\xi$ of variance one. Our current approach to prove Theorem \ref{thm:fourier:2} goes deeper than those of \cite{KSch,NgZ} where we need to solve certain inverse-type problems. (See Sections \ref{section:fourier} and \ref{section:fourier:2} for more details.)

Similarly to the case of $\R^2$, to establish these results we will study the characteristic function 
$$\phi_{\R^4}(x)= \prod_{i=1}^n \phi_i(x) = \prod_{i=1}^n \E e( y_{i1} \langle \Bv_i, x\rangle) \prod_{i=1}^n \E e( y_{i2} \langle \Bv_i', x\rangle), \quad x\in \R^{4}.$$

\begin{theorem}\label{thm:fourier} Let $C_\ast>0$ be any given constant, and assume that $s,t$ satisfy Condition \ref{cond:s,t} for some sufficiently small constant $\tau$. Then the following holds for sufficiently large $n$ and sufficiently small $\tau_\ast$ (depending on $C_\ast$ and $\tau$). For any $n^{5\tau -1/2} \le \|x\|_2 \le n^{C_\ast}$, we have 
 $$|\phi_{\R^4}(x)|  \le \exp(-n^{\tau_\ast}).$$
\end{theorem}

We note that this result implies Theorem \ref{thm:fourier:2} because with $t$ under Condition \ref{cond:t}, there exists $s\in [-n\pi,n\pi]$ so that $s,t$ satisfies Condition \ref{cond:s,t}. We then apply Theorem \ref{thm:fourier} with $x=(x_1,x_2,0,0)$. However, we will present a separate proof of Theorem \ref{thm:fourier:2} in Section \ref{section:fourier:2} to serve as a preparation for our more technical treatment of Theorem \ref{thm:fourier} in Section \ref{section:fourier}.

\subsection{Approximated Kac-Rice formula and proof conclusion}\label{sub:KR}

We next briefly recall the use of approximated Kac-Rice formula. 
 
  Consider a smooth function $f$ on an interval $[a,b]$ where for all $t\in [a, b]$, we have $|f(t)| + |f'(t)|>0 $. Then according to a celebrated formula of Kac and Rice, the number of roots of $ f $ in $ [a, b] $ is given by
$$\lim_{\delta \to 0} \frac{1}{2\delta} \int_a^b |f'(t)| 1_{|f(t)|<\delta} dt.$$ 
Using this approximated formula for our polynomial $P_n(\cdot, Y)$, we  will show that for 
\begin{equation}\label{eqn:delta}
\delta=\delta_n=n^{-5},
\end{equation}
we have
$$\lim_{n \to \infty} \frac{1}{n} \Var(N_n(Y)) = \lim_{n \to \infty} \frac{1}{n} \Var\left ( \frac{1}{2\delta} \int_{-n \pi}^{n\pi} |P'_n(t,Y)| 1_{ |P_n(t,Y)|<\delta }\right )dt.$$
After expanding out the integrals, we will need to compute 
$$\frac{1}{\delta^2}\int_{-n \pi}^{n \pi} \int_{-n \pi}^{n \pi} \cov\left (|P'_n(t,Y)| 1_{ |P_n(t,Y)|<\delta }, |P'_n(s,Y)| 1_{ |P_n(s,Y)|<\delta }\right ) dsdt.$$
Let  us introduce a few notations to simplify the discussion. We define the following even functions that appear in the above formula
\begin{equation}\label{eqn:F:delta} 
F_\delta(x) := \frac{1}{2 \delta} 1_{|x| <\delta}, \quad x \in \R
\end{equation}
and
\begin{equation}\label{eqn:Phi:delta} 
\Phi_\delta(x) := |x_2| F_\delta(x_1),  \quad x=(x_1,x_2)\in \R^2
\end{equation}
and 
\begin{equation}\label{eqn:Psi:delta}  
\Psi_\delta(x) := \Phi_\delta(x_1,x_2) \Phi_\delta(x_3,x_4) = |x_2| F_\delta(x_1) |x_4| F_\delta(x_3),  \quad x=(x_1,x_2,x_3,x_4) \in \R^4.
\end{equation}
We have
$$\Phi_\delta\left (\frac{1}{\sqrt{n}}  S_n(t,Y)\right )  = |P'_n(t,Y)| \times \frac{1}{2\delta} 1_{|P_n(t,Y)| \le \delta} =: \phi_\delta(t,Y)$$
and
$$\Psi_\delta\left (\frac{1}{\sqrt{n}}  S_n(s,t,Y)\right ) = \phi_\delta(s,Y) \phi_\delta(t,Y).$$
Finally, for short we introduce
\begin{align}\label{eq:def:vn}
v_n(s,t,Y):&=\cov(P'_n(s,Y) 1_{ |P_n(s,Y)|<\delta }, P'_n(t,Y) 1_{ |P_n(t,Y)|<\delta })\nonumber \\
&= \E \phi_{\delta}(s,Y)\phi_{\delta}(t,Y)  - \E \phi_{\delta}(s,Y) \E \phi_{\delta}(t,Y).
\end{align}
For a given $\ep>0$, we will decompose the interval $(-n\pi, n\pi)$ into subintervals of length $\ep$
\begin{equation}\label{def:Ik}
	I_{k} : = [k\ep, (k+1\ep)]\subset [-n\pi, n\pi]
\end{equation}
 and let 
 \begin{equation}\label{eq:def:D:n:ep}
 \quad D_{n, \ep} :=\bigcup_{(k, p)\in \mathcal D} I_{k} \times I_{p} 
 \end{equation}
 where  $\mathcal D$ is the set consisting of all $(k, p)$ with $-n\pi/\ep\le k<p\le n\pi/\ep$ such that for all $s\in I_p$ and $t\in I_k$,  $s$ and $t$ satisfy Condition \eqref{cond:s,t}.

Let $N_n(G)$ and $v_n(s,t,G)$ be the statistics when the $y_{ij}$ are standard Gaussian. In our next lemma, we show that, in comparison with the Gaussian part, the contribution $R_n$ from $(s,t) \notin D_{n, \ep}$ is negligible in the variance computation.
 
\begin{lemma}\label{lemma:var}
	With $\delta$ as in \eqref{eqn:delta} we have
	\begin{equation}\label{eq:var:v}
	\Var N_n(Y) = \Var N_n(G) + 2\int_{D_{n,\eps}} \left (v_n(s, t,  Y) - v_n(s, t, G)\right ) dsdt + R_{n, \ep}
	\end{equation}
	where
	\begin{equation}\label{key}
	\lim_{n} \frac{R_{n, \ep}}{n} = 0.\nonumber
	\end{equation}
\end{lemma}

Therefore, we will need to control $\int_{D_{n,\eps}} (v_n(s,t,Y) -v_n(s,t,G))dsdt$ from \eqref{eq:var:v}, for which we will use Proposition \ref{prop:Edgeworth:trig:delta} to show the following (see also \cite[Lemma 5.1]{BCP}).
\begin{prop}\label{prop:v_n}
For every $\eps>0$ we have 
$$\lim_{n} \frac{1}{n} \int_{D_{n,\eps}} \big(v_n(s,t,Y) -v_n(s,t,G)\big) ds dt = \frac{1}{15} \E (\xi^4-3) + r_\eps$$
with $|r_\eps| = O(\eps)$.
\end{prop}

Combining Lemma \ref{lemma:var} and Proposition \ref{prop:v_n}, with $\eps \to 0$, we obtain Theorem \ref{thm:var}. 
We will prove Lemma \ref{lemma:var} in Section \ref{section:lemma:var} and Proposition \ref{prop:v_n}  in Section \ref{section:proof:prop:v_n}. Notice that for these results we will also need to incorporate other existing results in the literature (notably \cite{ONgV}). We will also justify \eqref{eqn:var:asymp} by the same way (see Section \ref{section:var:asymp}).

\subsection{Edgeworth expansion}\label{sub:E} We now compare $v_n(s,t,Y)$ with $v_n(s,t,G)$ by using Edgeworth expansion of order three. This approach is originated from \cite{BCP}, but our proof is directly based on the study of characteristic functions. 

If $X_i$ are iid real random variables of mean zero and variance one, the Central Limit Theorem says that, with $\Phi$ being the C.D.F. of the standard Gaussian distribution, for any real number $x$, we have
$$\lim_{n\to \infty} \left |\P\left (\frac{S_n}{\sqrt{n}} \le x\right ) - \Phi(x)\right |=0$$
where $ S_n  = \sum_{k=1}^{n} X_k$.

The Edgeworth expansion by Edegworth \cite{Edgeworth}, Chebyshev \cite{Ch}, and Cram\'er \cite{Cramer} says that under the so-called Cram\'er condition, if the $X_i$ has bounded $s_0$ moments  then there exist explicit polynomials $P_0,\dots, P_{s_0-1}$ with coefficients depend on the cumulants of $\frac{S_n}{\sqrt{n}}$ such that
$$
\left |\P\left (\frac{S_n}{\sqrt{n}} \le x\right ) - \sum_{r=0}^{s_0-1} n^{-r/2} P_r(-D) (\Phi(x))\right |=O(n^{-s_0/2})
$$
where $D$ is the differential operator. 

To prove Proposition \ref{prop:v_n}, we will carry out the Edgeworth expansion for $\E \Psi_\delta\left (\frac{1}{\sqrt{n}}  S_n(s,t,Y)\right )$ as well as for  $\E \Phi_\delta\left (\frac{1}{\sqrt{n}}  S_n(s,Y)\right )$ and  $\E \Phi_\delta\left (\frac{1}{\sqrt{n}}  S_n(t,Y)\right )$, where we recall $S_n(t,Y)$ and $S_n(s,t,Y)$ from \eqref{eqn:St} and \eqref{eqn:Sst}, and the functions $\Phi_\delta$ and $\Psi_\delta$ from  \eqref{eqn:Phi:delta} and \eqref{eqn:Psi:delta}.

In what follows, we shall mention briefly our main contribution; we invite the reader to Section \ref{section:Edgeworth} and Section \ref{section:Edgeworth:trig:delta} for more details.

We let $X_n(t,Y)$ be the vector $(C_n(k,t) Y_k)_{k=1}^n$ and $X_n(s,t,Y)$ be the vector $(C_n(k,s,t) Y_k)_{k=1}^n$. We also let $V_n(t)=\frac{1}{n} \sum_{k=1}^n C_n(k,t) C_n(k,t)^\ast$ and $V_n(s,t)=\frac{1}{n} \sum_{k=1}^n C_n(k,s,t) C_n(k,s,t)^\ast$ be the average covariance matrices. Finally, we defer the technical definition of $\Gamma_{n,2}$, which occurs in the following statement, to \eqref{eqn:Gamma:2}. We will show the following CLT type estimates.

\begin{prop}\label{prop:Edgeworth:trig:delta} Assume that $\xi$ has mean zero, variance one, and $\E |\xi|^{M_0}<\infty$ for sufficiently large $M_0$. Assume that $s,t$ satisfy Condition \ref{cond:s,t}. Then we have 
\begin{equation}\label{eqn:P}
|\E F_\delta(P_n(t,Y))- \E F_\delta(P_n(t,G))| \le \frac{C}{n^{1/2}}, 
\end{equation}
and
\begin{align}\label{eqn:t,Y}
\Big|  \E \Phi_\delta\left (\frac{1}{\sqrt{n}} S_n(t,Y)\right ) - \E \Phi_\delta\left (\frac{1}{\sqrt{n}} S_n(t,G)\right ) & - \frac{1}{n} \E\big[\Phi_\delta(I_2(\la)^{1/2} W_2) \Gamma_{n,2} \big(I_2(\la)^{-1/2} X_n(t,Y), W_2)\big] \Big| \nonumber \\
&\le  \frac{C}{n^{3/2}}+  \frac{1}{n} r_n(t, \Phi_\delta), 
\end{align}
and
\begin{align}\label{eqn:s,t,Y}
\Big|  \E \Psi_\delta \left(\frac{1}{\sqrt{n}} S_n(s,t,Y)\right)- \E \Psi_\delta \left (\frac{1}{\sqrt{n}} S_n(s,t,G)\right ) & - \frac{1}{n} \E \big [\Psi_\delta (I_4(\la)^{1/2} W_4) \Gamma_{n,2} \big(I_4(\la)^{-1/2} X_n(s,t,Y),W_4\big )\big]\Big| \nonumber \\
& \le  \frac{C}{n^{3/2}} + \frac{1}{n} r_n(s,t,\Psi_\delta), 
\end{align}
where $I_2(\la)$ and $I_4(\la)$ are any invertible diagonal matrices  \footnote{The vector parameter $\la$ stands for the diagonal entries, see \eqref{eqn:VtoI}.} and $W_2,W_4$ are standard Gaussian vectors in $\R, \R^2$ and $\R^4$ respectively, and where the implied constants are allowed to depend on the $M_0$-moment of $\xi$, on the constants in Conditions \ref{cond:t} and \ref{cond:s,t}, and on a lower bound of the least singular values of $V_n(t), V_n(s,t)$ and $I_2(\la),I_4(\la)$. Furthermore we have the following bounds 
$$r_n(t, \Phi_\delta)=O(\| V_n(t) - I_2(\la)\|_{2})  \mbox{ and }  r_n(s,t,\Psi_\delta) = O(\| V_n(s,t) - I_4(\la)\|_{2}).$$
\end{prop}

We also refer the reader to \cite[Section 3]{BCP} where a better error bound was obtained under the Doeblin's conditions. In application (Section \ref{section:proof:prop:v_n}), we will choose $I_2(\la)$ and $I_4(\la)$ so that $r_n \to 0$.

 We will prove Proposition \ref{prop:Edgeworth:trig:delta} by giving a general Edgeworth expansion result in Section \ref{section:Edgeworth}, and then use it to conclude the proof in Section \ref{section:Edgeworth:trig:delta}. Roughly speaking, our approach here is based on the work of  Bhattacharya and Rao \cite{BR} (see also \cite{AP}) which relates Edgeworth expansion to the growth of characteristic functions of the corresponding random walks. 
 
{\bf Notations.} Throughout the note $n$ is the parameter to be sent to $\infty$. We write $X =
O(Y)$, $Y=\Omega(X)$, $X \ll Y$, or $Y \gg X$ if $|X| \leq CY$
for some fixed $C$; this $C$ can depend on
other fixed quantities such as the $M_0$-moment of $\xi$.  If $X\ll Y$ and $Y\ll X$, we say that $Y = \Theta(X)$ or $X \asymp Y$. We write $\omega(1)$ for a number that tends to $\infty$ as $n\to \infty$.

\section{Small ball probability}\label{section:smallball}
In this section, we address the small ball probabilities, we will just prove the $\R^4$ case (i.e. $d=4$) because the $\R^2$ case can be proved similarly (by using Theorem \ref{thm:fourier:2} instead of Theorem \ref{thm:fourier}). 

\begin{proof}[Proof of Theorem \ref{thm:smallball:4}]  Let 
$$t_0=\delta^{-1}= n^{C}.$$ 
By a standard procedure (see for instance \cite[Eq. 5.4]{AP}), we can bound the small ball probability by characteristic functions as follows
$$\P\left (\frac{1}{\sqrt{n}} \sum_i  z_i \Bv_i+z_i' \Bv_i'  \in B(a,\delta) \right ) \le C_d  \left (\frac{n}{t_0^2}\right )^{d/2} \int_{\R^d} \prod_i \phi_i(u) e^{-\frac{n \|u\|_2^2}{2 t_0^2}} du.$$
Choose $C_\ast$ to be sufficiently large compared to $C$. We break the integral into three parts, $J_1$ when $\|u\|_2 \le r_0 =O(1)$, $J_2$ when  $r_0 \le \|u\|_2 \le R=n^{C_\ast}$, and $J_3$ for the remaining part.

For $J_1$, recall that
$$ \left |\prod \phi_i( u)\right | \le \exp\left (-\sum_i \|\langle \Bv_i, u \rangle\|_{\R/\Z}^2/2\right ).$$
So if $\|u\|_2 \le c$ for sufficiently small $c$, then we have $\|\langle \Bv_i, u \rangle\|_{\R/\Z} = \|\langle \Bv_i, u \rangle\|_2$, and so because of Condition \ref{cond:s,t} (where we would need that $\sum_i \langle \Be, \Bv_i\rangle^2 \ge c' n$ for any unit vector $\Be$, see also Claim \ref{claim:t} with $|I| \asymp n$) we have

$$\sum_i \|\langle \Bv_i, u \rangle\|_{\R/\Z}^2/2 = \sum_i \|\langle \Bv_i, u \rangle\|_2^2/2 \ge c' n\|u\|_2^2.$$
Thus, 
\begin{align*}
J_1 &=  C_d \left (\frac{n}{t^2}\right )^{d/2} \int_{\|u\|_2 \le r_0} \prod_i \phi_i(u) e^{-\frac{n \|u\|_2^2}{2 t_0^2}} du \le  C_d \left(\frac{n}{t_0^2}\right )^{d/2} \int_{\|u\|_2 \le r_0}  e^{-\frac{n \|u\|_2^2}{2 t_0^2} - c'n \|u\|_2^2} du,
\end{align*}
and so
\begin{align*}
J_1 &\le
 C_d  \left (\frac{n}{t_0^2}\right )^{d/2} \int_{\|u\|_2 \le r_0}  e^{-(\frac{n}{2 t_0^2} + c'n) \|u\|_2^2} du =  O_d\left(\frac{1}{(c'' t_0^2 +1)^{d/2}}\right ) = O_d\left(\delta^{d}\right ).
\end{align*}

For $J_2$, recall by Theorem \ref{thm:fourier} that for $r_0 \le \|u\|_2 \le R=n^{C_\ast}$, we have
$$ |\prod \phi_i( u)| \le e^{-n^{-\tau_\ast}}.$$
Thus,
\begin{align*}
J_2 &=  C_d \left (\frac{n}{t_0^2}\right )^{d/2} \int_{r_0 \le \|u\|_2 \le R} \prod_i \phi_i(u) e^{-\frac{n \|u\|_2^2}{2 t_0^2}} du  \le  C_d  \left (\frac{n}{t_0^2}\right )^{d/2} \int_{r_0 \le \|u\|_2 \le R}  e^{-n^{\tau_\ast}} du,
\end{align*}
and so, 
\begin{align*}
J_2& \le  O_d\left (n^d \left (\frac{n}{t_0^2}\right )^{d/2}  e^{-n^{\tau_\ast}} \right ) = O_d\left ( e^{-n^{\tau_\ast/2}}\right ).
\end{align*}

For $J_3$, we have
\begin{align*}
J_3 &=  C_d\left (\frac{n}{t_0^2}\right )^{d/2} \int_{ \|u\|_2 \ge n^{C_\ast}} \prod_i \phi_i(u) e^{-\frac{n \|u\|_2^2}{2 t_0^2}} du = O_d\left (e^{-n}\right )
\end{align*}
as we chose $C_\ast$ sufficiently large compare to $C$.
 \end{proof}
Before concluding this section, we introduce some useful corollaries of our small ball estimates. For short, let $\CG=\CG_{\tau}$ be the collection of $t\in [-n\pi, n \pi]$ that satisfies Condition \ref{cond:t}.

We first deduce from Theorem \ref{thm:smallball:2} a small ball estimate for $P_n(t, Y)$ alone, which will be useful later.
 \begin{cor}\label{cor:smallball:2} Let $C>1$ be a given constant. Assume that $y_{ij}$ are iid copies of a random variable $\xi$ of mean zero, variance one, and bounded $\E(|\xi|^{M_0})<\infty$ for some even positive integer $M_0$. Assume that $t\in \CG$ with sufficiently small $\tau$. Then for $\delta = n^{-C}$ and any open interval $(a-\delta, a+\delta)$ we have
	$$\P\left (P_n(t,Y) \in (a-\delta, a+\delta) \right ) =O\left (\delta^{\frac{M_0}{M_0+1}}\right ).$$
\end{cor}
 \begin{proof} Since the random variables $y_{i1}, y_{i2}$ are uncorrelated with mean 0 and bounded $M_0$ moments, we have
$$\E \left (P_n'(t, Y)\right )^{M_0} =O (1).$$
Thus, by Markov's inequality, for a positive constant $A>1$ to be chosen, 
$$\P\left (|P'_n(t, Y)|\ge n^{A}\right )=O(n^{-M_0 A}).$$
We have
\begin{eqnarray}
\P\left (|P_n(t, Y)|<\delta\right )\le \P\left (|P'_n(t, Y)|\ge n^{A}\right ) + \P\left (|P_n(t, Y)|<\delta, |P'_n(t, Y)|\le n^{A}\right ).\nonumber
\end{eqnarray}
Since the latter event is a subset of a union of $n^{A}\delta^{-1}$ events of the form $S_n(Y, t)\in B(a, \delta)$ for some $a\in \C$, we apply Theorem \ref{thm:smallball:2} to get
\begin{eqnarray}
\P\left (|P_n(t, Y)|<\delta\right )=O(n^{-M_0 A} + n^{A}\delta^{-1} \delta^{2}).\nonumber
\end{eqnarray}
By choosing $A = \frac{C}{1+M_0}$, this proves Corollary \ref{cor:smallball:2}.
\end{proof}

Our next corollary is the following analog of \cite[Eq. 3.40]{BCP}. \begin{theorem}\label{thm:smallball:inf} Let $\theta$ and $\eps<1/2$ be given constants. Assume that $y_{ij}$ are iid copies of a random variable $\xi$ of mean zero, variance one, and bounded $\E(|\xi|^{M_0})<\infty$ for sufficiently large $M_0$ (in terms of $\theta$ and $\ep$).  We have
$$\P\left ( \inf_{|t| \in \CG} \left |\frac{1}{\sqrt{n}} S_n(t,Y)\right | \le n^{-\theta+\eps/2}\right ) =O(n^{-\theta + 1 +\eps}).$$
\end{theorem}

\begin{proof} First of all, let $\CE_\Bb$ be the event that $|y_{ij}|\le n^{\eps}$ for all $i,j$. Then as $M_0$ is sufficiently large, by a union bound and by Markov's inequality, we have 
$$\P(\CE_\Bb^c) \le 2n \P(|\xi|> n^\eps) \le 2n \E(|\xi|^{M_0})/n^{\eps M_0} = O(n^{-\eps M_0 +1})=O(n^{-\eps M_0/2}).$$ 
Hence it suffices to condition on $\CE_\Bb$. Next, for any fixed $t$ we control the magnitude of 
$$d\left (\frac{1}{\sqrt{n}} S_n(t,Y)\right )/dt=(f_1(t,Y), f_2(t,Y))$$ 
where 
$$f_1(t,Y) =\frac{1}{\sqrt{n}} \sum_{i=1}^n y_{i1} \left (-\frac{i}{n} \sin\left (\frac{it}{n}\right )\right ) + y_{i2} \left (\frac{i}{n} \cos\left (\frac{it}{n}\right )\right ) $$
and 
$$
f_2(t,Y) =-\frac{1}{\sqrt{n}} \sum_{i=1}^n y_{i1} \left (\frac{i}{n}\right )^2 \cos\left (\frac{it}{n}\right ) + y_{i2} \left (\frac{i}{n}\right )^2 \sin\left (\frac{it}{n}\right ).$$ 
For this, again as $\xi$ has mean zero and variance one and  $\E(|\xi|^{M_0})<\infty$, a moment computation shows that as long as $|c_i|,|d_i|\le 1$ we have 
$$\E\left (\left |\frac{1}{\sqrt{n}}  \sum_{i=1}^n c_i y_{i1} + d_i y_{i2}\right |^{M_0}\right ) =O_{M_0}(1).$$

Therefore for any fixed $t$ we have
\begin{equation}\label{eqn:infinity}
\P\left (|f_1(t,Y)| \ge n^{\eps/2}\right ) =O\left (n^{-\eps M_0/2}\right ) \mbox{ and }  \P\left (|f_2(t,Y)| \ge n^{\eps/2}\right )=O\left (n^{- \eps M_0/2}\right ).
\end{equation} 
Notice that on $\CE_b$, we trivially have $\sup_{t\in [-n \pi, n \pi]}|f_i'(t,Y)|=O(n^{1/2+\ep})$. By a standard net argument that considers $[-n\pi, n\pi]$ as a union of $n^{2}$ equal intervals, we obtain from \eqref{eqn:infinity} and the union bound that
\begin{equation}\label{eqn:derivative}
\P\left (\sup_{t\in [-n \pi, n \pi]}\left \|d\left (\frac{1}{\sqrt{n}} S_n(t,Y)\right )/dt\right \|_2 \ge n^{\eps/2}\right ) =O(n^{-\eps M_0/4}).
\end{equation}
We will condition the complement of this event. Decompose $\CG$ into $O(n^{1+\theta})$ intervals of length $n^{-\theta}$ each, whose midpoints satisfy Condition \ref{cond:t}. For each such interval $I$, we estimate the probability that $\inf_{t \in I}|S_n(t,Y)| \le n^{-\theta}$. By \eqref{eqn:derivative}, this implies that for the midpoint $t_I$ we have 
$$\frac{1}{\sqrt{n}} S_n(t_I,Y) \le n^{-\theta} + n^{\eps/2} n^{-\theta} = O(n^{\eps/2 - \theta}).$$ 
However, by using Theorem \ref{thm:smallball:2}, we can control this event by
$$\P\left (\left |\frac{1}{\sqrt{n}} S_n(t_I,Y)\right | \le n^{-\theta+ \eps/2}\right ) =O(n^{-2\theta+ \eps}).$$
Taking union bounds over the midpoints of the $O(n^{1+\theta})$ intervals we obtain the bound $O(n^{-\theta+1 +\eps})$ as claimed, provided that $M_0$ is sufficiently large. 
\end{proof}

\section{Edgeworth expansion involving trigonometric functions}\label{section:Edgeworth} 
Our goal in this section is to  establish an Edgeworth expansion for several sums of random vectors that  arise from random trigonometric functions. The results are formulated under very mild assumptions on the coefficient distribution(s), which  hold in discrete settings  (such as the Rademacher distribution)  beyond the scope of the Cram\'er condition and known extensions \cite{AP}.

Let $s, t\in \R$ be given. Let $d=4$.  Consider the following sequence of  random vectors in $\R^d$ 
\begin{equation}\label{e.X_k}
X_{n,k} := C_n(k) Y_{k}, \ \ k=1,\dots, n,
\end{equation}
where (i)  $Y_k$'s are random vectors in $\R^2$ and their coordinates are iid with mean zero and variance one (we'll actually assume in our result that furthermore $\E \|Y_j\|_2^{\ell+d+1}<\infty$ for some $\ell\ge 4$),  and (ii) the deterministic $d\times 2$ matrices $C_n(k)$ are defined below. Recall from Subsection \ref{sub:smallball} that
\begin{align}\label{e.C_n}
C_{n}(k,t) &=\left (
\begin{matrix}
\cos(\frac{kt}{n}) & \sin (\frac{kt}{n})\\
-\frac{k}{n} \sin (\frac{kt}{n}) & \frac{k}{n} \cos(\frac{kt}{n}) \\ 
\end{matrix} \right )
\end{align}
and \begin{align}\label{e.C_n}
C_{n}(k) &=\left (
\begin{matrix}
C_n(k,t)\\
C_n(k,s) \\ 
\end{matrix} \right )
\end{align}
 is the $4 \times 2$ matrix obtained as the joint of $C_n(k,t)$ and $C_n(k,s)$. Recall also $S_{n}(s,t,Y)$ from \eqref{eqn:Sst}, and for short let
\begin{equation}\label{e.S_nYst}
S_n := S_n(s,t,Y)=:  X_{n,1}+\dots+X_{n,n}. 
\end{equation}
Let the average covariance matrix be
\begin{equation}\label{eqn:V_n}
V_n :=\frac{1}{n} \sum_{k=1}^n C_n(k) C_n(k)^\ast.
\end{equation}
This is  the same as the covariance for $S_n/\sqrt n$.  Let $\widetilde Q_n$ denote the distribution of $S_n/\sqrt n$, and  let $\widetilde Q_n(x)$ denote the cumulative distribution function for this distribution. 

The main result of this section, stated below, shows that $\widetilde Q_n$ is asymptotically $\widetilde{Q}_{n,\infty}$, where for $\ell\ge 2$ let
\begin{equation}\label{e.Q_nl}
\widetilde{Q}_{n,\ell} := \sum_{r=0}^{\ell-2} n^{-r/2} P_r(-\Phi_{0,V_n}, \{\overline{\chi}_\nu\}).
\end{equation}
and we will define the signed measure $P_r(-\Phi_{0,V_n}, \{\overline{\chi}_\nu\})$ below after fixing a few notations. For convenience, the density of $\widetilde Q_{n,\ell}$ is denoted by $Q_{n,\ell}$ while the density of $\widetilde Q_n$ is denoted by $Q_n$.

First, let $W$ be the standard Gaussian vector in $\R^d$, then for any covariance matrix $V$,  $V^{1/2}W$ will be the Gaussian random variable in $\R^d$ with mean zero and covariance $V$. Let $\phi_{0,V}$ denote the density of its distribution and let $\Phi_{0,V}$ denote the cumulative distribution function. If $V$ is the identity matrix then we simply write $\phi$ and $\Phi$, respectively.  Note that this is consistent with our definition of $\Phi$ at the beginning  Section~\ref{sub:E}. 

Secondly, recall that the cumulants of a random vector $X$ in $\R^d$ are the coefficients in the following (multiple) power series expansion 
\begin{equation}\label{e.cumulant}
\log \E [ e^{z \cdot X}] = \sum_{\nu\in \BN^d} \frac{\chi_{\nu} z^\nu}{\nu!}, \ \ z\in \C^d.
\end{equation}
Given that $X$ has mean zero, it is standard that the cumulant $\chi_{\nu}$ is bounded above by the $|\nu|$-th moment of $X$. In our situation, using independence of $X_{n,1},\dots, X_{n,n}$, it follows that the cumulants of $S_n$ are the sum of the corresponding cumulants of $X_{n,1},\dots, X_{n,n}$.  Let $\overline{\chi}_\nu := \chi_{\nu}(S_n)/n$, then $\overline{\chi}_\nu$ is also the average cumulant of $X_{n,1}$, \dots, $X_{n,n}$.  

Now, note that  cumulants  of $V_n^{1/2}W$ matches with the cumulants of $S_n/\sqrt n$ for any $|\nu|\le 2$, at the same time the higher order cumulants of $V_n^{1/2}W$ vanish thanks to symmetries of centered Gaussian. Therefore,
\begin{eqnarray*}
\log \E \left [ e^{z \cdot  (S_n/\sqrt n)}\right ]  
&=&   \log  \E [e^{z\cdot (V_n^{1/2}W)}]  +   \sum_{\nu\in \BN^d: |\nu|\ge 3}  (n\overline{\chi}_{\nu})  \frac{z^\nu}{\nu!}  n^{-|\nu|/2} \\ 
&=& \log  \E [e^{z\cdot V_n^{1/2}W}] +    \sum_{\ell\ge 1} \left (\sum_{\nu\in \BN^d: |\nu|=\ell+2}  \overline{\chi}_{\nu} \frac{z^\nu}{\nu!}\right )  n^{-\ell/2}.
\end{eqnarray*}
Letting $\overline\chi_\ell(z) := \ell!\sum_{\nu\in \BN^d: |\nu|=\ell}  \overline{\chi}_{\nu} z^\nu$ for all $z\in \C^d$, we obtain
\begin{eqnarray*}
\E \left [ e^{z \cdot (S_n/\sqrt n)}\right ]/ \E \left [ e^{z \cdot   V_n^{1/2}W}\right ] 
&=&  \exp\left [\sum_{\ell\ge 1}  \frac{\overline{\chi}_{\ell+2}(z)}{(\ell+2)!}   n^{-\ell/2}\right ] \\
&=&  \sum_{m\ge 0} \frac 1{m!} \left (\sum_{\ell\ge 1}  \frac{\overline{\chi}_{\ell+2}(z)}{(\ell+2)!}   n^{-\ell/2}\right )^m \quad = \quad\sum_{\ell\ge 0}  \widetilde P_\ell n^{-\ell/2},
\end{eqnarray*}
where $\widetilde P_\ell$ is obtained by grouping terms of the same order $n^{-\ell/2}$. 
It is clear  that $\widetilde P_{\ell}$ depends only on $z$ and the average cumulants $\overline{\chi}_{\nu}, |\nu|\le \ell+2$. We'll write $\widetilde P_\ell(z, \{\overline{\chi}_\nu\})$ to stress this dependence.  Replacing $z$ by $iz$, we obtain the following expansion for the characteristic function of $S_n/\sqrt n$:
\begin{eqnarray*}
\E \left [ e^{iz \cdot (S_n/\sqrt n)}\right ] 
&=&   \E\left  [ e^{iz \cdot   V_n^{1/2}W}\right ] \sum_{\ell\ge 0}  \widetilde P_\ell (iz, \{\overline{\chi}_\nu\})n^{-\ell/2}.
\end{eqnarray*}

Now, let $D=(D_1,\dots, D_n)$ be the partial derivative operator and let $\widetilde{P}_{\ell}(-D, \{\overline{\chi}_\nu\})$ be the differential operator obtained by formally replacing all occurences of $iz$ by $-D$ inside $\widetilde P_\ell (iz, \{\overline{\chi}_\nu\})$.  The signed measure $P_{\ell}(-\Phi_{0,V_n}, \{\overline{\chi}_\nu\})$ in the definition \eqref{e.Q_nl} of $\widetilde Q_{n,\ell}$ now can be defined: it has  the following density with respect to the Lebesgue measure:
$$P_{\ell}(-\phi_{0,V_n},  \{\overline{\chi}_\nu\})(x):=  \Big(\widetilde{P}_{\ell}(-D, \overline{\chi}_\nu)\phi_{0,V_n}\Big)(x).$$

For convenience of notation, for each $\ell>0$, let $\rho_l := \frac{1}{n} \sum \E \|X_i\|_2^l$ and 
$$M_\ell(f) := \sup_{x\in \R^d} \frac{|f(x)|}{1+\|x\|_2^\ell}$$
for any   measurable function $f$.

\begin{theorem}\label{thm:Edgeworth:0} Let $S_n$ be defined as above using \eqref{e.S_nYst} where we assume that the distribution of $Y_j$ satisfies $\E \|Y_{j}\|_2^{\ell+d+1} <\infty$ for some  $\ell \ge 4$. Let $f$ be measurable such that $M_{\ell}(f)<\infty$. 

Suppose that:
\begin{enumerate}
\item  all eigenvalues of $V_n$ are larger than a constant $\sigma>0$ independent of $n$;    
\vskip .05in
\item the parameters $s,t$ in the definition of $C_n(1), \dots, C_n(n)$  satisfy Condition \ref{cond:s,t} for some sufficiently small $\tau$.  
\end{enumerate}
Then the following estimate holds for  $\eps=n^{-C_\ast}$ where $C_{\ast}$ is any given positive constant:
\begin{eqnarray*}
&& \left |\int f(x) d\widetilde Q_{n} - \int f(x) d \widetilde{Q}_{n,\ell} \right |\\
 &\le& C M_{\ell}(f) (n^{-(\ell-1)/2} + e^{-n^{-\tau_\ast}} + e^{-cn} )+  \overline{\omega}_f\left (2\eps: \sum_{r=0}^{\ell+d-2} n^{-r/2} P_r(-\phi_{0,V_n}: \{\overline{\chi}_\nu\})\right )
 \end{eqnarray*}
where  
$$\overline{\omega}_f(\eps:\phi) = \int \left (\sup_{y\in B(x,\eps)} f(y) -\inf_{y\in B(x,\eps)} f(y)\right ) d\phi(x),$$ 
and the implied constant $C$ depends  on $\{\rho_k, k\le \ell\}$, $\sigma$, $C_*$,  and the implicit constants from Condition \ref{cond:s,t}, but not on $f$.
\end{theorem}

Notice that the verification of condition (1) in this theorem on the invertibility of $V_n$ follows from \cite[Appendix C]{BCP}.  


The general strategy of our proof follows the  approach in \cite{BR}, here we   focus on the main differences while trying to keep the exposition  self-contained.  Here our goal is not about proving the sharpest possible version for Theorem~\ref{thm:Edgeworth:0} in terms of   the number of bounded moments for $Y_i$, rather our aim is to present a simpler argument  (compared to \cite{BR}) at the sake of a more stringent moment assumption.

Before starting the proof, we  include some estimates  that will be useful in the proof.  

\begin{lemma} Let $l, c_0>0$ be any given constants. Assume that $\E\|X_{n,k}\|_2^{\ell+1}=O(1)$ uniformly over $n$ and $k=1,\dots, n$. Then for some sufficiently small $c_1>0$ the following holds for all $\|\eta\|_2 < c_1 n^{1/2}$ and all muti-index $\alpha$:
$$D^{\alpha}_{\eta}\Big(\E [ e^{i\eta \cdot (S_n/\sqrt n)}] (\E [ e^{i\eta \cdot   V_n^{1/2}W}])^{-1}  -  \sum_{r=0}^{\ell-2}  \widetilde P_r(i\eta, \{\overline{\chi_\nu}\}) n^{-r/2} \Big)$$
$$ \le C n^{-(\ell-1)/2}e^{c_0\|\eta\|_2^2} (\|\eta\|_2^{\ell+1 - |\alpha|} + \|\eta\|_2^{3\ell+1-|\alpha|}).$$
Here the implicit constant may depend on $c_0, \ell, \alpha$ and $\overline \chi_{0}, \dots, \overline{\chi}_{\ell+1}$
\end{lemma}

\proof For brevity we will write $\widetilde P_r(i\eta)$ as a shortcut of    $\widetilde P_\ell (i\eta, \{\overline{\chi}_\nu\})$.

Let
$$f_{\eta,\ell}(u):=\exp(g_{\eta,\ell}(u)) := \exp\left (\sum_{m=1}^{\ell-2} \frac{\overline{\chi}_{m+2}(i\eta)}{(m+2)!} u^m\right ).$$
We first show that for any multi-index $\alpha$
\begin{eqnarray}\label{e.fell-Qnell}
\left |D^{\alpha}_{\eta}\left (f_{\eta,\ell}\left (\frac 1{\sqrt n}\right )  - \sum_{r=0}^{\ell-2}  \widetilde P_r(i\eta)  n^{-r/2}\right ) \right |  
&\le& C n^{-(\ell-1)/2}(\|\eta\|_2^{\ell+1-|\alpha|} + \|\eta\|_2^{3\ell-3 -|\alpha|})e^{c_0\|\eta \|_2^2}
\end{eqnarray}
for all $\|\eta\|_2 \le c_1 n^{1/2}$  and $c_1>0$ is sufficiently small.  
 
Let $u\in \R$ (that may depend on $n$). As a function of $u\in \R$, the polynomial $\sum_{r=0}^{\ell-2}  \widetilde P_r(i\eta)  u^r$ is the Taylor approximation of degree $\ell-2$ for $f_{\eta,\ell}(u)$. Now, if $\|\eta\|_2 u \ll 1$ then 
$$|g_{\eta,\ell}(u)|=O(\|\eta\|_2^3 |u|)<c_0\|\eta\|_2^2$$ 
and similarly $|g_{\eta,\ell}^{(k)}(u)|=O(\|\eta\|_2^{k+2})$. Thus, using the chain rule and the generalized Leibniz rule, we may  bound
\begin{eqnarray*}
\left |f_{\eta,\ell}^{(\ell-1)}(u)\right |   &\le& Ce^{g_{\eta,\ell}(u)} \left (\sum_{ j_1+2j_2+\dots =\ell-1} \prod_{k\ge 1}  |g_{\eta,\ell}^{(k)}(u)|^{j_k}\right )  \\
&\le& C e^{g_{\eta,\ell}(u)} \left (\sum_{ j_1+2j_2+\dots =\ell-1}   \prod_{k\ge 1}\|\eta\|_2^{(2+k)j_k}\right ) \\
&=& O\left ((\|\eta\|_2^{\ell+1}+ \|\eta\|_2^{3\ell-3})e^{c_0\|\eta\|_2^2}\right ).
\end{eqnarray*}
(Here the implicit constant may depend on $c_0, \ell, \alpha$ and $\overline \chi_{0}, \dots, \overline{\chi}_{\ell+1}$.)

We obtain, assuming $\|\eta\|_2<c_1 |u|^{-1}$,
\begin{eqnarray*}
  \Big|f_{\eta,\ell}(u) - \sum_{r=0}^{\ell-2}  \widetilde P_r(i\eta)  u^{r} \Big|  
&\le& C |u|^{\ell-1}(\|\eta\|_2^{\ell+1}  + \|\eta\|_2^{3\ell-3})e^{c_0\|\eta \|_2^2}.
\end{eqnarray*}
We now let $u=\frac 1{\sqrt n}$. Using analytic dependence on $\eta$ of $f_{\eta,\ell}\left (\frac 1{\sqrt n}\right )- \sum_{r=0}^{\ell-2}  \widetilde P_r(i\eta)   n^{-r/2}$ and Cauchy's theorem for analytic functions, we obtain \eqref{e.fell-Qnell}.

Now, it remains to show that
\begin{eqnarray}
\left |D^{\alpha}_{\eta}\left (f_{\eta,\infty}\left (\frac 1{\sqrt n}\right )- f_{\eta,\ell}\left (\frac 1{\sqrt n}\right )\right ) \right |  
&\le& C n^{-(\ell-1)/2}(\|\eta\|_2^{ \ell+1-|\alpha|} + \|\eta\|_2^{3\ell+1 -|\alpha|})e^{2c_0\|\eta \|_2^2}.\nonumber
\end{eqnarray}
As before it suffices to show the case $\alpha=0$ of this estimate, and then the desired estimate follows from an application of Cauchy's theorem.
Now, since $|f_{\eta,\ell}\left (\frac 1{\sqrt n}\right )| \le e^{c_0\|\eta\|_2^2}$ as proved above,   it suffices to show that
\begin{eqnarray}\label{e.finfty-fell}
\left  |f_{\eta,\infty}\left (\frac 1{\sqrt n}\right )  f_{\eta,\ell}\left (\frac 1{\sqrt n}\right )^{-1} - 1\right | \le C n^{-(\ell-1)/2}(\|\eta\|_2^{\ell+1} + \|\eta\|_2^{3\ell+1})e^{c_0\|\eta\|_2^2}. 
\end{eqnarray}
Let $u\in (-2,2)$ and let  $h(u) :=g_{u\eta,\infty}\left (\frac 1{\sqrt n}\right ) -  g_{u\eta,\ell}\left (\frac 1{\sqrt n}\right )$.  It is clear that the first $\ell$ derivatives with respect to $u\in \R$  of $h$ all vanish at $u=0$. Thus, using the chain rule and the generalized Leibniz rule, it follows that  the first $\ell$ derivatives with respect to $u$ of $f_{u\eta,\infty}\left (\frac 1{\sqrt n}\right )f_{u\eta,\ell}\left (\frac 1{\sqrt n}\right )^{-1}-1$ also vanish at $u=0$. With $|u|=O(1)$, we obtain
\begin{eqnarray}
\nonumber \left |f_{u\eta,\infty}\left (\frac 1 {\sqrt n}\right )  f_{u\eta,\ell}\left (\frac 1 {\sqrt n}\right )^{-1} - 1\right | &\le& C \left |\left (\frac{d}{du}\right )^{\ell+1} \left (f_{u\eta,\infty}\left (\frac 1 {\sqrt n}\right )  f_{u\eta,\ell}\left (\frac 1 {\sqrt n}\right )^{-1} - 1\right ) \right | \\
&=& e^{h(u)} O\left (\sum_{j_1+2j_2+\dots = \ell+1} \prod_{k\ge 1}  |h^{(k)}(u)|^{j_k}\right ) \label{e.finfty-fell-der}
\end{eqnarray}
Now, as the first $\ell$ derivatives of $h$ all vanish at $u=0$, for any $k\le \ell+1$ we have (with $|u|=O(1)$)
$$|h^{(k)}(u) |  \le C \sup_{|t|\le |u|} |h^{(\ell+1)}(t) |   = C  \sup_{|t|\le |u|} \left |\left (\frac{d}{dt}\right )^{(\ell+1)} g_{t\eta, \infty}\left (\frac 1{\sqrt n}\right ) \right |.$$
By definition we have
$$g_{u\eta,\infty}\left (\frac 1{\sqrt n}\right ) =  - \log (\E e^{i(\eta \cdot V_n^{1/2}W)u}) + \sum_{j=1}^n \log (\E e^{i(\eta\cdot X_{n,j})  u/\sqrt n}) $$
Using $\E \|X_{n,j}\|_2 =O(1)$ we obtain 
$$|\E e^{i\eta\cdot X_{n,j} u/\sqrt n }-1| = O(\|\eta\|_2/\sqrt n),$$
therefore using the given assumption we obtain $|\E e^{i\eta\cdot X_{n,j}u/\sqrt n}|> 1/2$. Consequently, using the chain rule and the Leibniz rule we have
$$\left |\left (\frac{d}{du}\right )^{\ell+1} \log (\E e^{i\eta\cdot X_{n,j}u/\sqrt n})\right |  = O\left (\sum_{j_1+j_2+\dots = \ell+1} \prod_{k\ge 1}  \E [|n^{-1/2}\eta\cdot X_{n,j}|^{j_k}]\right ) =   O(n^{-(\ell+1)/2} \|\eta \|_2^{\ell+1}).$$
Since $\log (\E e^{i(\eta \cdot V_n^{1/2}W) u}) $ is  quadratic with respect to $u$  and $\ell\ge 4$, 
we obtain
$$\left |\left (\frac{d}{du}\right )^{\ell+1} g_{u\eta,\infty}\left (\frac 1{\sqrt n}\right )\right |  =    O(n^{-(\ell-1)/2} \|\eta \|_2^{\ell+1}).$$
Consequently,
$$|h^{(k)}(u)| \le C    n^{-(\ell-1)/2} \|\eta \|_2^{\ell+1} \le   Cc_1^{(\ell-1)/2} \|\eta\|_2^2,$$
in particular by choosing $c_1$ small we can ensure that $|h(u)|\le c_0 \|\eta\|_2^2$. Therefore, using \eqref{e.finfty-fell-der}, we obtain
\begin{eqnarray*}
\left |f_{u\eta,\infty}\left (\frac 1{\sqrt n}\right )  f_{u\eta,\ell}\left (\frac 1{\sqrt n}\right )^{-1} - 1\right |  
&\le& C  e^{c_0\|\eta\|_2^2} \sum_{j_1+2j_2+\dots = \ell+1} \prod_{k\ge 1} ( n^{-(\ell-1)/2} \|\eta \|_2^{\ell+1})^{j_k} \\
&\le&  C e^{c_0\|\eta\|_2^2} (n^{-(\ell-1)/2} \|\eta\|_2^{\ell+1} + (n^{-(\ell-1)/2} \|\eta\|_2^{\ell+1})^{\ell+1}) \\
&\le&  C e^{c_0\|\eta\|_2^2} n^{-(\ell-1)/2} \|\eta\|_2^{\ell+1}(1+\|\eta\|_2^{2\ell})
\end{eqnarray*}
We then set $u=1$ to obtain the desired estimate.
\endproof

As a corollary, we obtain
\begin{corollary} \label{c.hatHn} Assume that $\E|X_{n,k}|^{\ell+1}=O(1)$ uniformly over $n$ and $k=1,\dots, n$. Assume that the eigenvalues of $V_n$ are bounded below by some positive constant independent of $n$. Then for some sufficiently small constants $c_0, c_1>0$, the following holds for all $\|\eta\|_2 < c_1 n^{1/2}$ and all muti-index $\alpha$:
$$D^{\alpha}_{\eta}\Big(\E [ e^{i\eta \cdot (S_n/\sqrt n)}]     - \E [ e^{i\eta \cdot   V_n^{1/2}W}] \sum_{r=0}^{\ell-2}  \widetilde P_r(i\eta) n^{-r/2} \Big)$$
$$ \le C n^{-(\ell-1)/2}e^{-c_0\|\eta\|_2^2} (\|\eta\|_2^{\ell+1 - |\alpha|} + \|\eta\|_2^{3\ell+1+|\alpha|}).$$
\end{corollary}
This corollary follows from the fact that $\E [ e^{i\eta \cdot   V_n^{1/2}W}]$ is $e^{-c\left<\eta, V_n^{-1}\eta\right>}$ for some $c>0$, so with $c_0>0$ sufficiently small one has
$$|D^{\alpha}_{\eta} \E [ e^{i\eta \cdot   V_n^{1/2}W}]| \le e^{-2c_0\|\eta\|_2^2}(\|\eta\|_2^{|\alpha|} + 1),$$
and combining these estimates with the Leibniz rule we obtain the desired conclusion.

\begin{proof} [Proof of Theorem \ref{thm:Edgeworth:0}] We now begin the proof of the main estimate.
Let
\begin{equation}\label{eqn:eps}
\eps=\eps_n = n^{-C_\ast}.
\end{equation}
For convenience, denote
$$\widetilde H_n = \widetilde Q_n - \widetilde Q_{n,\ell},$$
and let $H_n$ be its density. As usual the characteristic function of $H_n$ is   $\widehat {H_n}(\eta) =  \int_{\R^d} e^{i t \cdot \eta} \widetilde H_n(dt)$.

Let $\widetilde K$ be a probability measure supported inside the unit ball $B(0,1)=\{x\in \R^d: \|x\| \le 1\}$ (whose density is denoted by $K$) such that its characteristic function $\widehat K(\eta)$ satisfies
\begin{equation}\label{e.Keps}
|D^\alpha \widehat K(\eta)|  = O( e^{-\|\eta\|_2^{1/2}}), \quad |\alpha| \le \ell+d+1.
\end{equation}
Such a measure could be constructed using elementary arguments, see for instance \cite[Section 10]{BR}. We then let $\widetilde K_\epsilon$ be the $\epsilon$-dilation of $K$, namely $\widetilde K_\epsilon(A) = \widetilde K(\epsilon^{-1}A)$ and $\epsilon^{-1}A := \{x/\epsilon: x\in A\}$ 
for all measurable $A$. Note that $\widetilde K_\epsilon$ is a probability measure on $B(0, \epsilon)$ and it satisfies the dilated version of  \eqref{e.Keps}.

We will be using the following simple identity: for any two   measures $\mu_1$ and $\mu_2$ of bounded variation, $|\mu_1|(\R^d), |\mu_2|(\R^d)<\infty$, and any bounded  $f$ , it holds that
\begin{equation}\label{eq:convolution:id}
\int\int f(x+y)d\mu_1(x) d\mu_2(y) = \int f(t) (d\mu_1*d\mu_2)(t).
\end{equation}
Now, for each $x \in B(0,\epsilon)$ we have $f(y)\le \sup_{z\in B(0,\epsilon)} f(x+y+z)$, therefore using nonnegativity of $d\widetilde Q_n$ we obtain
\begin{eqnarray*}
\int f(y) d\widetilde H_n(y) &=& \int_{x\in B(0,\epsilon)}  \int f(y) d\widetilde H_n(y) d\widetilde K_\epsilon(x) \\
&\le& \int  \Big(\int \sup_{z\in B(0,\epsilon)} f(x+y+z) d\widetilde Q_n(y) - \int f(y) d\widetilde Q_{n,\ell}(y) \Big)d\widetilde K_\epsilon(x) \\
&=& \int \int \sup_{z\in B(0,\epsilon)} f(x+y+z) d \widetilde H_n(y) d \widetilde K_\epsilon(x)  + \\
&& + \int \int \left (\sup_{z\in B(0,\epsilon)} f(x+y+z)  - f(y) \right ) d\widetilde Q_{n,\ell}(y) d \widetilde K_\epsilon(x).
\end{eqnarray*}
Thus, by \eqref{eq:convolution:id},
\begin{eqnarray*}
	\int f(y) d\widetilde H_n(y)&\le& \int \sup_{z\in B(0,\epsilon)} f(t+z) (H_n*K_\epsilon)(t)dt + \int \left (\sup_{z\in B(0,\epsilon)} f(x+y+z)  - f(y) \right ) d\widetilde Q_{n,\ell}(y) d \widetilde K_\epsilon(x) \\
&\le& M_{\ell}(f) \int (1+\|t\|_2+\epsilon)^{\ell} |H_n*K_\epsilon|(t)dt + \int \int (\sup_{B(y,2\epsilon)} f(t) - \inf_{B(y,2\epsilon)} f(t) ) |d\widetilde Q_{n,\ell}|(y)d\widetilde K_\epsilon(x)\\
&\le& C_\ell M_{\ell}(f) \int (1+\|t\|_2)^{\ell} |H_n*K_\epsilon|(t)dt  +  \bar{\omega}_f(2\eps: |\widetilde Q_{n,\ell}|).
\end{eqnarray*}
By applying the above estimate for $-f$ in place of $f$, it follows immediately that $|\int f d \widetilde H_n| $ is bounded above by the same right hand side.  
By  standard Sobolev embedding estimates for Fourier transforms, we have 
\begin{eqnarray*} 
\int (1+\|t\|_2)^{\ell} |H_n \ast K_\eps|(t) dt
&=& O \Big(\max_{0\le |\alpha|\le d+\ell+1} \int |D^\alpha \widehat{H_n\ast K_\eps}(\eta)|d\eta\Big)\\
&=& O  \left (\max \left \{ \int |D^\alpha (\widehat {H_n})(\eta) D^{\beta} (\widehat {K_\epsilon})(\eta)| d\eta: \ \ |\alpha|+|\beta|\le \ell + d + 1\right \}\right ).
\end{eqnarray*}
Using \eqref{e.Keps} we have  $D^{\alpha} \widehat {K_\epsilon}(\eta)=O(1)$ for all $|\alpha|\le \ell + d+ 1$. While this estimate is fairly generous, it is good enough to control the contribution of small $\eta$ in the integrals. More specifically, let $B_n^2 = V_n^{-1}$, then by the given assumption the eigenvalues of $B_n$ are $O(1)$, so $\E \|B_n X_{n,k}\|_2^{\ell+d+1}=O(1)$.  For some $c_1>0$ sufficiently small, using Corolary~\ref{c.hatHn}, we obtain
\begin{eqnarray*}
\int_{\|\eta\|_2\le c_1 \sqrt n}|D^\alpha \widehat {H_n}(\eta) D^{\beta} \widehat {K_\epsilon}(\eta)| d\eta 
&=& O\left (\int_{\|\eta\|_2\le c_1 \sqrt n}|D^\alpha \widehat {H_n}(\eta) | d\eta\right )\\
&=&  O ( n^{-(\ell+d-1)/2}).
\end{eqnarray*}
We now consider the range $\|\eta\|_2\ge c_1 \sqrt n$. We estimate
\begin{align*}
\int_{\|\eta\|_2\ge c_1 \sqrt n} |D^{\alpha} \hat{H}_n(t) D^{\beta} \hat{K}_\eps |d \eta  &\le \int_{\|\eta\|_2 \ge c_1 \sqrt n} |D^{\alpha} \hat{Q}_n(t) D^{\beta} \hat{K}_\eps |d\eta \\
&+\int_{\|\eta\|_2 \ge c_1\sqrt  n} \left |D^{\alpha} \left ( \sum_{r=0}^{\ell-2+d} n^{-r/2} P_r(i\eta: \{\chi_{\nu,n}\})\right ) \exp(-1/2\langle \eta,B_n \eta \rangle) \right |d \eta, 
\end{align*}
and it is clear that the second term can be controlled by $O(e^{-cn})$  thanks to the Gaussian decay of $\exp(-1/2\langle \eta,B_n \eta \rangle)$.

Let $\phi_i (\eta) = \E e^{ i \eta \cdot X_i}$. Then for $|\alpha|\le \ell+d+1$ we have  $D^\alpha_{\eta}(\phi_i(\eta/\sqrt n)) = n^{-|\alpha|/2} O(\E\|X_{n,i}\|_2^{|\alpha|}) = O(1)$. Thus,
$$ |D^{\al} \widehat{Q}_n(\eta)|  = \left |D^{\alpha}\left (\prod_{i=1}^n \phi_i(\frac{\eta}{\sqrt n})\right )\right | = O\left (\sum_{\gamma_1+\dots+\gamma_n  = \al }   \left |\prod_{i=1, \gamma_i =0}^n \phi_i\left (\frac{\eta}{\sqrt{n}}\right )\right |\right ),$$
while we also have 
$$|D^\beta \hat{K}_{\eps}(\eta)| = O(\eps^{|\beta|} e^{-(\eps \|\eta\|_2)^{1/2}}) = O(e^{-(\eps \|\eta\|_2)^{1/2}}).$$
 Thus,  it remains to control, for each $(\gamma_1,\dots,\gamma_n)$ with $|\gamma_1|+\dots +|\gamma_n| \le \ell+d+1$ and each $r>0$ independent of $n$:
\begin{align*}
J_\gamma(n,\eps) &= \int_{\|\eta\|_2 \ge r\sqrt n}  \left |\prod_{i=1, \gamma_i =0}^n \phi_i(\frac{\eta}{\sqrt n})\right |  e^{-(\eps  \|\eta\|_2)^{1/2}}  d\eta\\
&= n^{d/2} \int_{\|\eta\|_2\ge r} \left |\prod_{i=1, \gamma_i =0}^n \phi_i(\eta)\right | e^{-(\eps \sqrt{n}\|\eta\|_2)^{1/2}} d\eta\\
&= n^{d/2} \int_{\|\eta\|_2\ge r} \left |\prod_{i=1, \gamma_i =0}^n \phi_i(\eta)\right | e^{-(n^{-C_\ast +1/2}\|\eta\|_2)^{1/2}} d\eta.
\end{align*}

Clearly it suffices to consider $r \le \|\eta\|_2 \le n^{C_\ast-1/2 + \tau}$ because the integral for $\|\eta\|_2 \ge n^{C_\ast-1/2 + \tau}$ is extremely small. Again, because $\al$ is fixed, by throwing away from the set $\{\Bv_i\}$ a fixed number of elements, let us assume that $\al=0$ for simplicity \footnote{In the general case $\al \neq 0$ we use Theorem \ref{thm:fourier'} instead of Theorem \ref{thm:fourier}.}.  
By Theorem \ref{thm:fourier} for sufficiently large $n$ we have
 $$|\prod_i \phi_i(  \eta)| \le e^{-n^{-\tau_\ast}}.$$
 Thus we just shown that, with $\eps= n^{-C_\ast}$,
  $$J_\gamma(n,\eps)= O(e^{-n^{-\tau_\ast}}).$$
  Putting the bounds together, we obtain the desired estimate:
  $$\left |\int f d\widetilde H_n\right | \le C M_{\ell}(f) (n^{-(\ell-1)/2} + e^{-n^{-\tau_\ast}} + e^{-cn} )+  \bar{\omega}_f\left (2\eps: \sum_{r=0}^{\ell-2+d} n^{-r/2} P_r(-\phi_{0,V_n}: \{\bar{\chi}_\nu\}\right ).$$
 \end{proof}

\subsection{A useful corollary}  Below we consider a consequence  of  Theorem \ref{thm:Edgeworth:0}  that will be convenient for our proof of Theorem~\ref{thm:var} in subsequent sections.

 With $Y_k=(y_{k1},y_{k2})$ where $y_{ij}$ are iid  with  mean zero and variance one, we recall the definition of $P_n(t,Y)$   from \eqref{eqn:Pn}. Let $G_k=(g_{k1},g_{k2})$ where $g_{ij}$ are iid standard Gaussian. 
Recall the definition of $S_n(Y,s,t)$ from \eqref{e.S_nYst} and let $S_n(G,s,t)$ be its Gaussian analogue.

Clearly $\widetilde P_0 = 1$ and by explicit computation we have
\begin{eqnarray}\label{e.P1P2}
\widetilde{P}_1(z, \{\overline{\chi}_\nu\}) = \sum_{|\nu|=3} \frac{\overline{\chi}_\nu}{\nu!} z^\nu, \qquad
\widetilde{P}_2(z, \{\overline{\chi}_\nu\}) = \frac{\overline{\chi}_4(z)}{24} +  \frac{\overline{\chi}_3^2(z)}{72}.
\end{eqnarray}

For convenience of notation, let $e_j = (\dots, 0, 1, 0,\dots) \in \R^d$ where $1$ is in the $j$th coordinate. Using \eqref{e.P1P2} we obtain
$$P_1(-\phi, \{\overline{\chi}_\nu\})  \quad = \quad \sum_{|\nu|=3} \frac{\overline{\chi}_\nu}{\nu!} (-D)^\nu\phi(x) \quad =\quad $$
\begin{eqnarray*}
&=& \Big[ \frac 1 6 \sum_{j=1}^4 \overline{\chi}_{3e_j} (x_j^3-3x_j) + \frac 1 2 \sum_{i\ne j} \overline{\chi}_{2e_i+e_j} (x_i^2 x_j - x_j) +  \sum_{i< j < k} \overline{\chi}_{e_i+ e_j+ e_k} x_i x_j x_k\Big]\phi(x)\\
&=& \Big[\frac 1 6 \sum_{j=1}^4 \overline{\chi}_{3e_j} h_3(x_j)  + \frac 1 2 \sum_{i\ne j} \overline{\chi}_{2e_i+e_j} h_2(x_i)h_1(x_j)  +  \sum_{i, j, k} \overline{\chi}_{e_i+ e_j+ e_k} h_1(x_i)h_1(x_j) h_1 (x_k)\Big]\phi(x),
\end{eqnarray*}
where $h_k(x)=(-1)^k e^{x^2/2} \frac{\partial^k}{\partial x^k} e^{-x^2/2} (k=0,1,2,\dots)$ are the (one dimensional) Hermite polynomials.

Now for any multi-index $\al=(\al_1,\dots,\al_{\ell}) \in \{1,\dots,d\}^{\ell}$, we let $|\alpha|=\ell$ and  let $n_j(\al) = |\{i: \al_i =j\}|$ for each $j=1,\dots, d$. We then define
\begin{equation}\label{H}
H_\al(x_1,\dots, x_d):= \prod h_{n_1}(x_1) \dots h_{n_d}(x_d).
\end{equation}
For a random vector $Z=(Z_1,\dots,Z_d)$ as usual let  $Z^\al = \prod_{j=1}^d Z_j^{\al_j}$.
With $X=(X_{n,1},\dots,X_{n,n} )$ define 
\begin{eqnarray}
\label{eqn:Delta}
\Delta_\al(X_{n,k}) &=& \E X_{n,k}^\al - \E G_{n,k}^\al,\\ 
\label{eq:def:c:n:alpha}
c_n(\al, X) &:= & \frac{1}{n} \sum_{k=1}^n \Delta_\al (X_{n,k})\\
\label{eqn:Gamma1}
\Gamma_{n,1}(X,x) &:=&  \frac{1}{6} \sum_{|\al|=3} c_n(\al,X) H_\al(x).
\end{eqnarray}
Note that if $\alpha'$ is a permutation of $\alpha$ then $H_{\alpha'}=H_{\alpha}$. Furthermore using \eqref{e.cumulant}  and explicit computations it follows that $\chi_{\nu}(X) = \E [X^\nu]$ for all $|\nu|=2,3$  if $X$ is a random vector in $\R^d$ with mean $\E X=0$. Thus,  for all distinct $i,j,k$,
 $$\overline{\chi}_{3e_j}=c_n((j,j,j),X), \quad \overline{\chi}_{2e_i+e_j}  = 0 = c_n((i,i,j), X), \qquad \chi_{e_i+e_j+e_k}  = 0 = c_n((i,j,k),X).$$
 Using these observations, we obtain
$$P_1(-\phi_{0,V_n},  \{\overline{\chi}_\nu\}) = \Gamma_{n,1} (X,x) \phi_{0,V_n}(x).$$


We also define
\begin{equation}\label{eqn:Gamma:2}
\Gamma_{n,2}(X,x) = \Gamma_{n,2}' + \Gamma_{n,2}''
\end{equation}
where 
$$\Gamma_{n,2}'(X,x) = \frac{1}{24} \sum_{|\beta|=4} c_n(\beta,X) H_\beta(x)$$
and 
$$\Gamma_{n,2}''(X,x) = \frac{1}{72} \sum_{|\rho|=3}\sum_{|\beta|=3} c_n(\beta,X) c_n(\rho,X)H_{\beta,\rho}(x).$$
Via explicit computations, it can also be checked that
$$P_2(-\phi_{0,V_n}, \{\overline{\chi}_\nu\}) = \Gamma_{n,2}(X,x) \phi_{0,V_n}(x).$$
Finally, recall the definition of $\widetilde Q_{n,2}$ from \eqref{e.Q_nl}, which has density
$$Q_{n,2}(X,x) = 1+n^{-1/2} P_1(-\Phi_{0,V_n}, \{\overline{\chi}_\nu\})+n^{-1} P_2(-\Phi_{0,V_n}, \{\overline{\chi}_\nu\}) .$$

It follows that
\begin{fact}\label{fact:identical}  
$$Q_{n,2}(X,x)  = 1 + \frac{1}{\sqrt{n}} \Gamma_{n,1}(X, x) + \frac{1}{n} \Gamma_{n,2}(B_nX,x).$$
\end{fact}

Now by applying Theorem \ref{thm:Edgeworth:0} and then swallow higher order terms  in the Edgeworth expansion into the error terms (resulting into $O(n^{-3/2})$, keeping the first three terms), we obtain the following corollary.

\begin{theorem}\label{thm:Edgeworth:alternating}  With the same assumption as in Theorem  \ref{thm:Edgeworth:0} the following holds for $\eps=n^{-C_*}$ (and $C_*$ is any given positive constant):
\begin{eqnarray*}
&&|\E (f(S_n(s,t,Y))) -  \E (f(V_n^{1/2}W) Q_{n,2}(X,W)) |\\
 &\le& C n^{-3/2}+ C M_{\ell}(f) (n^{-(\ell-1)/2} + e^{-n^{-\tau_\ast}} + e^{-cn} )+  \overline{\omega}_f\left (2\eps: \sum_{r=0}^{\ell+d-2} n^{-r/2} P_r(-\phi_{0,V_n}: \{\overline{\chi}_\nu\})\right ).
\end{eqnarray*}
where $W$ is the standard Gaussian vector in $\R^d$.  
\end{theorem}

\section{Proof of Proposition \ref{prop:Edgeworth:trig:delta} : asymptotic Kac-Rice formula}\label{section:Edgeworth:trig:delta}

We will show the following more precise statement.
\begin{prop}\label{prop:Edgeworth:trig:delta'} Let $\ell$ be a fixed positive integer. Let $\delta$ be as in \eqref{eqn:delta}. Assume that $\eta$ has  mean zero and variance one and $\E |\eta|^{M_0}<\infty$ for sufficiently large $M_0$. Assume that $s,t$ satisfy Condition \ref{cond:s,t}. Then for any $\eps=n^{-C_*}$ (where $C_*>0$ is any absolute constant),  we have 

\begin{equation}\label{eqn:P}
|\E F_\delta(P_n(t,Y))- \E F_\delta(P_n(t,G))| \le \frac{C}{n^{1/2}} +  C\delta^{-1} (n^{-(\ell-1)/2} + e^{-n^{\tau_\ast}} + e^{-cn} +\eps)
\end{equation}
and
\begin{align}\label{eqn:t,Y}
\Big|  \E \Phi_\delta\left (\frac{1}{\sqrt{n}} S_n(t,Y)\right )  &- \E \Phi_\delta\left (\frac{1}{\sqrt{n}} S_n(t,G)\right ) - \frac{1}{n} \E\big[\Phi_\delta(I_2(\la)^{1/2} W_2) \Gamma_{n,2} \big(I_2(\la)^{-1/2} X_n(t,Y), W_2)\big)\big] \Big| \nonumber \\
&\le  \frac{C}{n^{3/2}}+  \frac{1}{n} r_n(t, \Phi_\delta) +    C\delta^{-1} (n^{-(\ell-1)/2} + e^{-n^{\tau_\ast}} + e^{-cn} +\eps)
\end{align}
and
\begin{align}\label{eqn:s,t,Y}
\Big|  \E \Psi_\delta \left(\frac{1}{\sqrt{n}} S_n(s,t,Y)\right)&- \E \Psi_\delta\left  (\frac{1}{\sqrt{n}} S_n(s,t,G)\right ) - \frac{1}{n} \E \big [\Psi_\delta (I_4(\la)^{1/2} W_4) \Gamma_{n,2} \big(I_4(\la)^{-1/2} X_n(s,t,Y),W_4\big )\big]\Big| \nonumber \\
& \le  \frac{C}{n^{3/2}} + \frac{1}{n} r_n(s,t,\Psi_\delta) +    C\delta^{-2} (n^{-(\ell-1)/2} + e^{-n^{\tau_\ast}} + e^{-cn} +\eps),
\end{align}
where $I_2$ and $I_4$ are any invertible diagonal matrices and $W, W_2,W_4$ are standard Gaussian vectors in $\R, \R^2$ and $\R^4$ respectively, and where the implied constants are allowed to depend on the $M_0$-moment of $\eta$, on the constants in Conditions \ref{cond:t} and \ref{cond:s,t}, and on a lower bound of the least singular values of $V_n(t), V_n(s,t)$ and $I_2,I_4$. Furthermore, we have the following bounds 
$$r_n(t, \Phi_\delta)=O(\| V_n(t) - I_2\|_{2})  \mbox{ and }  r_n(s,t,\Psi_\delta) = O(\| V_n(s,t) - I_4\|_{2}).$$
\end{prop}

Note that if we apply the above theorem for $\eps= n^{-(\ell-1)/2}$ and for sufficiently large $\ell$ (for instance $\ell =16$ would suffice), then all the error bounds are absorbed into $O(\frac{1}{n^{3/2}})$, and hence proving Proposition \ref{prop:Edgeworth:trig:delta}.

We now discuss the proof. We first note that if $f$ is an even function then using the fact that the standard Gaussian distribution is symmetric and the fact that Hermite polynomials of odd degrees are odd functions we obtain
$$\E [f(V_n^{1/2}W)\Gamma_{n,1}(V_n^{-1/2}X,W)] = 0.$$
In our applications below the functions $f$ are indeed even therefore we could ignore the contribution of $\Gamma_{n,1}$ in the estimates.

Now, recall \eqref{e.S_nYst} and recall that $\delta= n^{-5}$ and $\eps=n^{-C_*}$ for some given constant $C_*>0$.
Recall also the definitions of the  even functions $F_\delta:\R\to\R_+$, $\Phi_\delta : \R^2 \to \R_+$, $\Psi_\delta:\R^4\to \R_+$ from Subsection~\ref{sub:KR}.
Now, using standard integration by parts (for details see  \cite[Eq. 3.23]{BCP}) we may rewrite the Gaussian part in a more canonical form: with $f$ being either $F_\delta, \Psi_\delta$ or $\Phi_\delta$, and with $W$ being either $W_2$ or $W_4$ we have
\begin{equation}\label{eqn:VtoI}
\E (f(V_n^{1/2} W)\Gamma_{n,2}(V_n^{-1/2}X,W)) = \E (f(I_d^{1/2}W)\Gamma_{n,2}(I_d^{-1/2} X,W)) + r_n(f)
\end{equation}
with $I_d$ being a diagonal matrix with diagonal entries  at least $\sigma$ and here
$$|r_n(f)| \le C  \|V_n -I_d\|_2.$$
For this proof, we will only work with $\Psi_\delta$ and prove \eqref{eqn:s,t,Y} as \eqref{eqn:P} and \eqref{eqn:t,Y} are similar and simpler. By Fact \ref{fact:identical} and by \eqref{eqn:VtoI}, to prove Proposition \ref{prop:Edgeworth:trig:delta'} for this $f=\Psi_\delta$ it suffices to show
\begin{equation}\label{eqn:EW:final}
\left |\E  f\left (\frac{1}{\sqrt{n}}S_n(s,t,Y)\right )  -  \E \left (f (V_n^{1/2}W) Q_{n,2}(X,W)\right )\right | \le C n^{-3/2} + C  \delta^{-2} (n^{-(\ell-1)/2} + e^{-n^{-\tau_\ast}} + e^{-cn} + \eps).
\end{equation}

\begin{proof}[Proof of Proposition \ref{prop:Edgeworth:trig:delta'}] Let $\lambda \in (0,1)$ and let $\varphi_\lambda: \R\to [0,1]$ be a $C^\infty(\R)$ function with support inside $[-\delta, \delta]$ such that 

(i)  $\varphi_\lambda(x) = \delta^{-1}$ for $|x|\le \delta(1-\lambda)$.

(ii) $|\varphi_\lambda^{(k)} (x)|  = O(\delta^{-(k+1)}\lambda^{-k})$ for any $k\ge 0$.

  Let  $\widetilde \varphi_\lambda: \R^4\to \R$ be defined by $\widetilde \varphi(x) = \varphi_\lambda(x_1)\varphi_\lambda(x_3)$. Let 
  $$f_\lambda(x_1,x_2,x_3, x_4)= |x_2|      |x_4|  \widetilde \varphi(x).$$ 
 Then $f_\lambda$ is locally Lipschitz, and its derivative (defined almost everywhere) satisfies 
 $$|\nabla f_\lambda|  \le  C\frac 1 {\delta^2 \lambda} (1+|x|)^4.$$
Recall that $\bar{\omega}_f(\eps:\phi) = \int (\sup_{y\in B(x,\eps)} f(y) -\inf_{y\in B(x,\eps)} f(y)) \phi(x)dx$, and $\phi$ is the density of a Gaussian vector.   Consequently, for any polynomial $p(x)$ with bounded degree and bounded coefficients we have
\begin{equation*} 
\bar{\omega}_{f_\lambda} (\eps:  p(x)\phi_{0,V_n}(x)) = O (\lambda^{-1}\delta^{-2}  \eps). 
\end{equation*}
Here we are implicitly using the fact that the  the eigenvalues of $V_n$ are bounded above by $O(1)$, which should follow from the fact that the singular values of $C_n(k)$ are bounded: they are bounded by the Hilbert--Schmidt norm, which is bounded since  the entries of $C_n(k)$ are bounded.

Note that  one could write
$$\sum_{r=0}^{\ell+d-2} n^{-r/2} P_r(-\phi_{0,V_n}: \{\overline{\chi}_\nu\})=p(x)\phi_{0,V_n}(x)$$
for some polynomial $p$ with degree at most $d+\ell$ and coefficients bounded by the first $d+\ell$ moments of the random coefficients $Y_1,\dots, Y_n$ of $P_n$. Therefore
\begin{equation}\label{e.modcont}
\bar{\omega}_{f_\lambda} (2\eps:  \sum_{r=0}^{\ell+d-2} n^{-r/2} P_r(-\phi_{0,V_n}: \{\overline{\chi}_\nu\})) = O (\lambda^{-1}\delta^{-2}  \eps). 
\end{equation}


We will also use the following elementary estimate: given any $a_1,\dots, a_n$ deterministic and $\eta_1,\dots, \eta_n$ independent with mean $0$ and bounded $4$th moment, the following holds
$$\E |a_1\eta_1+\dots + a_n \eta_n|^4 \le C (a_1^2+\dots + a_n^2)^2.$$
Indeed, thanks to independent and the mean zero property, the left hand side is
$$\E (\sum_{i,j,k,l} a_i a_j a_k a_l \eta_i \eta_j\eta_k\eta_l) = \sum_{i} a_i^4 \E \eta_i^4 +   O(\sum_{i<j} a_i^2 a_j^2 \E \eta_i^2 \eta_j^2) $$
$$=O(a_1^2+\dots + a_n^2)^2.$$
In particular, 
$$\E |P'_n(t,Y)|^4 = n^{-2} O((\sum_{j=1}^n \sin^2 (jt/n)  + \cos^2(jt/n))^2) = O(1).$$
We next proceed to conclude Proposition \ref{prop:Edgeworth:trig:delta'} for $f=\Psi_\delta$. Recall that $\phi_{0, V_n}$ denotes the density of $V_n^{1/2}W$. 
 Then using H\"older's inequality and Theorem~\ref{thm:smallball:2} we obtain
\begin{eqnarray*}
 &&  \E  (f-f_\lambda )\left (\frac{1}{\sqrt{n}}S_n(s,t,Y)\right )  \\
&\le& C\delta^{-2}\E [|P'_n(t,Y) P'_n(s,Y)| 1_{||P_n(t,Y)|-\delta|\le \lambda\delta} 1_{||P_n(s,Y)|-\delta|\le \lambda \delta}] \\
&\le& C \delta^{-2} (\E |P'_n(t,Y)|^4)^{1/4} (\E |P'_n(s,Y)|^4)^{1/4} \P(||P_n(t,Y)|-\delta|\le \delta \lambda)^{1/4} \P(||P_n(s,Y)|-\delta|\le \delta \lambda)^{1/4}\\
&\le& C  \delta^{-2}  (n\delta \lambda)^{2/5} \quad \text{(assume bounded $4$th moment for coefficients of $P_n$ and Corollary~\ref{cor:smallball:2})}\\
&\le& C  \delta^{-2}   \lambda^{2/5}.
\end{eqnarray*}
 
 By Theorem \ref{thm:Edgeworth:alternating} and \eqref{e.modcont}, applying for $\eps^{7/2}$ in place of $\eps$, we obtain
\begin{align*}
&\left |\E f_\lambda\left (\frac{1}{\sqrt{n}}S_n(s,t,Y)\right )  -  \E (f_\la (V_n^{1/2}W) Q_{n,2}(X,W))\right | \\
 &=  O(n^{-3/2}) + M_{\ell}(f_\la) O\Big(n^{-(\ell-1)/2} + e^{-n^{-\tau_\ast}} + e^{-cn} \Big) + \bar{\omega}_{f_\lambda} (2\eps^{7/2}:  \sum_{r=0}^{\ell+d-2} n^{-r/2} P_r(-\phi_{0,V_n}: \{\overline{\chi}_\nu\})) \\
 &=    O(n^{-3/2}) + \delta^{-2} O (n^{-(\ell-1)/2} + e^{-n^{-\tau_\ast}} + e^{-cn} ) + O (\delta^{-2}\la^{-1} \eps^{7/2} ),
 \end{align*}
where we note that $M_{\ell}(f_\la)=O(\delta^{-2})$. Note that as a special case, this bound also holds  for the Gaussian case. Consequently,
\begin{eqnarray*} 
&& \left |\E  f\left (\frac{1}{\sqrt{n}}S_n(s,t,Y)\right )  -  \E (f (V_n^{1/2}W) Q_{n,2}(X,W))\right | \\
&\le&    \left |\E  (f-f_\lambda )\left (\frac{1}{\sqrt{n}}S_n(s,t,Y)\right )\right |  +  |\E  (f-f_\lambda )((V_n^{1/2}W) Q_{n,2}(X,W))| +  \\
&& +   \quad \left |\E f_\lambda\left (\frac{1}{\sqrt{n}}S_n(s,t,Y)\right ) - \E (f_\la (V_n^{1/2}W) Q_{n,2}(X,W))\right |    \\
&\le&  C n^{-3/2} + C  \delta^{-2} (n^{-(\ell-1)/2} + e^{-n^{-\tau_\ast}} + e^{-cn} + \lambda^{-1}\eps^{7/2}+ \lambda ^{2/5})
\end{eqnarray*}
We then take $\lambda=\eps^{5/2}$ and obtain the desired estimate as in \eqref{eqn:EW:final}.
\endproof
\end{proof}

\section{Completing the proof of Proposition \ref{prop:v_n}}\label{section:proof:prop:v_n}

Recall the definition of $D_{n,\ep}$ from \eqref{eq:def:D:n:ep}. Here we will mainly follow the proof of \cite[Lemma 5.1]{BCP} with some modifications. Some key differences are that the Lebesgue measure of our set $D_{n, \ep}$ is $4+o(1)$ times larger, and  that we rely on our version of the Edgeworth expansion, Proposition \ref{prop:Edgeworth:trig:delta}. The proof consists of several steps as follows.
 
\textbf{Step 1: Making use of the Edgeworth expansion in Proposition \ref{prop:Edgeworth:trig:delta}.}
In this step, we shall apply the formulas in Proposition \ref{prop:Edgeworth:trig:delta}. We shall choose the diagonal matrices $I_2$ and $I_4$ that appear in Proposition \ref{prop:Edgeworth:trig:delta} to be the limit of $V_n(t)$ and $V_n(s, t)$. More precisely, by letting  $I_2$ be the $2\times 2$ diagonal matrix with diagonal entries $\lambda_1= 1$, $\lambda_2 = 1/3$, and $I_4$ be the $4\times 4$ diagonal matrix with diagonal entries $\lambda_1= \lambda_3=1$, $\lambda_2 = \lambda_4=1/3$, one can check that for all $s, t$ satisfying Condition \ref{cond:s,t}, 
\begin{equation}\label{key}
\lim_{n\to \infty} \| V_n(t) - I_2 \|_{2} = 0\quad\text{and}\quad \lim_{n\to \infty} \| V_n(s, t) - I_4 \|_{2} = 0.\nonumber
\end{equation}

Applying Proposition \ref{prop:Edgeworth:trig:delta}, we obtain the expansion
\begin{eqnarray}
v_n(s, t, Y) - v_n(s, t, G) = \frac{1}{n} \gamma_n(s, t) + \bar R_n(s, t)
\end{eqnarray}
where the $\gamma_n$ has the main terms and the $\bar R_n$ contains all the error terms and their product with the main terms in Proposition \ref{prop:Edgeworth:trig:delta}. In particular, 
\begin{eqnarray}
\gamma_n(s, t) &=& \E \big [\Psi_\delta (I_4^{1/2} W_4) \Gamma_{n,2} \big(I_4^{-1/2} X_n(s,t,Y),W_4\big )\big]\Big| \nonumber\\
&& - \E\left [\Phi_\delta(I_2^{1/2} W_2) \right ]\E\left [\Phi_\delta(I_2^{1/2} W_2) \Gamma_{n,2} \left (I_2^{-1/2} X_n(t,Y), W_2\right )\right ]  \nonumber\\
&&- \E\left [\Phi_\delta(I_2^{1/2} W_2) \right ]\E\left [\Phi_\delta(I_2^{1/2} W_2) \Gamma_{n,2} \left (I_2^{-1/2} X_n(s,Y), W_2\right )\right ]  \nonumber\\
&=&\E \bigg[ \Psi_{\delta_n}\left (I_4^{1/2} W_4\right)  \nonumber\\
&&\quad \times \left [\Gamma _{n, 2} \left (I_4^{-1/2} X_n(s, t, Y) , W_4\right ) - \Gamma _{n, 2} \left (I_2^{-1/2} X_n(t, Y) , W_2'\right )- \Gamma _{n, 2} \left (I_2^{-1/2} X_n(s, Y) , W_2''\right )\right ] \bigg] \nonumber
\end{eqnarray}

where $W_2', W_2''$ are independent standard Gaussian vectors in $\R^2$ and $W_4 = (W_2', W_2'')$ is a standard Gaussian vector in $\R^4$. We recall that
\begin{equation} 
\Gamma_{n, 2}(X, x)=\Gamma'_{n, 2}(X, x) + \Gamma''_{n, 2}(X, x)\nonumber
\end{equation}
which is the sum of the following fourth moment corrector
\begin{equation} 
\Gamma'_{n, 2}(X, x) = \frac{1}{24} \sum_{|\beta| = 4} c_n(\beta, X) H_{\beta}(x)  \nonumber
\end{equation}
and the following combined third moment corrector
\begin{equation} 
\Gamma''_{n, 2}(X, x) = \frac{1}{72} \sum_{|\rho| = 3}\sum_{|\beta| = 3} c_n(\beta, X)c_n(\rho, X) H_{\beta, \rho}(x)  .\label{eq:gamma''}
\end{equation}
We recall the definition of $c_n(\cdot, X)$ in \eqref{eq:def:c:n:alpha} and the polynomials $H_{\beta}, H_{\beta, \rho}$ in \eqref{H}.

Denote by $\gamma'_n(t, s)$ and $\gamma''_n(t, s)$ the corresponding quantities when replacing $\Gamma_{n, 2}(X, x)$ by $\Gamma'_{n, 2}(X, x)$ and $\Gamma''_{n, 2}(X, x)$, respectively, in the definition of $\gamma_n(t, s)$. Proposition \ref{prop:v_n} is reduced to showing the following:
\begin{equation}\label{eq:gamma'':cancel}
\lim _{n}\frac{1}{n^{2}} \int_{D_{n, \ep}} \gamma_n''(s, t) ds dt =0,
\end{equation}
\begin{equation}\label{gamma'}
\lim _{n}\frac{1}{n^{2}} \int_{D_{n, \ep}} \gamma_n'(s, t) ds dt = \frac{\E \xi^{4}-3}{15}+ O(\ep),
\end{equation}
and
\begin{equation}\label{bar R}
\lim _{n}\frac{1}{n} \int_{D_{n, \ep}} \bar R_n (s, t) ds dt =0.
\end{equation}
To see \eqref{bar R}, we simply note that for almost every $(s,t)\in [0,\pi]^2$, $\bar R_n(ns,nt)=o(\frac 1 n)$ as $n\to\infty$. This is mainly because (via examination), as $n\to\infty$,   the  re-scaled covariance matrix $V_n(nt,ns)$ converges to the diagonal matrix $(1,1/3, 1,1/3)$ if $\frac{t}\pi, \frac{s}\pi, \frac{t-s}{\pi}, \frac{t+s}{\pi}$ are irrational.

The remaining identities will be established in the next steps. For convenience of notation, in the rest of the section for each multi-index $\alpha$ (of dimension $3$ or $4$) we let 
$$c_n(\alpha,s,t)=c_n\left (\alpha, I_4^{-1/2} X_n(s, t, Y)\right ),  \ \ c_n(\alpha,t)=c_n\left (\alpha, I_2^{-1/2} X_n(t, Y)\right ),$$
and define $c_n(\alpha,s)$ similarly. Also,   we will denote by $\alpha-2$ the adjusted multi-index where each index in $\alpha$ will be subtracted by $2$.

\textbf{Step 2: Proving \eqref{eq:gamma'':cancel}.}

For each multi-index $\rho, \beta  \in \{1,2,3,4\}^3$,  define $\gamma_{n, \rho, \beta}''(s, t)= \E  \big\{\Psi_{\delta_n} (I_4^{1/2} W_4) \Delta c_n(\beta,\rho,s,t) \big\}$ where 
\begin{eqnarray}
\Delta c_n(\beta,\rho,s,t) &=& \begin{cases}
 c_n\left (\rho, s,t \right ) c_n\left (\beta, s,t \right )  H_{\beta, \rho}\left (W_4\right )  - c_n\left (\rho, t \right ) c_n\left (\beta, t \right )  H_{\beta, \rho}\left (W_2'\right ), & \text{if $\beta, \rho \in\{1,2\}^3$;} \\
 c_n\left (\rho, s,t \right ) c_n\left (\beta, s,t \right )  H_{\beta, \rho}\left (W_4\right )  - c_n\left (\rho-2, s \right ) c_n\left (\beta-2, s \right )  H_{\beta, \rho}\left (W_2''\right ), & \text{if $\beta, \rho \in\{3,4\}^3$;} \\
 c_n\left (\rho, s,t \right ) c_n\left (\beta, s,t \right )  H_{\beta, \rho}\left (W_4\right ), & \text{otherwise.} \\
\end{cases}
\nonumber
\end{eqnarray}
It is clear that $\gamma_n''(s,t) = \sum_{|\beta|=3} \sum_{|\rho|=3}\gamma_{n, \rho, \beta}''(s, t)$.

\underline{Step 2, the  cancellations:} We recall the following cancellations that were observed  in \cite{BCP}.
\begin{itemize}
	\item (First cancellation) 	If $\beta,\rho\in \{1,2\}^3$   then by examination we have $c_n(\beta,s,t)=c_n(\beta,t)$, $c_n(\rho,s,t)=c_n(\rho,t)$, and $H_{\beta,\rho}(W_4)=H_{\beta,\rho}(W_2')$. Consequently, $\Delta c_n(\beta,\rho,s,t)=0$ and thus
	$$\gamma_{n, \rho, \beta}''(s, t)=0.$$  
	 Similarly,	if $\beta,\rho\in \{3,4\}^3$   then $\Delta c_n(\beta,\rho,s,t)=0$ and so $\gamma_{n, \rho, \beta}''(s, t)=0$.   
	\item (Second cancellation) We now consider those $(\beta,\alpha)$ not part of above   cancellation scenarios, i.e. where there is a mixed of elements from $\{1,2\}$ and elements from $\{3,4\}$. Then if an index $j\in \{1,2,3,4\}$ appears an odd number of times   inside $(\beta,\rho)$ then 
	$$\E  \big\{\Psi_{\delta_n} (I_4^{1/2} W_4) H_{\beta,\rho}(W_4) \big\}  = 0,$$
	since $\Psi_{\delta_n} (I_4^{1/2} W_4)$ is an even function of $W_{4,j}$ (the $j$th coordinate of $W_4$) and $H_{\beta,\rho}(W_4)$ is an odd function of $W_{4,j}$.	 Consequently, in this case we also have $\gamma_{n, \rho, \beta}''(s, t)=0.$
\end{itemize}
For the remaining $(\beta,\rho)$,    for almost every $(s,t) \in [0,\pi]^2$ (with respect to the Lebesgue measure) we have
	\begin{equation}\label{eq:second:cancelation}
	c_n(\beta,   ns, nt)c_n(\rho,  ns, nt)\to 0, \ \ \text{as $n\to\infty$}.
	\end{equation}
Indeed,  the ``mixed'' nature of $(\beta,\rho)$ implies that one of $\beta$, $\rho$ will be mixed. Without loss of generality assume that $\beta$ is mixed, say $\beta=\{i,i,j\}$ where $i\le 2<j$. Then via examination $c_n(\beta,s,t)$ is an average (over $k$) of term of the following type:
$$\sum_{\ell_1,\ell_2,\ell_3=1}^2 A_k(ns,nt,\ell_1,\ell_2,\ell_3) \E (Y_{k,\ell_1} Y_{k,\ell_2}Y_{k,\ell_3})$$
whre $A_k(s,t,\ell_1,\ell_2,\ell_3)$ is a product of two terms from $C_n(k,t)$ and one term from $C_n(k,s)$, and $Y_{k,1},Y_{k,2}$ are the coordinates of $Y_k$. Now $\E (Y_{k,\ell_1} Y_{k,\ell_2}Y_{k,\ell_3})$ is constant with respect to $k$ (due to iid - although this is not essential, it suffices to assume convergence as $k\to\infty$ of this term), the desired convergence follows from the limit $\frac 1 n \sum_k A_k(ns,nt,\ell_1,\ell_2,\ell_3)\to 0$, which can be checked using elementary trigonometric identities, provided that $\frac{t}{\pi}, \frac{s}{\pi},\frac{t + s}{\pi}, \frac{t-s}{\pi}$ are all irrational.

Now, since the $c_n$ are uniformly bounded, \eqref{eq:second:cancelation} and  Lebesgue's dominated convergence theorem imply that
  \begin{equation} 
  \lim _{n}\frac{1}{n^{2}} \int_{D_{n, \ep}} \gamma_{n, \rho, \beta}''(s, t) ds dt =  \lim _{n}\int_0^{\pi} \int_0^{\pi} 1_{ D_{n, \ep}}(nu,nv) \gamma_{n, \rho, \beta}''(nu, nv) dudv = 0.\nonumber
  \end{equation}
which completes the proof of \eqref{eq:gamma'':cancel}.

\textbf{Step 3: Proving \eqref{gamma'}. } This follows from similar reasoning as in Step 2. First, using  cancellations similar to Step 2, we also obtain 
 	$$\gamma_n'(s, t) = \frac{1}{24} \sum_{\alpha} \E \left (\Psi_{\delta_n}(I_4^{1/2}W)H_\alpha(W)\right ) c_n\left (\alpha, s,t \right )$$
	where  the sum runs over all mixed $\alpha$ of the form $\alpha = (i, i, j, j)$ with $i\in \{1, 2\}$ and $j\in \{3, 4\}$, and their permutations. It is clear that $\gamma_n'(s,t)=O(1)$ uniformly over $n$ and $s,t$.

 Thus, 
 	 \begin{equation}\label{eq:gamma':c}
 \frac{1}{n^{2}} \int_{D_{n, \ep}} \gamma'(s, t) ds dt =O(\ep)+  \frac{1}{n^{2}} \frac{1}{24} \sum_{\alpha} \E \left (\Psi_{\delta_n}(I_4^{1/2}W)H_\alpha(W)\right )   \int_{D_{n, \ep}}c_n\left (\alpha, s,t \right ) dsdt.
 \end{equation}
 
Now, if $\alpha$ is a permutation of $(i, i, j, j)$ with $i\le 2<j$, by examination and using trigonometric identities as in Step 2, it follows that $c_n\left (\alpha,   s, t \right ) $ is an average of sums $\sum_{\ell_1,\ell_2,\ell_3,\ell_4=1}^2 A_k(s,t, \ell_1,\dots,\ell_4) \E (Y_{k,\ell_1}Y_{k,\ell_2}Y_{k,\ell_3}Y_{k,\ell_4})$, $k=1,\dots, n$. Here $A_k(s,t,\ell_1,\dots,\ell_4)$ is a product of two terms from $C_n(k,s)$ and two terms from $C_n(k,t)$. Via examination,  (for details see for instance the appendix in \cite{BCP}), one could show that
 \begin{lemma}\label{lm:step3:2} For almost every $(s,t) \in [0,\pi]^2$  (with respect to the Lebesgue measure) it holds that
 \begin{equation}\label{5.9}
 \lim_{n} c_n\left (\alpha,   ns, nt \right ) = \frac{2\cdot 3^{i+j-4}  (\E \xi^{4}-3)}{2i+2j-4}
 \end{equation}
 \end{lemma}
 On the other hand, it is clear that, for the same $\alpha$,
 $$\lim_{n} \E \left (\Psi_{\delta_n}(I_4^{1/2}W)H_\alpha(W)\right ) = \frac{1}{3\pi^{2}}(-1)^{i+j}.$$

Plugging in these limits to \eqref{eq:gamma':c}, we obtain \eqref{gamma'}.

\section{Proof of Lemma \ref{lemma:var}}\label{section:lemma:var}

   We recall the definition of the sub-intervals $I_k$ in \eqref{def:Ik} and the sets $D_{n, \ep}$ and $\mathcal D$ in \eqref{eq:def:D:n:ep}. Let $N_{I_k}(Y)$ be the number of roots in $I_k$.

Since $N_n(Y) = \sum_{k} N_{I_k}(Y)$, we have
\begin{equation}\label{key}
\Var N_n(Y) -\var N_n(G)= 2\V_1 +\V_2
\end{equation}
where
\begin{equation}\label{eq:def:v1}
\V_1 :=  \sum_{(k, p)\in \mathcal D} \left [\Cov (N_{I_k}(Y), N_{I_p}(Y)) - \Cov (N_{I_k}(G), N_{I_p}(G)) \right ],
\end{equation}
and
\begin{equation}\label{eq:def:v2}
\V_2 := 2 \sum_{(k, p)\notin \mathcal D, k<p} \left [\Cov (N_{I_k}(Y), N_{I_p}(Y) )- \Cov (N_{I_k}(G), N_{I_p}(G)) \right ] + \sum_{k} \left [\Var N_{I_k} (Y) - \Var N_{I_k} (G) \right ].
\end{equation}

Therefore, Lemma \ref{lemma:var} follows from the following two results concerning $\V_1$ and $\V_2$, respectively.
\begin{lemma}[asymptotic estimate for $\V_1$]\label{lm:v1}  We have
	\begin{equation}\label{key}
	\V_1 = \int_{D_{n,\eps}} \left (v_n(t, s, Y) - v_n(t, s, G)\right ) dsdt + R_{n, \ep}
	\end{equation} 
	where 
	\begin{equation}\label{key}
	\lim_{n} \frac{R_{n, \ep}}{n} = 0.\nonumber
	\end{equation}
\end{lemma}
We recall the definition of $v_n$ in \eqref{eq:def:vn}:
\begin{equation}
v_n(s,t,Y)=\cov(P'_n(s,Y) 1_{ |P_n(s,Y)|<\delta_n},P'_n(t,Y) 1_{ |P_n(t,Y)|<\delta_n }) = \E \Psi_{\delta_n}\left (\frac{1}{\sqrt{n}} S_n(s,t,Y)\right )  - \E \phi_{\delta_n}(s,Y) \E \phi_{\delta_n}(t,Y).\nonumber
\end{equation}

\begin{lemma}[$\V_2$ is negligible]\label{lm:v2}  There exists a constant $c$ such that 
	\begin{equation}\label{key}
	\V_2 \ll n^{1- c}.\nonumber
	\end{equation}
\end{lemma}
\begin{remark}\label{rmk:condition}
	As we shall show in the proof, for Lemma \ref{lm:v2}, we only need to assume that the random variables $y_{ij}$ are independent (not necessarily identically distributed) with mean 0, variance 1, and bounded $(2+\ep_0)$-moment, namely $\E |y_{ij}|^{2+\ep_0}<C$ for some positive constants $\ep_0, C$ and for all $i, j$.
\end{remark}
 
 The rest of this section is devoted to the proof of these lemmas.
 
 \subsection{Proof of Lemma \ref{lm:v1}}
 For this proof, we adapt the proof of \cite[Lemma 4.2]{BCP} using the new inequalities that we have obtained.

Recall that $\delta = n^{-5}$ in this proof. Let 
$$\delta_{a, b, Y}:= \min_{t\in [a, b]} \{|P_n(a, Y)|, |P_n(b, Y)|, |P_n(t, Y)|+|P_n'(t, Y)|\}.$$

By the Kac-Rice formula, for any interval $[a, b]$, the number of zeros of $P_n(\cdot, Y)$ in the interval $[a, b]$ is given by 
 \begin{equation}\label{KacRice}
 N_{n}([a, b], Y) = \int_{a}^{b} |P_n'(t, Y)|\textbf{1}_{|P_n(t, Y)|\le \delta} \frac{dt}{2\delta} \quad \text{ if } \delta\le \delta_{a, b, Y}.
 \end{equation}

To prove Lemma \ref{lm:v1}, it suffices to show that for any $(k, p)\in \mathcal D$, 
\begin{equation}\label{eq:v1:1}
\E N_{I_k}(Y) N_{I_p}(Y) = \int_{I_k\times I_p} \E \phi_{\delta} (t, Y) \phi_{\delta}(s, Y) dt ds + O(\ep_{k, p}) 
\end{equation}
and
\begin{equation}\label{eq:v1:2}
\E N_{I_k}(Y) \E N_{I_p}(Y) = \int_{I_k\times I_p} \E \phi_{\delta} (t, Y) \E \phi_{\delta}(s, Y) dt ds + O(\ep_{k, p}) 
\end{equation}
where 
\begin{equation}\label{key}
\sum_{(k, p)\in \mathcal D} \ep_{k, p} = o(n).\nonumber
\end{equation}

Since the proof of \eqref{eq:v1:1} and \eqref{eq:v1:2} are similar, we shall now only prove \eqref{eq:v1:1}. By the Kac-Rice formula \eqref{KacRice}, 
\begin{equation} 
\E N_{I_k}(Y) N_{I_p}(Y) \textbf{1}_{\delta\le \min\{\delta_{I_k, Y}, \delta_{I_p, Y}\}} = \int_{I_k\times I_p} \E \phi_{\delta} (t, Y) \phi_{\delta}(s, Y)\textbf{1}_{\delta\le \min\{\delta_{I_k, Y}, \delta_{I_p, Y}\}}  dt ds. \nonumber
\end{equation}
Thus, by setting
\begin{equation}\label{key}
\ep'_{k, p} = \E N_{I_k}(Y) N_{I_p}(Y) \textbf{1}_{\delta > \min\{\delta_{I_k, Y}, \delta_{I_p, Y}\}} 
\end{equation}
and 
\begin{equation}\label{key}
\ep''_{k, p} = \int_{I_k\times I_p} \E \phi_{\delta} (t, Y) \phi_{\delta}(s, Y)\textbf{1}_{\delta > \min\{\delta_{I_k, Y}, \delta_{I_p, Y}\}}  dt ds,
\end{equation}
we are left to show that 
\begin{equation}\label{eq:v1:ep'}
\sum_{(k, p)\in \mathcal D} \ep'_{k, p} = o(n)
\end{equation}
and
\begin{equation}\label{eq:v1:ep''}
\sum_{(k, p)\in \mathcal D} \ep''_{k, p} = o(n).
\end{equation}

For \eqref{eq:v1:ep'}, using the fact that the number of real roots inside $[-n\pi, n\pi]$ is at most $2n$ deterministically, we get that
\begin{eqnarray}\label{eq:v1:ep':5}
\ep'_{k, p} &\ll& n^{2}\P\left (\delta > \min\{\delta_{I_k, Y}, \delta_{I_p, Y}\}\right )\le n^{2}\P\left (\delta >  \delta_{I_k, Y} \right ) + n^{2}\P\left (\delta >  \delta_{I_p, Y} \right ).
\end{eqnarray}
Let $a, b$ be the endpoints of $I_k$. We have
\begin{eqnarray} \label{eq:v1:ep':2}  
\P\left (\delta >  \delta_{I_k, Y} \right )&\le&  \P\left (|P_n(a, Y)|<\delta\right ) + \P\left (|P_n(b, Y)|<\delta\right ) + \P\left (\min_{t\in I_k} |P_n(t, Y)|+|P_n'(t, Y)|<\delta\right ).
\end{eqnarray}
Observe that for any $(s, t)$ that satisfies Condition \ref{cond:s,t}, it is necessary that both $s$ and $t$ satisfy Condition \ref{cond:t}. Thus, for all $(k, p)\in \mathcal D$, we have $I_k\subset \mathcal G$ in Theorem \ref{thm:smallball:inf}. Applying this theorem, we get
\begin{eqnarray}\label{eq:v1:ep':3}
\P\left (\min_{t\in I_k} |P_n(t, Y)|+|P_n'(t, Y)|<\delta\right ) \ll n^{-4 +\ep}  
\end{eqnarray}
where we recall that in Theorem \ref{thm:smallball:inf}, $\frac{1}{\sqrt{n}}S_n(Y, t) = (P_n(t, Y), P_n'(t, Y))$ as defined in \eqref{eqn:C_n:t}.

Applying Corollary \ref{cor:smallball:2} with $M_0 = 4$, for all $t$ satisfying Condition \ref{cond:t}, we have
\begin{equation}\label{eq:v1:ep':4}
\P\left (|P_n(t, Y)|<\delta\right ) \ll n^{-3-\ep}.
\end{equation}

Applying \eqref{eq:v1:ep':4} for $t = a, b$, we get
$$\P\left (|P_n(a, Y)|<\delta\right ) + \P\left (|P_n(b, Y)|<\delta\right ) \ll n^{-3-\ep}.$$
Plugging this together with \eqref{eq:v1:ep':3} to \eqref{eq:v1:ep':2}, we obtain 
$$\P\left (\delta >  \delta_{I_k, Y} \right ) \ll n^{-3-\ep}.$$
Similarly for $I_p$. Thus, from \eqref{eq:v1:ep':5}, we have $\ep'_{k, p}\ll n^{-1-\ep}$ which gives \eqref{eq:v1:ep'}.

For \eqref{eq:v1:ep''}, we argue similarly using the observation in \cite[Inequality (4.2)]{BCP} that, deterministically, 
\begin{equation}\label{key}
\int_{I_k} \phi_\delta(t, Y) dt \le 1 + N_{I_k}(Y)\le 2n+1 \quad\text{and}\quad \int_{I_p} \phi_\delta(t, Y) dt \le 1 + N_{I_p}(Y) \le 2n+1.\nonumber
\end{equation}

\subsection{Proof of Lemma \ref{lm:v2}}\label{sec:proof:lm:v2} As in Remark \ref{rmk:condition}, in this subsection, we only assume that the random variables $y_{ij}$ are independent (not necessarily identically distributed) with mean 0, variance 1, and bounded $(2+\ep_0)$-moment.
To prove Lemma \ref{lm:v2}, we shall use the following result.
\begin{lemma}\label{lm:correlation} There exists a constant $c$ such that for all $k, p$ that are not necessarily distinct,
	\begin{equation}\label{eq:localuniv:1}
	\E N_{I_k}(Y) N_{I_p}(Y) - \E N_{I_k}(G)  N_{I_p}(G) \ll n^{-2c},
	\end{equation}
	\begin{equation}\label{eq:localuniv:2}
	\E N_{I_k}(Y) - \E N_{I_k}(G) \ll n^{-2c},
	\end{equation}
	and
	\begin{equation}\label{eq:localuniv:3}
	\E N_{I_k}(Y)\ll 1, \E N_{I_k}(G)\ll 1.
	\end{equation}

\end{lemma}

\begin{proof}[Proof of Lemma \ref{lm:v2}] Assuming Lemma \ref{lm:correlation}, we have for all $k, p$, 
	\begin{equation}\label{key}
	\E N_{I_k}(Y)\cdot \E N_{I_p}(Y) -\E N_{I_k}(G)\cdot \E N_{I_p}(G)\ll n^{-2c}\nonumber
	\end{equation}
where we used the triangle inequality, \eqref{eq:localuniv:2}, and \eqref{eq:localuniv:3}.
Combining this with \eqref{eq:localuniv:1}, we obtain 
\begin{equation}\label{eq:cov:compare} 
\Cov (N_{I_k}(Y), N_{I_p}(Y) )- \Cov (N_{I_k}(G), N_{I_p}(G)) \ll n^{-2c} \quad \text{for all } k, p,
\end{equation}
and in particular when $k=p$, we have
\begin{equation}\label{key}
\Var N_{I_k} (Y) - \Var N_{I_k} (G) \ll n^{-2c}.\nonumber
\end{equation}
Plugging these estimates into \eqref{eq:def:v2}, we obtain
\begin{equation}\label{eq:v2:last}
\V_2\ll n^{-2c}\#\{(k, p)\notin \mathcal D\}.
\end{equation} 
Observe that for each $k$, the number of $p$ such that $(k, p)\notin \mathcal D$ is $O_{\ep}(n^{11\tau})$. Indeed, by the definition of $\mathcal D$ and Condition \ref{cond:s,t}, for each $l_1, l_2 \neq 0$ with $|l_1|, |l_2|\le n^{\tau}$, it suffices to show that the number of $p$ such that there exist $t\in I_k, s\in I_p$ with 
\begin{equation}\label{eq:l1:l2}
\left \|\frac{l_1 t+l_2 s}{n}\right \|_{\R/\pi\Z} \le n^{-1+8\tau}
\end{equation}
is $O_{\ep}(n^{9\tau})$. The inequality \eqref{eq:l1:l2} is equivalent to
$$l_1 t+l_2 s \in [n a-n^{8\tau}, n a+n^{8\tau}]\quad \text{ for some } a\in \pi\Z, |a|\le 2\pi n^{\tau};$$
in other words, 
$$ s\in \frac{1}{l_2} \left ([n a-n^{8\tau}, n a+n^{8\tau}] - l_1 I_k\right )\quad \text{ for some } a\in \pi\Z, |a|\le 2\pi n^{\tau}.$$
For each $a$, the right-hand side is contained in an interval of length $O(n^{8\tau})$ which corresponds to $O_{\ep}(n^{8\tau})$ values of $p$. Taking union bound over $O(n^{\tau})$ choices of $a$ gives the stated claim.

Using this observation, the right-hand side of \eqref{eq:v2:last} is $O(n^{-2c+1+11\tau})= O(n^{1-c})$ by choosing $\tau$ to be sufficiently small compared to $c$. 
\end{proof}


To prove Lemma \ref{lm:correlation}, we denote the roots of $P_n(\cdot, Y)$ by $\zeta_1(Y), \dots, \zeta_n(Y)$ and the roots of $P_n(\cdot, G)$ by $\zeta_1(G), \dots, \zeta_n(G)$. We shall use the following result in \cite{ONgV}.
\begin{theorem}\label{thm:local:uni} \cite[Theorem 3.3]{ONgV}
	There exist constants $c,C'$ such that for any real numbers $x_1, x_2$ and for any function $F: \mathbb{R}^{2}\to \mathbb{R}$ supported on $[x_1-1, x_1+1] \times[x_2-1, x_2+1] $ with  continuous derivatives up to order $8$ and $||{\triangledown^a F}||_\infty\le 1$ for all $0\le a\le 8$, we have
	\begin{eqnarray}\nonumber
	\left |\E\sum F\left (\zeta_{i}(Y), \zeta_{j}(Y)\right) 
	-\E\sum F\left ( \zeta_{i}(G),   \zeta_{j}(G)\right) \right | \le C' n^{-c},
	\end{eqnarray}
	where the first sum runs over all pairs $(\zeta_{i}(Y), \zeta_{j}(Y))\in \R^{2}$ of the roots of $P_n(\cdot, Y)$ and the second sum runs over all pairs $( \zeta_{i}(G),  \zeta_{j}(G))\in \R^{2}$ of the roots of $P_n(\cdot, G)$. 
\end{theorem}
\begin{proof}[Proof of Lemma \ref{lm:correlation}]

	Let $x_k, x_p$ be the midpoint of $I_k$, $I_p$, respectively. Let $\gamma = n^{-s}$ for $s = c/100$ and $c$ be the constant in Theorem \ref{thm:local:uni}.

	The inequalities in \eqref{eq:localuniv:3} follow from \cite[Theorem 3.6]{ONgV}. For the rest of the proof, we show \eqref{eq:localuniv:1}. The proof of \eqref{eq:localuniv:2} is similar (and simpler).

	We approximate the indicator function on the interval $[-\ep/2, \ep/2]$ by a smooth function $\phi$ satisfying 
	$$\textbf{1}_{[-\ep/2+\gamma, \ep/2 - \gamma]}\le \phi\le \textbf{1}_{[-\ep/2, \ep/2]}$$
	and 
	$$||{\triangledown^a \phi}||_\infty\ll \gamma^{-a}, \quad\forall 0\le a\le 8.$$
	Let 
	$$F(x, y): = \phi(x+x_k)\phi(y+x_p).$$
	Let 
	$$M_k(Y): = \left (\sum_{i=1}^{n} \phi\left (\zeta_i (Y) -x_k\right) \right ),\quad M_p (Y) : = \left (\sum_{i=1}^{n} \phi\left (\zeta_i(Y) -x_p\right) \right ).$$
	Denote by $M_k(G)$ and $M_p(G)$ the corresponding terms for the Gaussian case, i.e., with $\zeta_i(G)$ in place of $\zeta_i(Y)$. 
	Applying Theorem \ref{thm:local:uni} to the function $\gamma ^{8}F$, we obtain
	\begin{eqnarray}\nonumber
	\left |\E\sum F\left (\zeta_i(Y), \zeta_j(Y)\right) 
	-\E\sum F\left ( \zeta_i(G),   \zeta_j(G)\right) \right |\le C'n^{-c},
	\end{eqnarray}
	and so
	\begin{eqnarray}\label{eq:mmmm}
	\left |\E M_k(Y) M_p(Y) -\E  M_k(G)  M_p(G) \right |\le C'n^{-c/2},
	\end{eqnarray}
	We shall show that
	\begin{equation}\label{eq:nnmm}
	\E N_{I_k}(Y) N_{I_p}(Y) - \E M_k(Y) M_p(Y) = O\left (n^{-s/10}\right ).
	\end{equation}
	The same argument applied to the Gaussian case will show that 
	\begin{equation}\label{eq:nnmm:g}
	\E N_{I_k}(G) N_{I_p}(G) - \E M_k(G) M_p(G) = O\left (n^{-s/10}\right ).
	\end{equation}
	Combining \eqref{eq:mmmm}, \eqref{eq:nnmm}, and \eqref{eq:nnmm:g}, we obtain \eqref{eq:localuniv:1} as desired (by choosing the $c$ in \eqref{eq:localuniv:1} to be $s/10$). 
	
	To prove \eqref{eq:nnmm}, by Holder's inequality, we have
	\begin{equation}\label{eq:ninjminj}
	( \E N_{I_k} (Y) N_{I_p} (Y) - \E M_k(Y) N_{I_p}(Y))^{2} \ll \E (N_{I_k}(Y) - M_k(Y))^{2} \E N_{I_p}^{2}(Y).
	\end{equation}
	Let 	$N_{\gamma}(Y)$ be the number of roots of $P_n(\cdot, Y)$ in the union of the intervals $[x_k+\ep/2-\gamma, x_k+\ep/2]$, $[x_k-\ep/2, x_k-\ep/2-\gamma]$, $[x_p+\ep/2-\gamma, x_p+\ep/2]$, and $[x_p-\ep/2, x_p-\ep/2-\gamma]$. We observe that $N_{\gamma}(Y)$ is at least $|N_{I_k}(Y) - M_k(Y)|$. 
 
	By \cite[Formula (28), page 32]{ONgV}, there exists an $x\in I_k$ such that
	\begin{equation}\label{key}
	\P\left (\log |P_n(x, Y)|\le -n^{s/10}\right )\ll n^{-100}.\nonumber
	\end{equation}
	By \cite[Lemma 9.4]{ONgV}, 
	\begin{equation}\label{key}
	\P\left (\log \max_{z\in B(x, 100\ep)}|P_n(z, Y)|\ge n^{s/10}\right )\ll n^{-100}.\nonumber
	\end{equation}
	By Jensen's inequality (see, for example, \cite[Formula (8), page 22]{ONgV}), under the event that $\log |P_n(x, Y)|\ge -n^{s/10}$ and $\log \max_{z\in B(x, 100\ep)}|P_n(z, Y)|\le n^{s/10}$, we have $N_{I_k}(Y) \le n^{s/10}$. Thus, 
	\begin{equation}\label{eq:ngamma:1}
	\P\left (N_{I_k}(Y)\ge n^{s/10}\right )\ll n^{-100}.
	\end{equation}

	And by \cite[Lemma 8.6]{ONgV}, 
	\begin{equation}\label{eq:n:gamma:3s2}
	\P\left (N_{\gamma}(Y)\ge 2\right )\ll n^{-3s/2}.
	\end{equation}
	When $N_{\gamma}(Y)< 2$, we have $N_{\gamma}(Y)^{2} = N_{\gamma}(Y)$. Thus,
	\begin{eqnarray}\label{eq:n:gamma}
	\E N_{\gamma}(Y)^{2} &\le& \E N_{\gamma}(Y) + \E N_{\gamma}(Y)^{2}\textbf{1}_{2\le N_{\gamma}(Y)\le n^{s/10}} + \E N_{\gamma}(Y)^{2}\textbf{1}_{n^{s/10}\le N_{\gamma}(Y)\le n}\nonumber\\
	&\ll& n^{-s/2} +  \E N_{\gamma}(Y)^{2}\textbf{1}_{2\le N_{\gamma}(Y) \le n^{s/10}} + \E N_{\gamma}(Y)^{2}\textbf{1}_{n^{s/10}\le N_{\gamma}(Y)\le n}\quad \text{ by \cite[Corollary 3.7]{ONgV}}\nonumber\\
&\ll& n^{-s/2} +n^{-s}+ n^{-90} \ll n^{-s/2} \quad\text{by \eqref{eq:ngamma:1} and \eqref{eq:n:gamma:3s2}}.\nonumber
	\end{eqnarray}
 	 
	Similarly, 
	\begin{eqnarray}\label{eq:n:gamma:p}
	\E N_{I_p}^{2}(Y) &\le& \E N_{I_p}^{2}(Y)\textbf{1}_{N_{I_p}(Y)\le n^{s/10}} + \E N_{I_p}^{2}(Y)\textbf{1}_{n^{s/10}\le N_{I_p}(Y)\le n}\ll n^{s/5}+ n^{-90} \ll n^{s/5}.
	\end{eqnarray}
	
	Plugging \eqref{eq:n:gamma} and \eqref{eq:n:gamma:p} into \eqref{eq:ninjminj}, we get
	\begin{equation} 
	( \E N_{I_k} (Y) N_{I_p}(Y) - \E M_k(Y) N_{I_p})^{2} \ll  n^{-s/2}\E N_{I_p}^{2} (Y)\ll n^{-s/10}.\nonumber
	\end{equation}
	Similarly, 
	\begin{equation} 
	( \E M_k(Y) N_{I_p} (Y)- \E M_k(Y) M_p(Y))^{2} \ll n^{-s/2}.\nonumber
	\end{equation}
	Combining these two inequalities gives \eqref{eq:nnmm} and completes the proof.
\end{proof}

\section{Variance estimate under the non-iid regime}\label{section:var:asymp}
 
\begin{theorem}\label{thm:var:asymp} Assume that the coefficients $y_{ij}$ are independent (but not necessarily identically distributed) of mean zero, variance one, and bounded $(2+\eps_0)$-moment: $\E |y_{ij}|^{2+\ep_0}<C$ for some positive constants $\eps_0, C$ and for all $i,j$. Then  there exists a constant $c>0$ such that 
\begin{equation}\label{eq:asym:expec}
\E N_n = \left (\frac{2}{\sqrt 3}+O(n^{-c})\right )n
\end{equation}
and 
\begin{equation}\label{eq:asym:var}
\Var(N_n) = O(n^{2-c}),
\end{equation}
where the implied constants depend on $C$ and $\eps_0$.
\end{theorem}

\begin{proof}
	Equation \eqref{eq:asym:expec} is simply \cite[Corollary 3.7]{ONgV} in which the $c_i$ are all 1 and $u_n = 0$.
	
	For Equation \eqref{eq:asym:var}, as in Section \ref{section:lemma:var},  we recall the definition of the sub-intervals $I_k$ in \eqref{def:Ik} and we denote by $N_{I_k}(Y)$ the number of roots of $P_n(\cdot, Y)$ in $I_k$. By \eqref{eq:cov:compare}, there exists a constant $c>0$ such that for all indices $k, p$,
	\begin{equation} 
	\Cov (N_{I_k}(Y), N_{I_p}(Y) )- \Cov (N_{I_k}(G), N_{I_p}(G)) \ll n^{-2c}.\nonumber
	\end{equation}
	As mentioned in Remark \ref{rmk:condition} and the beginning of Subsection \ref{sec:proof:lm:v2}, this inequality was proven under the more general assumption that the $y_{ij}$ are independent with mean zero, variance one, and bounded $(2+\eps_0)$-moment.
	
	Since $N_n(Y) = \sum_{k\in \Z \cap [-n\pi/\ep, n\pi/\ep]} N_{I_k}(Y)$ and similarly for $N_n(G)$, we have 
	\begin{equation}\label{key}
	\Var (N_n(Y)) - \Var (N_n(G)) = \sum_{k, p\in \Z \cap [-n\pi/\ep, n\pi/\ep]} \Cov (N_{I_k}(Y), N_{I_p}(Y) )- \Cov (N_{I_k}(G), N_{I_p}(G))  \ll n^{2-2c}.\nonumber
	\end{equation}
	By Theorem \ref{thm:GW}, 
	\begin{equation}\label{key}
	\Var (N_n(G)) \ll n^{2 - 2c}.\nonumber
	\end{equation}
	Combining these bounds gives \eqref{eq:asym:var}.
\end{proof}

\section{Characteristic functions in $\R^2$, proof of Theorem \ref{thm:fourier:2}}\label{section:fourier:2}

Given a real number $w$ and a random variable $\xi$, we define the $\xi$-norm of $w$ by
$$\|w\|_\xi := (\E\|w(\xi_1-\xi_2)\|_{\R/\Z}^2)^{1/2},$$
where $\xi_1,\xi_2$ are two iid copies of $\xi$. For instance if $\xi$ is Bernoulli with $\P(\xi=\pm 1) =1/2$ (which is our main focus), then $\|w\|_\xi^2 = \| 2w\|_{\R/\Z}^2/2$.

The following works for general $\R^d$: consider the random walk $\sum_i y_{i1} \Bw_i + y_{i2} \Bw_i'$, where $\Bw_i, \Bw_i'$ are vectors in $\R^d$. Then its characteristic function can be bounded by (see \cite[Section 5]{TVcir})
\begin{align}\label{eqn:cahr:bound}
 \prod|\E e( y_{i1} \langle \Bw_i, x\rangle)  \prod \E e( y_{i2} \langle \Bw_i', x \rangle) | &\le \prod_i [|\E e( y_{i1} \langle \Bw_i, x\rangle)|^2/2+1/2] \prod_i [|\E e( y_{i2} \langle \Bw_i', x\rangle)|^2/2+1/2] \nonumber \\
& \le \exp(-(\sum_i \|\langle \Bw_i, x/2\pi \rangle\|_{\xi}^2 +\|\langle \Bw_i',x/2\pi \rangle\|_{\xi}^2) /2).
\end{align}

Hence if we have a good lower bound on the exponent $\sum_i \|\langle \Bw_i, x/2\pi \rangle\|_\xi^2+\|\langle \Bw_i', x/2\pi \rangle\|_\xi^2$ then we would have a good control on $|\prod \phi_i(x)|$. Furthermore, by definition
\begin{align}\label{eqn:decoupling}
\sum_i \|\langle \Bw_i, x/2\pi \rangle\|_\xi^2 + \|\langle \Bw_i', x/2\pi \rangle\|_\xi^2 &= \sum_i \E\| \langle \Bw_i, x/2\pi \rangle (\xi_1-\xi_2)\|_{\R/\Z}^2 + \E\| \langle \Bw_i', x/2\pi \rangle (\xi_1-\xi_2)\|_{\R/\Z}^2\nonumber \\
& = \E (\sum_i \| \langle \Bw_i, x/2\pi \rangle (\xi_1-\xi_2)\|_{\R/\Z}^2 + \sum_i \| \langle \Bw_i', x/2\pi \rangle (\xi_1-\xi_2)\|_{\R/\Z}^2) \nonumber \\
&= \E_y (\sum_i \| y \langle \Bw_i, x/2\pi \rangle\|_{\R/\Z}^2+ \sum_i \| y \langle \Bw_i', x/2\pi \rangle\|_{\R/\Z}^2),
\end{align}
where $y=\xi_1-\xi_2$.
As $\xi$ have mean zero, variance one and bounded $(2+\eps_0)$-moment, there exist positive constants $c_1\le c_2,c_3$ such that $\P(c_1 \le |y| \le c_2) \ge c_3$, and so 
\begin{equation}\label{eqn:y}
\E_y \sum_i \| y \langle \Bw_i, x/2\pi \rangle\|_{\R/\Z}^2 +\sum_i \|y \langle \Bw_i', x /2\pi \rangle\|_{\R/\Z}^2  \ge c_3 \inf_{c_1\le |y| \le c_2} \sum_i \| y \langle \Bw_i, x/2\pi \rangle\|_{\R/\Z}^2+\sum_i \| y \langle \Bw_i', x/2\pi \rangle\|_{\R/\Z}^2.
\end{equation}

Hence for Theorem \ref{thm:fourier:2} (and similarly for Theorem \ref{thm:fourier}) it suffices to show that for any $\BD=(d_1,d_2)$ (which plays the role of $(y/2\pi)x$) such that $c_1 n^{5\tau-1/2}  \le \|\BD\|_2\le c_2 n^{C_\ast}$ we have 
\begin{equation}\label{eqn:thm:fourier}
\sum_i \|\langle \Bu_i,   \BD \rangle\|_{\R/\Z}^2 +  \|\langle \Bu_i',   \BD \rangle\|_{\R/\Z}^2 \ge n^{2\tau}.
\end{equation}

For $t\in [-n \pi, n\pi]$, we define $\psi_i(t), \psi_i(t)'$ by
 \begin{equation}\label{eqn:psi1}
 \psi_i = d_1 \cos(it/n) -d_2 \frac{i}{n} \sin (it/n) \mbox{ and } \psi_i' = d_1 \sin(it/n) +d_2 \frac{i}{n} \cos (it/n) .
\end{equation}
In other words,
$$\psi_i = \langle \BD, \Bu_i \rangle \mbox{ and } \psi_i' = \langle \BD, \Bu_i' \rangle.$$
Let $\Be$ be the unit vector in the direction of $D$, $\Be = \frac{\BD}{\|\BD\|_2}$. Define
$$T:= n^{\tau_\ast}.$$ 
Our key ingredient in the proof of Theorem \ref{thm:fourier:2} is the following substantial generalization of \cite[Lemma 4.3]{KSch}.

\begin{prop}\label{prop:largepsi:t} Assume that $t$ satisfies Condition \ref{cond:t}. We have
$$\sum_j \|\psi_j\|_{\R/\Z}^2 + \sum_j \|\psi_j'\|_{\R/\Z}^2 \ge T.$$
\end{prop}
It is clear this result implies Theorem \ref{thm:fourier:2} via \eqref{eqn:cahr:bound}. We note that \cite{KSch} treated mainly with Bernoulli and only up to $\|\BD\|_2= n^{1/2+o(1)}$. This was generalized to $\|\BD\|_2=n^{1-o(1)}$ in \cite{NgZ} for general ensembles. The innovative of the current note is that we can extend to $n \ll \|\BD\|_2 \le n^{C_\ast}$ for any fixed $C_\ast$, so that we can obtain small ball estimates and Edgeworth expansion under microscopic scales, as needed for our variance calculations.

Before proving this main result of the section, we introduce a useful bound as follows.
\begin{claim}\label{claim:t} Assume that $\tau_\ast$ is sufficiently small given $\tau$, and assume that $t$ satisfies Condition \ref{cond:t}. Let $I \subset [n]$ be any arithmetic progression of length $n^{1-6 \tau_\ast}$. Then
\begin{enumerate} 
\item For all $\eps_1,\eps_2\in \{-1,1\}$, and any positive integer $A_0=O(n^{\tau_\ast})$ there exists $i\in I$ so that  
$$\eps_1 \sin(i A_0 s/n), \eps_2 \cos (i A_0s/n)>0.$$
\item For any unit vector $\Be \in \R^2$ we have 
\begin{equation}\label{eqn:Be:t}
\sum_{i\in I} \langle \Be, \Bu_i\rangle^2 \ge n^{1-\tau} \mbox{ and } \sum_{i\in I} \langle \Be, \Bu_i'\rangle^2 \ge n^{1-\tau}.
\end{equation}
\end{enumerate}
\end{claim}

\begin{proof} See the proof of Claim \ref{claim:s,t}. 
\end{proof}

For the rest of this section we prove Proposition \ref{prop:largepsi:t} by contradiction: assume the opposite that we have
\begin{equation}\label{eqn:contradiction:t}
\sum_j \|\psi_j\|_{\R/\Z}^2 +\|\psi_j'\|_{\R/\Z}^2 \le T.
\end{equation}
We will then show that this is impossible as long as $t$ satisfy Condition \ref{cond:t}. We will do so by many steps. First, it follows from \eqref{eqn:contradiction:t} that
$$|\{j \in [0,n) \cap \Z: \|\psi_j\|_{\R/\Z}+\|\psi_j'\|_{\R/\Z} > 1/T\}|\le 2T^3$$
and so for large $n$ there exists an interval $J=[a,b] \subset [n]$ of length $n/T^6$ so that for $j\in J$
\begin{equation}\label{eqn:psi:small:t}
\|\psi_j\|_{\R/\Z}+\|\psi_j'\|_{\R/\Z} <1/T.
\end{equation}



{\bf Differencing.} Let $A,k$ be chosen later so that 
\begin{equation}\label{eqn:A,k}
4\|\BD\|_2 \frac{(4\pi)^{k}}{A^{(k-3)/2}} + 4 \times 2^{k}\frac{1}{T}<1.
\end{equation} 
By pigeonholing, we can find $p_0\in \Z, p_0 \neq 0$ and $t_0$ so that  
\begin{equation}\label{eqn:approx:pq}
p_0 \frac{t}{2\pi} -t_0 \in \Z, 1\le |p_0| \le A, |t_0| \le \frac{4}{A}.
\end{equation}
From the approximation we infer that 
\begin{equation}\label{eqn:closeone:t}
|e^{\sqrt{-1} p_0 \frac{t}{n}} -1|= |e^{-\sqrt{-1} (2\pi t_0)}-1| \le |2\sin(\pi t_0)| \le 4 \pi/A.
\end{equation}

Next, for a sequence $\{g_j\}_{j\in [n]}$  we define the discrete differentials with step $p_0$ by
$$\Delta^k g_{j,p_0}:= \sum_{i=0}^k \binom{k}{i} (-1)^i g_{j+ i p_0}.$$
 Let $m_j$ and $m_j'$ be the integers closest to $\psi_j$ and $\psi_j'$ respectively. Thus for $j\in J$ by \eqref{eqn:psi:small} we have $|\psi_j - m_j| \le 1/T$ and $|\psi_j' - m_j'| \le 1/T$ and
 $$|m_j|, |m_j'| \le 2\|\BD\|_2 .$$
  We show that
\begin{lemma}\label{lemma:diff:t} We have
\begin{equation}\label{eqn:Delta:diff:t}
|\Delta^k m_{j,p_0}|  + |\Delta^k m_{j,p_0}'| \le 4\|\BD\|_2 \frac{(4\pi)^{k}}{A^{(k-3)/2}} + 4 \times 2^{k}\frac{1}{T}
\end{equation}
provided that $[j, j+ k p_0 ] \subset J$.
\end{lemma}
\begin{proof} [Proof of Lemma \ref{lemma:diff:t}]
Recall that ${\psi}_j =d_1 \cos(j t/n) -d_2 \frac{j}{n} \sin (j t/n)$ and $\|\psi_j\|_{\R/\Z} \le 1/T$ over $j\in J$. Consider 
$$\Delta^k {\psi}_{j,p_0} = \sum_{i=0}^k \binom{k}{i} (-1)^i {\psi}_{j+ ip_0}.$$
We first have
\begin{align}\label{eqn:z_{q_0}:1:t}
\sum_{i=0}^k \binom{k}{i} (-1)^i  \cos\left( \frac{i p_0 t}{n} + \frac{j t}{n}\right) &=\Re \sum_{i=0}^k \binom{k}{i} (-1)^i  e\left (\frac{i p_0 t}{n} + \frac{jt}{n}\right ) \nonumber \\
&= \Re \left [\left (\sum_{i=0}^k \binom{k}{i} (-1)^i  e(\frac{i p_0 t}{n})\right ) e\left(\frac{jt}{n}\right) \right ] \nonumber \\
&= \Re \left [\left (1-e\left(\frac{p_0 t}{n}\right)\right )^k e\left(\frac{jt}{n}\right) \right ] \nonumber\\
&\le (4 \pi/A)^k < (4 \pi/\sqrt{A})^k
\end{align}
where we used \eqref{eqn:closeone:t} in the last estimate. It also follows that
\begin{align}\label{eqn:z_{q_0}:2:t}
\sum_{i=0}^k \binom{k}{i} (-1)^i  \frac{j+ip_0}{n}\sin\left( \frac{i p_0 t}{n} + \frac{jt}{n}\right) &=\frac{\partial}{\partial t}\left (\sum_{i=0}^k \binom{k}{i} (-1)^i  \cos\left (\frac{i p_0 t}{n} + \frac{jt}{n}\right )\right ) \nonumber \\
&=\frac{\partial}{\partial t} \left (\Re \left (\left (1-e\left(\frac{p_0 t}{n}\right)\right )^k\right ) e\left(\frac{jt}{n}\right) \right ) \nonumber \\
&=\Re \left ( \frac{\partial}{\partial t} \big((1-e\left(\frac{p_0 t}{n}\right))^k) e\left(\frac{jt}{n}\right) \big)\right ) \nonumber \\
&= \Re \Bigg ( -k \sqrt{-1} \frac{p_0}{n} \left (\left (1-e\left(\frac{p_0 t}{n}\right)\right )^{k-1}  e\left(\frac{jt}{n}\right) \right ) \nonumber \\
&+ \sqrt{-1} \frac{j}{n}\left (\left (1-e\left(\frac{p_0 t}{n}\right)\right )^k e\left (\frac{jt}{n} \right ) \right ) \Bigg) \nonumber\\
&\le A (4 \pi/A)^{k-1} + (4 \pi/A)^k < (4\pi)^{k}/A^{(k-3)/2}.
\end{align}

Putting the bounds together we obtain
$$|\Delta^k {\psi}_{j,p_0}| = |\sum_{i=0}^k \binom{k}{i} (-1)^i {\psi}_{j+ ip_0}| \le |d_1| (\frac{4 \pi}{\sqrt{A}})^k + |d_2|\frac{(4\pi)^{k}}{A^{(k-3)/2}} < 4\|\BD\|_2 \frac{(4\pi)^{k}}{A^{(k-3)/2}}.$$
One can also obtain similarly estimates for $|\Delta^k {\psi}_{j,p_0}'|$ where we recall that $\psi_i' = d_1 \sin(it/n) +d_2 \frac{i}{n} \cos (it/n)$. More precisely, as in \eqref{eqn:z_{q_0}:1:t}
\begin{align}\label{eqn:z_{q_0}:1:t'}
\sum_{i=0}^k \binom{k}{i} (-1)^i  \sin\left( \frac{i p_0 t}{n} + \frac{j t}{n}\right) &=\Im \sum_{i=0}^k \binom{k}{i} (-1)^i  e\left (\frac{i p_0 t}{n} + \frac{jt}{n}\right ) \nonumber \\
&= \Im \left [\left (1-e\left(\frac{p_0 t}{n}\right)\right )^k e\left(\frac{jt}{n}\right) \right ] \nonumber\\
&\le (4 \pi/\sqrt{A})^k.
\end{align}
Also, as in \eqref{eqn:z_{q_0}:2:t}
\begin{align}\label{eqn:z_{q_0}:2:t'}
\sum_{i=0}^k \binom{k}{i} (-1)^i  \frac{j+ip_0}{n}\sin\left( \frac{i p_0 t}{n} + \frac{j t}{n}\right) &=\frac{\partial}{\partial t}\left (\sum_{i=0}^k \binom{k}{i} (-1)^i  \sin\left (\frac{i p_0 t}{n} + \frac{jt}{n}\right )\right ) \nonumber \\
&=\Im \left ( \frac{\partial}{\partial t} \left (\left (1-e\left(\frac{p_0 t}{n}\right)\right )^k e\left(\frac{jt}{n}\right) \right )\right )\nonumber \\
&< (4\pi)^{k}/A^{(k-3)/2}.
\end{align}
It thus follows that 
\begin{align*}
|\Delta^k m_{j, p_0}| &\le |\Delta^k {\psi}_{j+ l p_0 }|  + |\Delta^k ({\psi}_{j+ l p_0 } - m_{j+ l p_0})|\\  
 &\le 4\|\BD\|_2 \frac{(4\pi)^{k}}{A^{(k-3)/2}} + 4 \times 2^{k}\frac{1}{T},
\end{align*}
and similarly for $|\Delta^k m_{j, p_0}'|$.
\end{proof}

Now, with a sufficiently small $\tau_\ast$ (given $\tau$) and $\eps$ (given $\tau_\ast$) we choose 
\begin{equation}\label{eqn:A,k:value}
A=n^{\tau_\ast} \mbox{ and } k=16\lfloor \frac{\log \|\BD\|_2}{\eps \log n} \rfloor.
\end{equation}

We note that $k=O(1)$ because $\|\BD\|_2\le n^{C_\ast}$. (We restrict $\|\BD\|_2$ to grow polynomially here so that $k$ is bounded, which significantly simplifies our later analysis). We also recall that $p_0 \le A \le n^{\tau_\ast}$.

With these choices we see that \eqref{eqn:A,k} is fulfilled, and hence the RHS terms in Lemma \ref{lemma:diff:t} are strictly smaller than 1. But because the numbers $m_j, m_j'$ are integers, it follows that as long as  $\{j, j+p_0,\dots, j+k p_0\} \subset J=[a,b]$ we must have 
\begin{equation}\label{eqn:j:start}
\Delta^{k} m_{j, p_0}=0 \mbox{ and } \Delta^{k} m_{j, p_0}'=0.
\end{equation}
Now we consider the sequence of integers $\{m_{j+ i p_0, 0\le i\le (b-j)/p_0}\}$. By repeatedly applying \eqref{eqn:j:start} for $j=j, j+p_0,j+2p_0,\dots$ we see the $k$-differential of any $k+1$ consecutive terms of this sequence is zero. We thus deduce from here a crucial conclusion below.
\begin{lemma}\label{lemma:poly:t} For given $j\in J$, there exists a polynomial of degree at most $k-1$ so that for any $0\le i\le (b-j)/p_0$ we have 
$$m_{j+ i p_0 } = P_{j,p_0}(i).$$
We also have a similarly conclusion for $m'_{j+ip_0}$ where the polynomial can be different. 
\end{lemma}
We note that this result holds for any $j$ such that $[j,j+kp_0] \subset J$. Now we will exploit these polynomial properties furthermore by specifying the choices of parameters. We first consider the case that $\|\BD\|_2$ is small.

{\bf Case 1.} Assume that $\|\BD\|_2 \le n^{1-4\tau}$. It suffices to work with the $\psi_j$ sequence, the treatment for $\psi_j'$ is identical. Here we will show that $P_j$ is a constant. Indeed, otherwise then as $P'_j$ has at most $k-2$ roots, there is an interval of length $|J|/k$ where $P_j$ is strictly monotone. But on this interval (of length $|J|/k \ge n/kT^6 \ge n^{1-6 \tau_\ast -o(1)}$ at least), $m_j \in [-2n^{1-4\tau}, 2n^{1-4\tau}]$ because $|m_j| \le 2\|\BD\|_2$, so this is impossible if $\tau_\ast$ is sufficiently small compared to $\tau$. Thus we have shown that the polynomials are constant, 
$$m_{j+ l p_0 } = m_j, \forall j,l\in \Z, [j+ l p_0, j+(l+k) p_0 ] \subset J=[a,b].$$

Let's next fix $j$, then the range for $l$ is $I=(a-j)/p_0 \le l \le (b-j)/p_0-O(1)$, which is an interval of length of order $n^{1-7\tau_\ast}$. On this range of $l$, the condition of $t$ in \ref{cond:t} shows that $ \psi_{j+lp_0} =d_1 \cos((j+ l p_0)t) -d_2 \frac{i}{n} \sin ((j+ l p_0) t)$ changes size. But as $m_{j+ l p_0}=m$ is the common integral part for all $l$, this is impossible unless $m_{j+l p_0}=0$ for all $l$ in the range $I$ above. 

Argue similarly for $\psi_j'$, we have thus obtained in this case that all the integral parts are zero, and hence 
$$\sup_{l\in I_L} |\psi_{j+lp_0}| = \sup_{l\in I} |\psi_{j+l p_0}|_{\R/\Z}, \sup_{l\in I} |\psi_{j+lp_0}'| = \sup_{l\in I} |\psi_{j+l p_0}'|_{\R/\Z}  \le 1/T = n^{-\tau_\ast}.$$ 
It follows that
\begin{align*}
\sum_i \|\langle \Bu_i,   \BD \rangle\|_{\R/\Z}^2+\|\langle \Bu_i',   \BD \rangle\|_{\R/\Z}^2 &\ge \sum_{l\in I} \|\langle \Bu_{j+lp_0},   \BD \rangle\|_{\R/\Z}^2+\sum_{l\in I} \|\langle \Bu_{j+lp_0}',   \BD \rangle\|_{2}^2\\ 
&= \sum_{l\in I} \|\langle \Bu_{j+lp_0},   \BD \rangle\|_{2}^2+\sum_{l\in I} \|\langle \Bu_{j+lp_0}',   \BD \rangle\|_{2}^2\\
&= \|\BD\|_2^2 \sum_{l\in I} \langle \Bu_{j+lp_0}, \Be\rangle ^2+ \langle \Bu_{j+lp_0}', \Be\rangle ^2 \ge r^2 n^{1-8\tau}>n^{2\tau},
\end{align*}
where we used Claim \ref{claim:t}. This bound clearly contradicts \eqref{eqn:contradiction:t}.

{\bf Case 2.} Assume that $ n^{1-4\tau} \le \|\BD\|_2 \le n^{C_\ast}$. Roughly speaking, our approach here is of inverse-type in the sense that we will try to gain as much as possible information for $t$ given the obtained bounds; and our final result on $t$ is almost optimal.  

Recall that $\sum_{j\in J} \|\psi_j\|_{\R/\Z}^2 + \|\psi_j'\|_{\R/\Z}^2\le T$ and by Cauchy-Schwarz we have
$$\sum_{j\in J} \|\psi_j\|_{\R/\Z} + \|\psi_j'\|_{\R/\Z} \le 2\sqrt{n T}.$$
We will reapply the process from \eqref{eqn:z_{q_0}:1:t} and \eqref{eqn:z_{q_0}:2:t} with $q_0=l p_0$ (for a given positive integer $l$). By the polynomial properties of the $m_{j+ip_0}=P(i)$ for a polynomial $P$ of degree at most $k-1$, we also have that 
$$\sum_{i=0}^k \binom{k}{i} (-1)^i m_{j+ iq_0} =\sum_{i=0}^k \binom{k}{i} (-1)^i m_{j+ ilp_0}=\sum_{i=0}^k \binom{k}{i} (-1)^i P_{j,p_0}(il) =0.$$
Set 
\begin{equation}\label{eqn:def:z_{q_0}}
z_{q_0}:=z(t, q_0) = e(q_0 t/n).
\end{equation}
We then write as follows
\begin{align*}
|\Delta^k {\psi}_{j,q_0}| &= |\sum_{i=0}^k \binom{k}{i} (-1)^i {\psi}_{j+ iq_0}| = |\sum_{i=0}^k \binom{k}{i} (-1)^i ({\psi}_{j+ iq_0}-m_{j+iq_0})+ \sum_{i=0}^k \binom{k}{i} (-1)^i m_{j+iq_0})|\\ 
&= |\sum_{i=0}^k \binom{k}{i} (-1)^i ({\psi}_{j+ iq_0}-m_{j+iq_0})|\le 2^k \sum_{i=0}^k \|\psi_{j+iq_0}\|_{\R/\Z}.
\end{align*}
Similarly,
\begin{align*}
|\Delta^k {\psi}_{j,q_0}'| &= |\sum_{i=0}^k \binom{k}{i} (-1)^i {\psi}_{j+ iq_0}'| = |\sum_{i=0}^k \binom{k}{i} (-1)^i ({\psi}_{j+ iq_0}'-m_{j+iq_0}')+ \sum_{i=0}^k \binom{k}{i} (-1)^i m_{j+iq_0}')|\\ 
&= |\sum_{i=0}^k \binom{k}{i} (-1)^i ({\psi}_{j+ iq_0}'-m_{j+iq_0}')|\le 2^k \sum_{i=0}^k \|\psi_{j+iq_0}'\|_{\R/\Z}.
\end{align*}
On the other hand, 
\begin{align*}
\Delta^k {\psi}_{j,q_0} &= \sum_{i=0}^k \binom{k}{i} (-1)^i {\psi}_{j+ iq_0} = d_1 \sum_{i=0}^k \binom{k}{i} (-1)^i  \cos( \frac{i q_0 t}{n} + \frac{j t}{n}) + d_2\sum_{i=0}^k \binom{k}{i} (-1)^i  \frac{j+iq_0}{n}\sin( \frac{i q_0 t}{n} + \frac{jt}{n})\\
&= d_1\Re \Big((1-z_{q_0})^k e(jt/n)\Big) +d_2 \Re \Big( -k \sqrt{-1} \frac{q_0}{n} \big((1-z_{q_0})^{k-1} e(jt/n)\big)+ \sqrt{-1}j/n \big((1-z_{q_0})^k e(jt/n)  \Big) 
\end{align*}
and
\begin{align*}
\Delta^k {\psi}_{j,q_0}' &= \sum_{i=0}^k \binom{k}{i} (-1)^i {\psi}_{j+ iq_0}' = d_1 \sum_{i=0}^k \binom{k}{i} (-1)^i  \sin( \frac{i q_0 t}{n} + \frac{j t}{n}) + d_2\sum_{i=0}^k \binom{k}{i} (-1)^i  \frac{j+iq_0}{n}\cos( \frac{i q_0 t}{n} + \frac{jt}{n})\\
&= d_1\Im \Big((1-z_{q_0})^k e(jt/n)\Big) +d_2 \Im \Big( -k \sqrt{-1} \frac{q_0}{n} \big((1-z_{q_0})^{k-1} e(jt/n)\big)+ \sqrt{-1}j/n \big((1-z_{q_0})^k e(jt/n)  \Big) .
\end{align*}
By the triangle inequality, the above then implies that as long as $[j,j+kq_0] \subset J$ we have
\begin{align}\label{eqn:jt'}
& \Big | (1-z_{q_0})^{k-1}  \Big( d_1(1-z_{q_0}) + \sqrt{-1} d_2 \frac{-k q_0}{n} + \sqrt{-1} d_2 (1-z_{q_0}) \frac{j}{n} \Big)| \nonumber \\
&\le 2^k \sum_{i=0}^k \|\psi_{j+iq_0}\|_{\R/\Z} + \|\psi_{j+iq_0}'\|_{\R/\Z}.
\end{align}
We recall that this holds for any $j\in J$ as long as $[j,j+kq_0]\subset J$, and there is no $P_{j,p_0}(.)$ or $m_{j+iq_0}$ on the LHS. Applying the estimate for $j=j+L$, and using triangle inequality we obtain the following.

\begin{lemma} Assume that $[j, j+L+kq_0]  \subset J$. We then have 
 \begin{align}\label{eqn:jt}
& \big | (1-z_{q_0})^{k}\big |  \frac{d_2L}{n} \nonumber \\
&\le 2^k \sum_{i=0}^k \|\psi_{j+iq_0}\|_{\R/\Z} + \|\psi_{j+iq_0}'\|_{\R/\Z}+2^k \sum_{i=0}^k \|\psi_{j+L+iq_0}\|_{\R/\Z} + \|\psi_{j+L+iq_0}'\|_{\R/\Z}.
\end{align}
\end{lemma}

Now we complete the proof of Proposition \ref{prop:largepsi:t}. We will work with $q_0=lp_0$ so that $kq_0\le |J|/2$. Recall that $J=[a,b]$ with length $|J|$ at least $n/T^6 = n^{1-6\tau_\ast}$, and hence $(a+b)/2 \ge |J|/2 \ge n^{1-6\tau_\ast}/2$. We divide the treatment into two cases depending on the parameters $d_1,d_2$.

{\bf Subcase 1.} Assume that $|d_2|\ge n^{-\tau_\ast}$. We first fix $q_0=lp_0$ such that 
\begin{equation}\label{eqn:q_0value:t}
|q_0|\le |J|/8k.
\end{equation}
Choose any $L \in [|J|/8, |J|/4]$, say $L= |J|/8$. With this fixed $L$, it is clear that $|d_2L/n| \ge n^{-\tau_\ast} n^{1-7\tau_\ast}/n \ge n^{-8\tau_\ast}$. 
We will then choose $j$ in the interval $j\in [(a+b)/2 - |J|/8, (a+b)/2+|J|/8]$ such that the RHS of \eqref{eqn:jt} is as small as possible. (Clearly with these choices we have $[j,j+L+kq_0] \subset J$.)
If we sum of the RHS of \eqref{eqn:jt} for $j$ from the above range, then this sum is bounded by  $O_k(\sum_{i\in J} \|\psi_i\|_{\R/\Z} + \|\psi_i'\|_{\R/\Z}) = O(\sqrt{n T}) = O(n^{1/2+\tau_\ast})$ because each term $\|\psi_i\|_{\R/\Z}$ appears a bounded number of times in the total sum. Hence by averaging, there exists $j\in  [(a+b)/2 - |J|/8, (a+b)/2+|J|/8]$ such that the RHS of \eqref{eqn:jt} is bounded by $O(n^{1/2+\tau_\ast}/(|J|/4)) = O(n^{-1/2+7\tau_\ast})$.

Putting together, with such choice of $L$ and $j$, the equation \eqref{eqn:jt} implies that
$$|(1-z_{q_0})^{k}| \times n^{-8\tau_\ast} \le O(n^{-1/2+7\tau_\ast}).$$
Thus we have that $|(1-z_{q_0})^{k}|  \le n^{-1/2+15\tau_\ast}$ in this case, and so with sufficiently small $\tau_\ast$
\begin{equation}\label{eqn:q_0:t}
|1-z_{q_0}|  \le n^{-1/3k}.
\end{equation}

As of this point, recall that $z_{q_0} = e(q_0 t/n)$. Because \eqref{eqn:q_0:t} holds for any $q_0=lp_0$ satisfying \eqref{eqn:q_0value:t}, we thus have for all $1\le l \le n^{1-8\tau_\ast} \le |J|/(8kp_0)$ 
\begin{equation}\label{eqn:q_0':t}
|1-e(lp_0 t/n)|  \le n^{-1/3k}.
\end{equation}
As of this point we then use the following elementary result to obtain more information on $t$.
\begin{claim}\label{claim:dividing} Assume that $z=e(\theta), |\theta|\le \pi/8 $ such that for all $1\le l \le M$ we have $|1-z^l| \le 1/32$ for a sufficiently large $M$. Then $|\theta|=O(1/M)$.
\end{claim}
\begin{proof} By assumption, $|\theta| \le \pi/8$ and $\|2^m\theta\|_{\R/\Z} \le \pi/8$ for all $1\le m\le \log M$, and so we can repeatedly estimate $|\theta|$ to obtain $|\theta| =O(1/M)$. 
\end{proof}
Claim \ref{claim:dividing} and \eqref{eqn:q_0':t} then implies that for large enough $n$,
$$\|p_0t/\pi n\|_{\R/\Z}=O(n^{-1+8\tau_\ast}).$$
However this contradicts Condition \ref{cond:t} because $p_0 \le A =n^{\tau_\ast}$ and $\tau_\ast$ is sufficiently small given $\tau$.

{\bf Subcase 2.} Now we consider the remaining (very degenerate) case that $|d_2| \le n^{-\tau_\ast}$. Then $|d_1| \asymp \|\BD\|_2 \ge n^{1-4\tau}$. In the case that $d_1|z_{q_0}-1| \le n^{4 \tau_\ast}$ then we have 
$$|z_{q_0}-1| \le n^{-1+4\tau-4\tau_\ast}.$$ 
In the other case that $d_1|z_{q_0} -1| \ge n^{4\tau_\ast}$, then this term dominates all other terms involving $d_2$ on the LHS of \eqref{eqn:jt'} (because each of which has order $O(1)$ as $d_2$ is small). So we have 
\begin{align*}
|(1-z_{q_0})^{k}| d_1/2 & \le 2^k \sum_{i=0}^k \|\psi_{j+iq_0}\|_{\R/\Z} + 2\|\psi_{j+iq_0}'\|_{\R/\Z}=O(1).
\end{align*} 
Hence,
$$|z_{q_0}-1| =O(1/d_1^{1/k})  \le n^{-1/2k}.$$ 
Thus from both scenarios on the magnitude of $d_1|z_{q_0}-1|$ we at least have $|z_{q_0}-1| \le n^{-1/2k}$. We can then repeat the argument as in the previous case to vary $q_0$ and use Claim \ref{claim:dividing}. It thus also follows that $\|p_0t/ \pi n\|_{\R/\Z}=O(n^{-1+8\tau_\ast})$, which is again impossible.

Before concluding this section, as our approach to prove Proposition \ref{prop:largepsi:t} starts with \eqref{eqn:psi:small:t}, by considering subintervals of $J$ when needed (where we note that at least one of such subintervals still has length $\Omega(n/T^6)$), we obtain the following analog of Theorem \ref{thm:fourier:2}.

\begin{theorem}\label{thm:fourier:2'} Let $C_\ast$ and $\ell$ be given positive constants, and $t$ satisfies Condition \ref{cond:t} for some sufficiently small constant $\tau$. The following holds for sufficiently large $n$ and sufficiently small $\tau_\ast$ (depending on $C_\ast, \tau$ and $\ell$). For any $n^{5\tau -1/2} \le \|x\|_2 \le n^{C_\ast}$, for any set $I\subset [n]$ of at most $\ell$ entries we have 
$$|\prod_{i \notin I} \phi_{i}(x)| \le \exp(-n^{\tau_\ast}).$$
\end{theorem}

\section{Characteristic functions in $\R^4$, proof of Theorem \ref{thm:fourier}}\label{section:fourier}
In this section we continue our ``inverse-type" analysis of the characteristic functions, but now in $\R^4$. As there are two parameters $s,t$ from Condition \ref{cond:s,t}, the situation is much more complicated, but we will call on the previous sections whenever possible.

First we show the following analog of Claim \ref{claim:t}.

\begin{claim}\label{claim:s,t} Assume that $\tau_\ast$ is sufficiently small given $\tau$,  and assume that $s, t$ satisfy Condition \ref{cond:s,t}. Let $I=\{a+lq, 0\le l \le L\} \subset [n]$ be any arithmetic progression of length $L=n^{1-6 \tau_\ast}$. Then

\begin{enumerate} 
\item For all $\eps_1,\eps_2,\eps_3, \eps_4 \in \{-1,1\}$ and $\theta\in [-\pi, \pi]$, and for any positive integer $A_0=O(n^{\tau_\ast})$ there exists $i\in I$ so that  
$$\eps_1 \sin(i A_0 s/n + \theta), \eps_2 \cos (i A_0s/n+ \theta) ,\eps_3 \sin(i A_0t/n + \theta), \eps_4 \cos (i A_0 t/n+ \theta ) >0.$$
\item For any unit vector $\Be \in \R^4$ we have 
\begin{equation}\label{eqn:Be}
\sum_{i\in I} \langle \Be, \Bv_i\rangle^2 \ge n^{1-\tau} \mbox{ and } \sum_{i\in [n]} \langle \Be, \Bv_i'\rangle^2 \ge n^{1-\tau}.
\end{equation}
\end{enumerate}
\end{claim}

\begin{proof} Note that $q \le n/L \le n^{6\tau_\ast}$, which is much smaller than $n^\tau$ if $\tau_\ast$ is chosen sufficiently small. We first show \eqref{eqn:Be} for $\Bv_i$, given any unit vector $\Be = (x_1,x_2,x_3,x_4)$. First, replacing $s,t$ by $qs,qt$ and rotate $\Be$ if need, without loss of generality we assume that $I=[0,L]$. Then 
\begin{align*}
\sum_{i\in I} \lang \Be, \Bv_i \rang^2 &= \sum_i (x_1 \cos(it/n) - x_2 (i/n) \sin(i t/n) + x_3 \cos(i s/n) - x_4 (i/n) \sin(is/n) )^2 
\end{align*}

The sum over the diagonal terms, under Condition \ref{cond:s,t} (and hence Condition \ref{cond:t}) for $s$ and $t$ separately, is of order
 $L^3/n^2 \ge n^{1-\tau}$. We thus need to work with the cross terms
$$A_1 = \sum_{i\in I} (i/n) \sin(it/n) \cos(it/n), B_1 =  \sum_{i\in I}  (i/n) \sin(is/n) \cos(is/n)$$
and
$$C_1 = \sum_{i\in I} \cos(is/n) \cos(it/n), D_1 =  \sum_{i\in I}  (i/n)^2 \cos(is/n) \cos(it/n)$$ 
and
$$E_1 = \sum_{i\in I} \sin(is/n) \sin(it/n), F_1 =  \sum_{i\in I}  (i/n)^2 \sin(is/n) \sin(it/n).$$

We show that under the condition on $s,t$ (and $(qs,qt)$) from the claim these terms are all of order $o(n^{1-\tau})$ (actually we obtain a slightly strong bound).  First for $A_1$ and $B_1$, with $t'=t/n$ we write
\begin{align*}
|A_1| = |\sum_{i\in I} (i/n)  \sin (it') \cos(it')| &=  |\frac{1}{2}\sum_{i\in I} (i/n) \sin ( 2 it')| = \frac{1}{4n}  | \frac{\partial}{\partial t'}(\sum_{i\in I} \cos( 2 it'/n))| \\
&= \frac{1}{4n}   |  \frac{\partial}{\partial t'} \Re[ \sum_{i=0}^L e( 2 it')]|=  \frac{1}{4n}   |  \frac{\partial}{\partial t'} \Re[\frac{e(2(L+1)t')-1}{e(2t')-1}]|.
\end{align*}
After some simplifications we obtain 
\begin{align*}
|A_1| &=O\Big(\frac{L}{n}\frac{1}{|\sin t'|} +  \frac{1}{n} \frac{1}{(\sin t')^2} \Big) = O(n^{1-8\tau}),
\end{align*}
where we used Condition \ref{cond:s,t} (more precisely Condition \ref{cond:t}) that $\|t'/\pi\|_{\R/ \Z} \ge n^{-1+8\tau}$.

Let's next work with the terms involving $s,t$, such as $D_1$. We have 
\begin{align*}
2D_1 &= \sum_{i\in I} (i/n)^2 (\cos(is/n - it/n) + \cos(is/n+it/n)) = (-\sum_{i\in I} \cos(ix/n))^{''}|_{x=s-t} + (-\sum_{i\in I} \cos(ix/n))^{''}|_{x=s+t}\\
&=(\frac{\sin((L+1/2)x/n)}{\sin(x/2n)}-1)^{''}|_{x=s-t} +(\frac{\sin((L+1/2)x/n)}{\sin(x/2n)}-1)^{''}|_{x=s+t}= O\Big(\frac{1}{|\sin^3((s-t)/2n)|} +\frac{1}{|\sin^3((s+t)/2n)|}\\
&+ \frac{1}{n}\frac{1}{|\sin^2((s-t)/2n)|} +   \frac{1}{n} \frac{1}{|\sin^2((s+t)/2n)|} +  \frac{1}{n^2} \frac{1}{|\sin((s-t)/2n)|} +  \frac{1}{n^2}\frac{1}{|\sin^2((s+t)/2n)|}\Big).
\end{align*}
It thus follows that, because $|(s-t)/ \pi n|_{\R/ \Z} \ge n^{-1+8\tau}$ and $|(s+t)/ \pi n|_{\R/\Z}\ge n^{-1+8\tau}$, 
$$|D_1| = O(n^{1-8\tau}).$$
The treatments for $C_1, E_1, F_1$ are somewhat simpler, and hence we omit.

Now we focus on the first part. By the (quantitative) Weyl's  equi-distribution criterion on $\T^2$ (see for instance \cite[Proposition 9; Exercises 18, 19]{T}) if the sequence $\{(i(qA_0s/\pi n)+\theta, i(qA_0t/\pi n)+\theta), 1\le i \le L\}$ in the two dimensional torus $(\R/\Z)^2$, where $A_0=O(n^{\tau_\ast})$ and $q\le n^{6\tau_\ast}$, is not $\delta$-equidistributed \footnote{Actually we just need the points to appear in all four quadrants of the plane.} (for some fixed sufficiently small constant $\delta$ to guarantee the sign changes) then  there exist positive integers $k_0,l_0 = n^{O_{\tau_\ast}(1)}$ such that 
$$\|k_0 (A_0qs/ \pi n)+ l_0(A_0 qt/\pi n)\|_{\R/\Z} = O(\frac{n^{O_{\tau_\ast}(1)}}{L}) = O(\frac{1}{n^{1-8\tau}}),$$
provided that $\tau_\ast$ is sufficiently small compared to $\tau$. This contradicts with our condition.
\end{proof}



Let $\BD=(d_1,d_2,d_3,d_4) \in \R^4$ be any non-zero vector and let $\Be$ be the unit vector in the direction of $\BD$. For $s,t \in [-n\pi, n\pi]$ we define   
 \begin{equation}\label{eqn:psi1}
 \psi_i :=  \langle \BD, \Bv_i \rangle = d_1 \cos(it/n) -d_2 \frac{i}{n} \sin (it/n) + d_3 \cos(is/n) -d_4 \frac{i}{n} \sin (is/n)
 \end{equation}
 \begin{equation}\label{eqn:psi1'}
\psi_i':=\langle \BD, \Bv_i' \rangle = d_1 \sin(it/n) +d_2 \frac{i}{n} \cos (it/n) + d_3 \sin(is/n) + d_4 \frac{i}{n} \cos (is/n).
\end{equation}
Define $T= n^{\tau_\ast}$. Our key ingredient in the proof of Theorem \ref{thm:fourier} is the following analog of Proposition \ref{prop:largepsi:t}.

\begin{prop}\label{prop:largepsi} Assume that $s,t$ satisfy Condition \ref{cond:s,t}. We have
$$\sum_j \|\psi_j\|_{\R/\Z}^2 + \sum_j \|\psi_j'\|_{\R/\Z}^2 \ge T.$$
\end{prop}
It is clear this result implies Theorem \ref{thm:fourier} via \eqref{eqn:cahr:bound}, \eqref{eqn:decoupling} and \eqref{eqn:y}. For the rest of this section we prove Proposition \ref{prop:largepsi} by contradiction: assume the opposite that we have
\begin{equation}\label{eqn:contradiction}
\sum_j \|\psi_j\|_{\R/\Z}^2 +\|\psi_j'\|_{\R/\Z}^2 \le T.
\end{equation}
We will then show that this is impossible as long as $s,t$ satisfy Condition \ref{cond:s,t}. First, argue as in \eqref{eqn:psi:small:t}, it follows from \eqref{eqn:contradiction} that there exists an interval $J=[a,b] \subset [n]$ of length $n/T^6$ so that for $j\in J$
\begin{equation}\label{eqn:psi:small}
\|\psi_j\|_{\R/\Z}+\|\psi_j'\|_{\R/\Z} <1/T.
\end{equation}

{\bf Differencing.} Let $A,k$ be chosen as in \eqref{eqn:A,k}. We then can find $p_0\in \Z, p_0 \neq 0$ and $s_0, t_0$ so that  
\begin{equation}\label{eqn:approx:pq}
p_0 \frac{s}{2\pi n} -s_0 \in \Z, p_0 \frac{t}{2\pi n} -t_0 \in \Z, 1\le |p_0| \le A, |s_0|^2 + |t_0|^2 \le \frac{4}{A}.
\end{equation}
Indeed, consider the sequence of pairs $(\{q (s/2\pi n)\} ,\{q (t/2\pi n)\}), 1\le q \le A$ in $[0,1]^2$. Using Dirichlet's principle, there exists $q_1,q_2$ such that the distance of the pairs is at most $2/\sqrt{A}$. In other words,
$$|\{q_1 (s/2\pi n)\} - \{q_2 (s/2\pi n)\} |^2 + |\{q_1 (t/2\pi n)\} -\{q_2 (t/2\pi n)\}|^2 \le \frac{4}{A}.$$
This implies that there exists integers $p_1,p_2$ such that 
$$|(q_1-q_2)(s/2\pi n)-p_1|^2 + |(q_1-q_2)(t/2\pi n)-p_2|^2 \le  \frac{4}{A}.$$
Set $s_0 = (q_1-q_2)(s/2\pi n)-p_1$ and $t_0 = (q_1-q_2)(t/2\pi n)-p_1$ and $p_0 =q_1-q_2$ we obtain as claimed.

From the approximation we infer that 
\begin{equation}\label{eqn:closeone}
|e(p_0 \frac{t}{n}) -1|= |e(2\pi t_0)-1| \le |2\sin(\pi t_0)| \le 4 \pi/\sqrt{A}
\end{equation}
and
\begin{equation}\label{eqn:closeone'}
|e(p_0 \frac{s}{n}) -1|= |e(2\pi s_0)-1| \le |2\sin(\pi s_0)| \le 4 \pi/\sqrt{A}.
\end{equation}
 
Next, consider the differential operation $\Delta^k g_{j,p_p}$ as in the previous section. Let $m_j$ and $m_j'$ be the integers closest to $\psi_j$ and $\psi_j'$ respectively. We then have
\begin{lemma}\label{lemma:diff} We have
\begin{equation}\label{eqn:Delta:diff}
|\Delta^k m_{j,p_0}|  + |\Delta^k m_{j,p_0}'| \le 4\|\BD\|_2 \frac{(4\pi)^{k}}{A^{(k-3)/2}} + 4 \times 2^{k}\frac{1}{T}
\end{equation}
provided that $[j, j+ k p_0 ] \subset J$.
\end{lemma}
\begin{proof}[Proof of Lemma \ref{lemma:diff}] Argue similarly as in the proof of Lemma \ref{lemma:diff:t}, by using \eqref{eqn:closeone} and \eqref{eqn:closeone'} we  can show that
$$|\Delta^k {\psi}_{j}| = |\sum_{i=0}^k \binom{k}{i} (-1)^i {\psi}_{j+ ip_0}| \le (|d_1|+|d_3|) (\frac{4 \pi}{\sqrt{A}})^k + (|d_2|+|d_4|)\frac{(4\pi)^{k}}{A^{(k-3)/2}} < 4\|\BD\|_2 \frac{(4\pi)^{k}}{A^{(k-3)/2}}.$$
One can also obtain similarly estimates for $|\Delta^k {\psi}_{j}'|$. It thus follows that 
\begin{align*}
|\Delta^k m_{j}| &\le |\Delta^k {\psi}_{j+ l p_0 }|  + |\Delta^k ({\psi}_{j+ l p_0 } - m_{j+ l p_0})|\\  
 &\le 4\|\BD\|_2 \frac{(4\pi)^{k}}{A^{(k-3)/2}} + 4 \times 2^{k}\frac{1}{T},
\end{align*}
and similarly for $|\Delta^k m_{j}'|$.
\end{proof}

As of this point, we choose $A$ and $k$ as in \eqref{eqn:A,k:value}. Then as long as  $\{j, j+p_0,\dots, j+k p_0\} \subset J=[a,b]$ we must have 
\begin{equation}\label{eqn:j:start}
\Delta^{k} m_{j, p_0}=0 \mbox{ and } \Delta^{k} m_{j, p_0}'=0.
\end{equation}
This results into the following analog of Lemma \ref{lemma:poly:t}.
\begin{lemma}\label{lemma:poly:st} For given $j\in J$, there exists a polynomial of degree at most $k$ so that for any $0\le i\le (b-j)/p_0$ we have 
$$m_{j+ i p_0 } = P_{j,p_0}(i).$$
We also have a similarly conclusion for $m'_{j+ip_0}$. 
\end{lemma}
We again note that this result holds for any $j$ such that $[j,j+kp_0] \subset J$. We next consider the case that $\|\BD\|_2$ is small.

{\bf Case 1.} Assume that $\|\BD\|_2 \le n^{1-4\tau}$. Here our treatment is identical to Case 1. of the previous proof, that we can deduce from here that $m_{j+ l p_0 } = m_j, \forall j,l\in \Z, [j+ l p_0, j+(l+k) p_0 ] \subset J=[a,b]$ and as over the interval $I=(a-j)/p_0 \le l \le (b-j)/p_0-O(1)$, the condition of $s,t$ in \ref{cond:s,t} shows that $ \psi_{j+lp_0} =d_1 \cos((j+ l p_0)t) -d_2 \frac{i}{n} \sin ((j+ l p_0) t)+ d_3 \cos((j+ l p_0) s/n) -d_4 \frac{(j+ l p_0)}{n} \sin ((j+ l p_0)s/n)$ changes size and this implies that $m_{j+l p_0}=0$ for all $l$ in the range $I$ above. 

One will then have 
\begin{align*}
\sum_i \|\langle \Bv_i,   \BD \rangle\|_{\R/\Z}^2+\|\langle \Bv_i',   \BD \rangle\|_{\R/\Z}^2 &\ge \sum_{l\in I} \|\langle \Bu_{j+lp_0},   \BD \rangle\|_{\R/\Z}^2+\sum_{l\in I} \|\langle \Bv_{j+lp_0}',   \BD \rangle\|_{2}^2\\ 
&= \sum_{l\in I} \|\langle \Bv_{j+lp_0},   \BD \rangle\|_{2}^2+\sum_{l\in I} \|\langle \Bv_{j+lp_0}',   \BD \rangle\|_{2}^2\\
&= \|\BD\|_2^2 \sum_{l\in I} \langle \Bv_{j+lp_0}, \Be\rangle ^2+ \langle \Bv_{j+lp_0}', \Be\rangle ^2 \ge r^2 n^{1-8\tau}>n^{2\tau},
\end{align*}
where we used the second point of Claim \ref{claim:s,t}. This bound contradicts \eqref{eqn:contradiction}.

{\bf Case 2.} Assume that $ n^{1-4\tau} \le \|\BD\|_2 \le n^{C_\ast}$. In this case our ``inverse" analysis is more complicated than that of the previous section because there are two unknowns. Nevertheless, our final bounds are almost optimal.

Recall that $\sum_{j\in J} \|\psi_j\|_{\R/\Z}^2 + \|\psi_j'\|_{\R/\Z}^2\le T$ and by Cauchy-Schwarz we have
$$\sum_{j\in J} \|\psi_j\|_{\R/\Z} + \|\psi_j'\|_{\R/\Z} \le 2\sqrt{n T}.$$
By reapplying the differential process as in Case 2 of Section \ref{section:fourier:2} (this time for $s$ and $t$) with $q_0=l p_0$ for a given positive integer $l$. Set 
\begin{equation}\label{eqn:def:z_{q_0}:st}
z_{q_0,t}= e(q_0 t/n) \mbox{ and } z_{q_0,s}= e(q_0 s/n).
\end{equation}
By the polynomial properties of the $m_{j+ip_0}=P(i)$ for a polynomial $P$ of degree at most $k-1$, we the obtain the following analog of \eqref{eqn:jt'}.
\begin{align}\label{eqn:jst}
& \Big | (1-z_{q_0,t})^{k-1} e(jt/n)  \Big( d_1(1-z_{q_0,t}) + \sqrt{-1} d_2 \frac{-k q_0}{n} + \sqrt{-1} d_2 (1-z_{q_0,t}) \frac{j}{n} \Big) \nonumber \\
&+(1-z_{q_0}(s))^{k-1} e(js/n)  \Big( d_3(1-z_s) + \sqrt{-1} d_4 \frac{-k q_0}{n} + \sqrt{-1} d_4 (1-z_s) \frac{j}{n} \Big) \Big| \nonumber \\
&\le 2^k \sum_{i=0}^k \|\psi_{j+iq_0}\|_{\R/\Z} + \|\psi_{j+iq_0}'\|_{\R/\Z}.
\end{align}
In what follows, for a fixed $q_0$, we will take advantage of the above inequality by varying $j$ and eliminate the terms involving $(1-z_s)$. Let $L$ be a parameter to be chosen, we show the following
\begin{lemma} Assume that $[j_1, j_1+j+L+kq_0]  \subset J=[a,b]$. We then have 
 \begin{align*}
|(1-z_{q_0,t})^k| &\times (d_2\frac{j}{n}) |(e(L(s-t)/n)-1)|^2 \nonumber \\
&\le 2^k \sum_{i=0}^k \|\psi_{j_1+iq_0}\|_{\R/\Z} + 2\|\psi_{j_1+L+iq_0}\|_{\R/\Z}+  \|\psi_{j_1+2L+iq_0}\|_{\R/\Z} \nonumber  \\
& + 2^k \sum_{i=0}^k  \|\psi_{j_1+iq_0}'\|_{\R/\Z} + 2\|\psi_{j_1+L+iq_0}'\|_{\R/\Z}+  \|\psi_{j_1+2L+iq_0}'\|_{\R/\Z} \nonumber  \\
& + 2^k \sum_{i=0}^k \|\psi_{j_1+j+iq_0}\|_{\R/\Z} + 2\|\psi_{j_1+j+L+iq_0}\|_{\R/\Z}+  \|\psi_{j_1+j+2L+iq_0}\|_{\R/\Z} \nonumber  \\
& + 2^k \sum_{i=0}^k  \|\psi_{j_1+j+iq_0}'\|_{\R/\Z} + 2\|\psi_{j_1+j+L+iq_0}'\|_{\R/\Z}+  \|\psi_{j_1+j+2L+iq_0}'\|_{\R/\Z}.
\end{align*} 
We also have the same bound for $|(1-z_{q_0,s})^{k}| \times (d_4 \frac{j}{n}) |(e(L(s-t)/n)-1)|^2 $.
\end{lemma}
\begin{proof} In what follows we note that if $[j_1, j_1+j+L+kq_0]  \subset J$ then automatically $[j_1,j_1+kq_0] \cup [j_1+L, j_1+L+kq_0] \cup [j_1+j,j_1+j+kq_0] \cup [j_1+j+L, j_1+j+L+kq_0] \subset J$, and so we can apply the followings as the indices are all in $J$. First, multiply both side of \eqref{eqn:jst} by $e(Ls/n)$
 \begin{align*}
& \Big | (1-z_{q_0,t})^{k-1} e(Ls/n) e(jt/n)  \Big( d_1(1-z_{q_0,t}) + \sqrt{-1} d_2 \frac{-k q_0}{n} + \sqrt{-1} d_2 (1-z_{q_0,t}) \frac{j}{n} \Big) \\
&+(1-z_{q_0,s})^{k-1} e((j+L)s/n)  \Big( d_3(1-z_{q_0,s}) + \sqrt{-1} d_4 \frac{-k q_0}{n} + \sqrt{-1} d_4 (1-z_{q_0,s}) \frac{j}{n} \Big) \Big | \\
&\le 2^k \sum_{i=0}^k \|\psi_{j+iq_0}\|_{\R/\Z} + \|\psi_{j+iq_0}'\|_{\R/\Z}.
\end{align*}

One the other hand, \eqref{eqn:jst} applied for $j$ being replaced by $j+L$ shows 
\begin{align*}
& \Big | (1-z_{q_0,t})^{k-1} e((j+L)t/n)  \Big( d_1(1-z_{q_0,t}) + \sqrt{-1} d_2 \frac{-k q_0}{n} + \sqrt{-1} d_2 (1-z_{q_0,t}) \frac{j+L}{n} \Big) \nonumber \\
&+(1-z_{q_0,s})^{k-1} e((j+L)s/n)  \Big( d_3(1-z_{q_0,s}) + \sqrt{-1} d_4 \frac{-k q_0}{n} + \sqrt{-1} d_4 (1-z_{q_0,s}) \frac{j+L}{n} \Big) \Big| \nonumber \\
&\le 2^k \sum_{i=0}^k \|\psi_{j+L+iq_0}\|_{\R/\Z} + \|\psi_{j+L+iq_0}'\|_{\R/\Z}.
\end{align*}
By the triangle inequality, it follows from these two inequalities that 
\begin{align}\label{eqn:jst'}
& \Big | (1-z_{q_0,t})^{k-1} e(Ls/n) e(jt/n)  \Big( d_1(1-z_{q_0,t}) + \sqrt{-1} d_2 \frac{-k q_0}{n} + \sqrt{-1} d_2 (1-z_{q_0,t}) \frac{j}{n} \Big) \nonumber \\
&-(1-z_{q_0,t})^{k-1} e((j+L)t/n)  \Big( d_1(1-z_{q_0,t}) + \sqrt{-1} d_2 \frac{-k q_0}{n} + \sqrt{-1} d_2 (1-z_{q_0,t}) \frac{j+L}{n} \Big) \nonumber \\
&-(1-z_{q_0,s})^{k-1} e((j+L)s/n) \sqrt{-1} d_4 (1-z_{q_0,s}) \frac{L}{n} \Big | \nonumber \\
&\le 2^k \sum_{i=0}^k \|\psi_{j+iq_0}\|_{\R/\Z} + \|\psi_{j+iq_0}'\|_{\R/\Z}+ 2^k \sum_{i=0}^k \|\psi_{j+L+iq_0}\|_{\R/\Z} + \|\psi_{j+L+iq_0}'\|_{\R/\Z}.
\end{align} 
Multiply both sides with $e(Ls/n)$ again we obtain
\begin{align*}
& \Big | (1-z_{q_0,t})^{k-1} e(2Ls/n) e(jt/n)  \Big( d_1(1-z_{q_0,t}) + \sqrt{-1} d_2 \frac{-k q_0}{n} + \sqrt{-1} d_2 (1-z_{q_0,t})\frac{j}{n} \Big) \nonumber \\
&-(1-z_{q_0,t})^{k-1} e(Ls/n)e((j+L)t/n)  \Big( d_1(1-z_{q_0,t}) + \sqrt{-1} d_2 \frac{-k q_0}{n} + \sqrt{-1} d_2 (1-z_{q_0,t}) \frac{j+L}{n} \Big) \nonumber \\
&-(1-z_{q_0,s})^{k-1} e((j+2L)s/n) \sqrt{-1} d_4 (1-z_{q_0,s})\frac{L}{n} \Big | \nonumber \\
&\le 2^k \sum_{i=0}^k \|\psi_{j+iq_0}\|_{\R/\Z} + \|\psi_{j+iq_0}'\|_{\R/\Z}+ 2^k \sum_{i=0}^k \|\psi_{j+L+iq_0}\|_{\R/\Z} + \|\psi_{j+L+iq_0}'\|_{\R/\Z}.
\end{align*} 
Applying the triangle inequality once more, it follows from this inequality and from \eqref{eqn:jst'} applied for $j$ being replaced by $j+L$ we can eliminate $(1-z_s)^{k-1} e((j+2L)s/n)$ and obtain
\begin{align*}
&  \Big | (1-z_{q_0,t})^{k-1} e(2Ls/n) e(jt/n)  \Big( d_1(1-z_{q_0,t}) + \sqrt{-1} d_2 \frac{-k q_0}{n} + \sqrt{-1} d_2 (1-z_{q_0,t}) \frac{j}{n} \Big) \\
&-(1-z_{q_0,t})^{k-1} e(Ls/n) e((j+L)t/n)  \Big( d_1(1-z_{q_0,t}) + \sqrt{-1} d_2 \frac{-k q_0}{n} + \sqrt{-1} d_2 (1-z_{q_0,t}) \frac{j+L}{n} \Big)\\
&-(1-z_{q_0,t})^{k-1} e(Ls/n) e((j+L)t/n)  \Big( d_1(1-z_{q_0,t}) + \sqrt{-1} d_2 \frac{-k q_0}{n} + \sqrt{-1} d_2 (1-z_{q_0,t}) \frac{j+L}{n} \Big)\\
&+(1-z_{q_0,t})^{k-1} e((j+2L)t/n)  \Big( d_1(1-z_{q_0,t}) + \sqrt{-1} d_2 \frac{-k q_0}{n} + \sqrt{-1} d_2 (1-z_{q_0,t}) \frac{j+2L}{n} \Big)  \Big |\\
&\le 2^k \sum_{i=0}^k \|\psi_{j+iq_0}\|_{\R/\Z} + \|\psi_{j+iq_0}'\|_{\R/\Z}+ 2^k \sum_{i=0}^k \|\psi_{j+L+iq_0}\|_{\R/\Z} + \|\psi_{j+L+iq_0}'\|_{\R/\Z}\\
& + 2^k \sum_{i=0}^k \|\psi_{j+L+iq_0}\|_{\R/\Z} + \|\psi_{j+L+iq_0}'\|_{\R/\Z}+ 2^k \sum_{i=0}^k \|\psi_{j+2L+iq_0}\|_{\R/\Z} + \|\psi_{j+2L+iq_0}'\|_{\R/\Z} .
\end{align*} 

After pulling out the common factor $(1-z_{q_0,t})^{k-1}e(jt/n)$ and simplifying we then have
\begin{align}\label{eqn:j_1}
|(1-z_{q_0,t})^{k-1}| &\times \Big| \big( d_1(1-z_{q_0,t}) + \sqrt{-1} d_2 \frac{-k q_0}{n}+  \sqrt{-1} d_2 (1-z_{q_0,t}) \frac{j}{n}\big)(e(Ls/n)-e(Lt/n))^2 \nonumber \\
& +   \sqrt{-1} d_2 (1-z_{q_0,t}) \frac{2L}{n} (e(2Lt/n) - e(Ls/n) e(Lt/n))\Big| \nonumber \\
&\le 2^k \sum_{i=0}^k \|\psi_{j+iq_0}\|_{\R/\Z} + 2\|\psi_{j+L+iq_0}\|_{\R/\Z}+  \|\psi_{j+2L+iq_0}\|_{\R/\Z} \nonumber \\
& + 2^k \sum_{i=0}^k  \|\psi_{j+iq_0}'\|_{\R/\Z} + 2\|\psi_{j+L+iq_0}'\|_{\R/\Z}+  \|\psi_{j+2L+iq_0}'\|_{\R/\Z}.
\end{align} 
Now if we apply \eqref{eqn:j_1} for $j=j_1$ and $j=j_1+j$, and then use the triangle inequality once more we have
\begin{align}\label{eqn:z_{q_0}}
|(1-z_{q_0,t})^{k-1}| &\times (d_2(1-z_{q_0,t})\frac{j}{n}) |(e(L(s-t)/n)-1)|^2 \nonumber \\
&\le 2^k \sum_{i=0}^k \|\psi_{j_1+iq_0}\|_{\R/\Z} + 2\|\psi_{j_1+L+iq_0}\|_{\R/\Z}+  \|\psi_{j_1+2L+iq_0}\|_{\R/\Z} \nonumber  \\
& + 2^k \sum_{i=0}^k  \|\psi_{j_1+iq_0}'\|_{\R/\Z} + 2\|\psi_{j_1+L+iq_0}'\|_{\R/\Z}+  \|\psi_{j_1+2L+iq_0}'\|_{\R/\Z} \nonumber  \\
& + 2^k \sum_{i=0}^k \|\psi_{j_1+j+iq_0}\|_{\R/\Z} + 2\|\psi_{j_1+j+L+iq_0}\|_{\R/\Z}+  \|\psi_{j_1+j+2L+iq_0}\|_{\R/\Z} \nonumber  \\
& + 2^k \sum_{i=0}^k  \|\psi_{j_1+j+iq_0}'\|_{\R/\Z} + 2\|\psi_{j_1+j+L+iq_0}'\|_{\R/\Z}+  \|\psi_{j_1+j+2L+iq_0}'\|_{\R/\Z}.
\end{align} 
The bound for $|(1-z_{q_0,s})^{k-1}| \times (d_4(1-z_{q_0,s})\frac{j}{n}) |(e(L(s-t)/n)-1)|^2$ is identical, and we omit.

\end{proof}
Using the obtained inequalities, we can conclude the section as follows.

\begin{proof}[Proof of Proposition \ref{prop:largepsi}] Recall the notation $J=[a,b]$. We will work with $q_0$ so that $kq_0\le |J|/2$. We divide the treatment into two cases depending on the parameters $d_1,d_2,d_3,d_4$.

{\bf Subcase 1.} Assume that either $|d_2|\ge n^{-\tau_\ast}$ or $|d_4|\ge n^{-\tau_\ast}$. Without loss of generality we assume the first case. Notice that with $j=n^{1-8\tau_\ast}$, then $|j d_2|/n \ge n^{-9\tau_\ast}$, while as by Condition  \ref{cond:s,t}, $\|(s-t)/n\|_{\R/\Z} \ge n^{-1+8\tau}$, so there exists $1\le L \le |J|/2 = n^{1-6\tau_\ast}/2$ such that $|e(L(s-t)/n)-1| \ge n^{-\tau_\ast}$. Let us fix such an $L$. 

We next observe that the RHS of \eqref{eqn:z_{q_0}}, for some $j_1$ from the interval $a\le j_1 \le b-2L-kq_0$ (there are at least $|J|/3$ such $j_1$), is at most $O(n^{-1/2+7\tau_\ast})$. Indeed, this is because the sum of the RHS of \eqref{eqn:z_{q_0}} for $j_1$ from the above range is bounded by $O(\sum_{j\in J} \|\psi_j\|_{\R/\Z} + \|\psi_j'\|_{\R/\Z}) = O(\sqrt{n T}) = O(n^{1/2+\tau_\ast})$ (as each term $\|\psi_i\|_{\R/\Z}$ appears a bounded number of times in the total sum). 

Putting together, with such choice of $L$ and $j_1$, the equation \eqref{eqn:z_{q_0}} implies that
$$|(1-z_{q_0,t})^{k}| \times n^{-9\tau_\ast} \le O(n^{-1/2+7\tau_\ast}).$$
Thus we have that $|(1-z_{q_0,t})^{k}|  \le n^{-1/2+16\tau_\ast}$ in this case, and so with sufficiently small $\tau_\ast$
\begin{equation}\label{eqn:q_0}
|1-z_{q_0,t}|  \le n^{-1/3k}.
\end{equation}
Recall that $z_{q_0,t} = e(q_0 t/n)$. Because \eqref{eqn:q_0} holds for any $q_0=lp_0$ as long as $kq_0 \le |J|/2$, we thus have for all $1\le l \le n^{1-8\tau_\ast}$,
\begin{equation}\label{eqn:q_0'}
|1-e(lp_0 t/n)|  \le n^{-1/3k}.
\end{equation}
As a consequence, by Claim \ref{claim:dividing} we then have
$$\|p_0t/\pi n\|_{\R/\Z}=O(n^{-1+8\tau_\ast}),$$
which contradicts Condition \ref{cond:s,t}.

{\bf Subcase 2.} Now we consider the remaining case that $|d_2|,|d_4| \le n^{-\tau_\ast}$. Without loss of generality we assume $|d_1| \asymp \|\BD\|_2 \ge n^{1-4\tau}$. In the case that $d_1|z_{q_0}-1| \le n^{4 \tau_\ast}$ then we have 
$$|z_{q_0}-1| \le n^{-1+4\tau-4\tau_\ast}.$$ 
In the other case that $d_1|z_{q_0} -1| \ge n^{4\tau_\ast}$, then with $L$ so that $|e(L(s-t)/n)-1|\ge n^{-\tau_\ast}$ as in the previous subcase, the factor $d_1|z_{q_0} -1| ((e(L(s-t)/n)-1))$ is at least $n^{3 \tau_\ast}$, which clearly dominates all other terms involving $d_2$ on the LHS of \eqref{eqn:j_1} (because each of which has order $O(1)$ as $d_2$ is small). So by \eqref{eqn:j_1} we have 
\begin{align*}
|(1-z_{q_0})^{k}| d_1/2 & \le 2^k \sum_{i=0}^k \|\psi_{j+iq_0}\|_{\R/\Z} + 2\|\psi_{j+L+iq_0}\|_{\R/\Z}+  \|\psi_{j+2L+iq_0}\|_{\R/\Z} \nonumber \\
& + 2^k \sum_{i=0}^k  \|\psi_{j+iq_0}'\|_{\R/\Z} + 2\|\psi_{j+L+iq_0}'\|_{\R/\Z}+  \|\psi_{j+2L+iq_0}'\|_{\R/\Z}\\
&=O(1).
\end{align*} 
Hence,
$$|z_{q_0}-1| =O(1/d_1^{1/k}) = O(n^{-1/2k}).$$ 
Thus from both cases we always have $|z_{q_0}-1| = O(n^{-1/2k})$. Varying $q_0$ and using Claim \ref{claim:dividing} we deduce that $\|p_0t/ \pi n\|_{\R/\Z}=O(n^{-1+8\tau_\ast})$, a contradiction.
\end{proof}

Finally, similarly to Theorem \ref{thm:fourier:2'}, as our approach to prove Proposition \ref{prop:largepsi} starts with \eqref{eqn:psi:small}, by passing to subintervals of $J$ when needed we obtain the following analog of Theorem \ref{thm:fourier}.

\begin{theorem}\label{thm:fourier'} Let $C_\ast$ and $\ell$ be given positive constants, and $s,t$ satisfy Condition \ref{cond:s,t} for some sufficiently small constant $\tau$. Then following holds for sufficiently large $n$ and sufficiently small $\tau_\ast$ (depending on $C_\ast, \tau$ and $\ell$). For any $n^{5\tau -1/2} \le \|x\|_2 \le n^{C_\ast}$, for any set $I\subset [n]$ of at most $\ell$ entries we have 
$$|\prod_{i \notin I} \phi_{i}(x)| \le \exp(-n^{\tau_\ast}).$$
\end{theorem}

\appendix

\end{document}